\documentclass{article}
\usepackage{amsmath}
\usepackage{amsfonts}
\usepackage{amsthm}
\usepackage{enumerate}

\newtheorem{thm}{Theorem}[section]
\newtheorem{prop}[thm]{Proposition}
\newtheorem{cor}[thm]{Corollary}

\newtheorem{lem}[thm]{Lemma}

\newtheorem*{thmnn}{Theorem}

\theoremstyle{remark}

\begin{document}
\title{On Spaces Associated with Invariant Divisors on Galois Covers of Riemann Surfaces and Their Applications}

\author{Yaacov Kopeliovich, Shaul Zemel}

\maketitle

\begin{abstract}
Let $f:X \to S$ be a Galois cover of Riemann surfaces, with Galois group $G$. In this paper we analyze the $G$-invariant divisors on $X$, and their associated spaces of meromorphic functions, differentials, and $q$-differentials. We generalize the trace formula for non-trivial elements of $G$ on $q$-differentials, as well as the Chevalley--Weil Formula. When $G$ is Abelian or when the genus of $S$ is 0 we prove additional results, and we also determine the non-special $G$-invariant divisors when both conditions are satisfied.
\end{abstract}

\section*{Introduction}

\smallskip

This paper grew out of an attempt to generalize Thomae's formula to general Abelian covers of the Riemann sphere $\mathbb{P}^{1}(\mathbb{C})$. Such formulae were obtained for fully ramified cyclic covers using the Szeg\H{o} kernel function in \cite{[K1]}, and in \cite{[FZ]} and \cite{[Z]} using more elementary methods. These formulae are relations between theta constants on an appropriate cover of $\mathbb{P}^{1}(\mathbb{C})$ and algebraic parameters defining the cover (i.e., the branching values). Recently \cite{[K2]} managed to extend the Szeg\H{o} kernel construction to non-cyclic 2-covers of $\mathbb{P}^{1}(\mathbb{C})$.

Given an Abelian cover $f:X\to\mathbb{P}^{1}(\mathbb{C})$, the first stage in stating Thomae formulae is finding the non-special $G$-invariant divisors of degree $g_{X}$ on $X$. We achieve this goal fully in this paper, using the methods from \cite{[FZ]} and \cite{[Z]}. On the other hand, \cite{[K1]} and \cite{[K2]} use non-integral $G$-invariant divisors of degree $g_{X}-1$ with a similar non-specialty criterion (these divisors appear implicitly in \cite{[FZ]} and \cite{[Z]} as well). Using the same methods, we determine all the divisors of this sort on general Abelian covers of $\mathbb{P}^{1}(\mathbb{C})$ as well. In the sequel \cite{[KZ]} to the current paper we could show how to apply the Szeg\H{o} kernel techniques, combined with these results about non-special divisors, for indeed stating and proving Thomae formulae for any Abelian cover of $\mathbb{P}^{1}(\mathbb{C})$.

\smallskip

Most of the results of this paper concern general Galois covers of Riemann surfaces. When $f:X \to S$ is a Galois cover with Galois group $G$, we first define a refinement of the classical signature of the cover, associating to each branching image $\eta \in S$ a conjugacy class in $G$ that arises from the monodrony of $f$ around $\eta$. Our method of investigating $G$-invariant divisors on $X$ is then based on a certain normalization, for which one has to choose an arbitrary point in $S$ that is not a branching image. The normalization is with respect to an equivalence relation between $G$-invariant divisors that is finer than linear equivalence, a relation that we call invariant linear equivalence. Two divisors are called \emph{invariantly linearly equivalent} if their difference is the divisor of a \emph{$G$-invariant} meromorphic function on $X$. Using this normalization we are also able to generalize some classical results about the action of $G$ on spaces of $q$-differentials. Explicitly, we extend the validity of two classical results. One is the Chevalley--Weil Formula from \cite{[CW]} and \cite{[W]} for the multiplicity of an irreducible representation of $G$ on the space of $q$-differentials. The other one is the formula, called the Eichler Trace Formula in Chapter 5 of \cite{[FK]}, for the trace of a non-trivial automorphism of finite order of $X$ on that space. Our generalization involves the spaces of $q$-differentials that are associated with the pullback of a positive divisor on $S$. We mention that \cite{[JK]} investigates similar questions as well, but only for functions (i.e., with $q=0$), and under the assumption that the divisor is non-special and that all the irreducible complex representations of $G$ can be defined over $\mathbb{Q}$. The earlier paper \cite{[EL]} has also established results of this type, but only for the Euler characteristic of the associated line bundles (or vector bundles), and under some reductivity assumptions. Certain special cases are also treated in the references cited in \cite{[JK]}, as well as in \cite{[R]} mentioned below and the references therein, while the case of function spaces associated with pullbacks of divisors of large degrees on $S$ is dealt with in \cite{[VL]}.

The way we prove these results is as follows. We deduce the generalized Chevalley--Weil Formula for characters (i.e., 1-dimensional representations) directly from our normalization method. As any automorphism of $X$ lies in the cyclic group that it generates, this suffices for establishing the generalized Eichler Trace Formula. The complete generalized Chevalley--Weil Formula then follows from basic representation theory. Note that our proof  is straightforward, using only classical results like the Riemann--Roch Theorem and character theory, without the infinite covers of \cite{[CW]} or any other complicated objects.

Finally, we recall that \cite{[LR]} establishes a decomposition of $J(X)$ according to the rational irreducible representations of $G$, while \cite{[R]} evaluates the dimensions of the resulting components. We show how the results of \cite{[R]} (as well as some of those of the previous reference) follow from the (classical) Chevalley--Weil Formula, and relate certain parts of this decomposition to quotients of $X$ that are cyclic over $S$. This is of interest, since these parts can be easily described using equations over $\mathbb{C}(S)$. Our description is based on a refinement of the Prym varieties from \cite{[LR]} (which generalize vastly the classical notion of Prym varieties), which we define for cyclic covers and call primitive Prym varieties.

\smallskip

In order to present the results of this paper in more detail we shall need some notation. Given a non-trivial conjugacy class $C$ in $G$, there is a finite number of points $\eta \in S$ such that the local monodromy action induced by $f:X \to S$ belongs to the conjugacy class $C$. We denote the number of such points by $r_{C}$, and the order in $G$ of elements in $C$ by $o(C)$. For any character $\chi$ from $G$ to $S^{1}=\big\{z\in\mathbb{C}\big||z|=1\big\}$ and any conjugacy class $C$ we set $u_{\chi,C}$ to be the minimal non-negative exponent $u$ such that $\chi(\sigma)=e^{2\pi iu/o(C)}$ for $\sigma \in C$, and let $t_{\chi}=\sum_{C \neq Id_{X}}\frac{r_{C}u_{\chi,C}}{o(C)}$. Fixing a base point $\nu \in S$, a normalized $G$-invariant divisor on the Riemann surface $X$ can then be written in terms of the pullback of a non-special positive divisor on $S$, an arbitrary multiple of $f^{*}\nu$, and a partition of the $r_{C}$ branching images associated with $C$ into $o(C)$ sets $B_{C,i}$, $0 \leq i<o(C)$. The first result of this paper (see Theorems \ref{intnonsp} and \ref{eqnointg-1}) determines the divisors that are required for the Thomae formulae from \cite{[KZ]}, in which the Galois group $G$ is Abelian and $S=\mathbb{P}^{1}(\mathbb{C})$. In this case the conjugacy classes are just elements of $G$, there are no non-trivial positive divisors that are non-special, the base point $\nu$ is $\infty$, and the divisors are as follows.
\begin{thmnn}
Let $\Delta$ be a normalized $G$-invariant positive divisor on $X$. Then $\Delta$ is non-special of degree $g_{X}$ if and only if the multiple of $f^{*}\nu$ is 0 and the sets $B_{\sigma,i}$ with non-trivial $\sigma \in G$ and $0 \leq i<o(\sigma)$ satisfy the condition $\sum_{\sigma \neq Id_{X}}\sum_{i=0}^{u_{\chi,\sigma}-1}|B_{\sigma,i}|=t_{\chi}-1$ for every non-trivial character $\chi$ of $G$. On the other hand, a divisor $\Delta$ on $X$ is normalized, $G$-invariant, of degree $g_{X}-1$, and is not linearly equivalent to any positive divisor, if and only if $f^{*}\nu$ appears with the multiplicity $-1$ and the equality $\sum_{\sigma \neq Id_{X}}\sum_{i=0}^{u_{\chi,\sigma}-1}|B_{\sigma,i}|=t_{\chi}$ holds for every character $\chi$ of $G$.
\end{thmnn}

We now describe our generalization of the Chevalley--Weil Formula. Recall that for every conjugacy class $C$ and every irreducible representation $\rho$ of $G$, of dimension $d_{\rho}$, the representation space of $\rho$ decomposes into $o(C)$ possible eigenspaces corresponding to the eigenvalues of $\rho(\sigma)$ for $\sigma \in C$. For $\alpha\in\mathbb{Z}/o(C)\mathbb{Z}$ we denote by $N_{C,\alpha}^{\rho}$ the dimension of the $e^{2\pi i\alpha/o(C)}$-eigenspace of $\rho(\sigma)$. Given a divisor $\Gamma$ on $S$ we write $\deg\Gamma$ for the degree of $\Gamma$ and $f^{*}\Gamma$ for the pullback of $\Gamma$ to $X$, and for $q\in\mathbb{Z}$ we denote by $\Omega^{q}(-f^{*}\Gamma)$ the space of $q$-differentials on $X$ whose divisors are bounded from below by $-f^{*}\Gamma$. We also introduce the notation $\{x\}$ for the fractional part of the real number $x$. Our formula, which is proved in Theorem \ref{CWgen}, is the following.
\begin{thmnn}
Take $q\in\mathbb{Z}$ with $(g_{X}-1)(q-1)\geq0$, a positive divisor $\Gamma$ on $S$, and an irreducible representation $\rho$ of $G$. Then there exists $\delta\in\{0,1\}$, which may not vanish only if the conditions that $\Gamma$ is trivial, that $(g_{X}-1)(q-1)=0$, and that $d_{\rho}=1$ are all satisfied, such that the multiplicity of $\rho$ in $\Omega^{q}(-f^{*}\Gamma)$ is \[d_{\rho}[(2q-1)(g_{S}-1)+\deg\Gamma]+\sum_{C \neq Id_{X}}\!r_{C}\sum_{\alpha=0}^{o(C)-1}N_{C,\alpha}^{\rho}\Big[(q-1)\big(1-\tfrac{1}{o(C)}\big)+\big\{\tfrac{q-1-\alpha}{o(C)}\big\}\Big]+\delta.\]
\end{thmnn}

Our investigation of the action of $G$ on the Jacobian $J(X)$ of $X$ also leads to the following interpretation of its decomposition. For every cyclic quotient $Q$ of $G$ there is an associated cyclic cover $Y_{Q}$ of $S$ that is a quotient of $X$, and in Section \ref{JacDecom} we define for each such cover its primitive Prym variety $\widetilde{P}(Y_{Q}/S)$ as the complement in $J(Y_{Q})$ of the images of the Jacobians of all the intermediate Riemann surfaces between $Y_{Q}$ and $S$. Then we get the following result in Theorem \ref{PryminJac}.
\begin{thmnn}
The product $\prod_{Q=G/N\ \mathrm{cyclic}}\widetilde{P}(Y_{Q}/S)$ maps into $J(X)$ with a finite kernel, and in case $G$ is Abelian, this map also surjects. The constituent $\widetilde{P}(Y_{Q}/S)$ associated with $Q$ is non-trivial wherever the genus of $Y_{Q}$ is positive, unless $Y_{Q}$ and $S$ are both of genus 1 and $Q$ is non-trivial.
\end{thmnn}

\smallskip

The paper is divided into 7 sections. Section \ref{GalCov} introduces the finer invariant lying over points of $S$, and presents some explicit formulae for it in the Abelian case. In Section \ref{NormForm} we define the normalized form of $G$-invariant divisors modulo invariant linear equivalence, and prove some properties of the normalized divisors of meromorphic functions on $X$ associated with characters of $G$. Section \ref{LDeltaInv} relates the dimensions of function spaces associated with (normalized) $G$-invariant divisors and characters of $G$ to dimensions of certain spaces of meromorphic function on $S$, and gives explicit expressions for these dimensions when $S$ has genus 0. In Section \ref{AbCovP1C} we give more details about the case of Abelian $G$, and determine the sets of the two types of $G$-invariant non-special divisors on Abelian covers of $\mathbb{P}^{1}(\mathbb{C})$. Section \ref{Diffs} investigates differentials and $q$-differentials on Galois covers, and determines the dimensions of the relevant spaces. For pullbacks of positive divisors on $S$, a deeper analysis of the resulting representations is carried out in Section \ref{RepQDiff}, where the generalizations of the Chevalley--Weil Formula and the Eichler Trace Formula are proved. Finally, Section \ref{JacDecom} presents the decomposition of $J(X)$, as well as the relation between the relevant parts of this decomposition and cyclic covers of $S$ that are quotients of $X$.

\smallskip

We are thankful to H. M. Farkas for many intriguing discussions on the subject of this work, as well as to the anonymous referee for many valuable suggestions, which greatly improved the presentation of this paper.

\section*{List of Notation}

$A_{C,\chi}$, $A_{\sigma,\chi}$ --- The unions $\bigcup_{i=0}^{u_{\chi,C}-1}B_{C,i}$ and $\bigcup_{i=0}^{u_{\chi,\sigma}-1}B_{\sigma,i}$ respectively.


\noindent
$\alpha_{C,\overline{\chi}}^{q}$ --- The integral part $\big\lfloor\frac{q(o(C)-1)-u_{\overline{\chi},C}}{o(C)}\big\rfloor$.

\noindent
$B_{C,i}$ --- The set of branch points with monodromy $C$ whose pre-images appear with multiplicity $o(C)-1-i$ in a divisor $\Delta$.

\noindent
$B_{\sigma,i}$ --- The set $B_{\{\sigma\},i}$ when $G$ is Abelian and $C=\{\sigma\}$.




\noindent
$\beta_{C,\overline{\chi}}^{q}$ --- The residue $o(C)\cdot\big\{\frac{q(o(C)-1)-u_{\overline{\chi},C}}{o(C)}\big\}$.



\noindent
$\mathbb{C}(X)_{\chi}$ --- The space of functions in $\mathbb{C}(X)$ on which $G$ acts via $\chi\in\widehat{G}$.









\noindent
$\delta_{\alpha,\beta}$ --- The Kronecker $\delta$-symbol, which equals 1 when $\alpha=\beta$ and 0 otherwise.










\noindent
$h_{\chi}$ --- A meromorphic function on $X$ on which $G$ acts via $\chi$, whose divisor is normalized.

\noindent
$i^{q}(\Delta)$, $i^{q}(\Gamma)$, $i^{q}_{\chi}(\Delta)$ --- The dimension of the space $\Omega^{q}(\Delta)$, $\Omega^{q}(\Gamma)$, or $\Omega^{q}(-\Delta)_{\chi}$.

\noindent
$i(\Delta)$, $i(\Gamma)$, $i_{\chi}(\Delta)$ --- A simpler notation for $i^{1}(\Delta)$, $i^{1}(\Gamma)$, and $i^{1}_{\chi}(\Delta)$ respectively.






\noindent
$L(-\Delta)$ (resp. $L(-\Gamma)$) --- The space of meromorphic functions on $X$ (resp. $S$) whose divisors are bounded from below by $-\Delta$ (resp. $-\Gamma$).

\noindent
$L(-\Delta)_{\chi}$ --- The subspace of $L(-\Delta)$ on which $G$ acts via $\chi$.



\noindent
$N_{C,\alpha}^{\rho}$ --- The dimension of the $\zeta_{o(C)}^{\alpha}$-eigenspace of $\rho(\sigma)$ for $\sigma \in C$.







\noindent
$\psi(\eta)$ --- The conjugacy class of the monodromy at the point $\eta \in S$.




\noindent
$r_{C}$ (resp. $r_{\sigma}$) --- The number of points in $S$ with monodromy belonging to $C$ (resp. with monodromy $\sigma$).

\noindent
$r(-\Delta)$, $r(-\Gamma)$, $r_{\chi}(-\Delta)$ --- The dimension of the space $L(-\Delta)$, $L(-\Gamma)$, or $L(-\Delta)_{\chi}$.






\noindent
$s_{\overline{\chi},q}$ --- The maximal number such that $i^{q}\Big(\Upsilon_{\overline{\chi}}+s_{\overline{\chi},q}\nu-\sum_{C \neq Id_{X}}\sum_{j=1}^{r_{C}}\alpha_{C,\overline{\chi}}^{q}\eta_{C,j}\Big)$ is positive.



\noindent
$t_{\chi}$ --- The expression $\sum_{C \neq Id_{X},}\frac{r_{C}u_{\chi,C}}{o(C)}$ (the sum is over conjugacy classes).

\noindent
$u_{\chi,C}$ --- The minimal $0 \leq u\in\mathbb{Z}$ such that $\chi(\sigma)$ with $\sigma \in C$ equals $e^{2\pi iu/o(C)}$.


\noindent
$\Omega^{q}(X)$ (resp. $\Omega^{q}(S)$) --- The space of meromorphic $q$-differentials on the compact Riemann surface $X$ (resp. $S$).

\noindent
$\Omega^{q}(\Delta)$ (resp. $\Omega^{q}(\Gamma)$) --- The space of $q$-differentials on $X$ (resp. $S$) whose divisors are bounded from below by $\Delta$ (resp. $\Gamma$).

\noindent
$\Omega^{q}(\Delta)_{\chi}$ --- The subspace of $\Omega^{q}(\Delta)$ on which $G$ acts via $\chi$.


\noindent
$\omega_{\chi,q}$ --- A meromorphic $q$-differential on $X$ on which $G$ acts by $\chi$, whose divisor is normalized.

\noindent
$\omega_{\chi}$ --- A shorthand for $\omega_{\chi,1}$.

\noindent
$\varpi_{\chi,q}$ --- A $q$-differential generating $\Omega^{q}\Big(\Upsilon_{\overline{\chi}}+s_{\overline{\chi},q}\nu-\sum_{C \neq Id_{X}}\sum_{j=1}^{r_{C}}\alpha_{C,\overline{\chi}}^{q}\eta_{C,j}\Big)$.

\noindent
$\varpi_{\chi}$ --- A shorthand for $\varpi_{\chi,1}$.




\noindent
$\lfloor x \rfloor$ --- The integral part of the real number $x$, namely $\max\{n\in\mathbb{Z}|n \leq x\}$.

\noindent
$\{x\}$ --- The fractional part $x-\lfloor x \rfloor$ of the real number $x$.



\noindent
$\Upsilon_{\chi}$ --- The divisor on $S$ appearing in the divisor of $h_{\chi}$.

\noindent
$\widetilde{\Upsilon}_{\chi,q}$ --- The divisor on $S$ appearing in the divisor of $\varpi_{\chi,q}$.

\noindent
$\widetilde{\Upsilon}_{\chi}$ --- A shorthand for $\widetilde{\Upsilon}_{\chi,1}$.






\section{Galois Covers of Compact Riemann Surfaces \label{GalCov}}

Given a non-trivial map $f:X \to S$ between Riemann surfaces and a point $P \in X$, we denote by $b_{P}$ the branching number of $f$ at $P$. As in \cite{[FK]}, \cite{[FZ]}, \cite{[K1]}, \cite{[K2]}, \cite{[Z]}, and others, we say that $P$ is a \emph{branch point} of $f$ if $b_{P}>0$, and we call $f(P) \in S$ a \emph{branching image}. We shall consider in this paper only maps $f:X \to S$ that are \emph{Galois}, i.e., in which the (finite) group $G$ of automorphisms of $X$ that commute with $f$ operates transitively on each fiber of $f$. Then $G$ is the \emph{Galois group} of $f$, and by denoting by $\mathbb{C}(X)$ (resp. $\mathbb{C}(S)$) the field of meromorphic functions on $X$ (resp. $S$), the group $G$ is also the Galois group of the field extension $\mathbb{C}(X)/\mathbb{C}(S)$ via the embedding $f^{*}$. We recall from the correspondence between compact Riemann surfaces and fields of transcendence degree 1 over $\mathbb{C}$ that every finite subgroup $G$ of the group $Aut(X)$ of automorphisms of $X$ is the Galois group of such a map $f:X \to S$, where $S$ is the compact Riemann surface associated with the subfield of $\mathbb{C}(X)$ that is fixed by $G$. We may thus use the settings of ``$f:X \to S$ Galois'' or ``finite subgroup $G$ of $Aut(X)$'' interchangeably. In addition, for every field $\mathbb{F}$ we shall denote its multiplicative group by $\mathbb{F}^{\times}$.

\smallskip

For a point $\eta \in S$, the stabilizers of the pre-images of $\eta$ in $X$ are conjugate cyclic subgroups. Gathering these conjugacy classes of subgroups that are non-trivial yields, with the genus of $S$, the object that is called in \cite{[R]} and others the \emph{signature} of $f:X \to S$. However, $f$ yields a finer invariant at $\eta$, the existence of which already appears in, e.g., Proposition 2.6 of \cite{[V]} (at least in case $S$ is the \emph{Riemann sphere} $\mathbb{P}^{1}(\mathbb{C})=\mathbb{C}\cup\{\infty\}$), using the local rings of points in $S$ and in $X$. We now present this invariant from a more geometric point of view, which we include here also for setting up notation for the rest of the paper.
\begin{prop}
Lifting small simple positively oriented closed paths around elements of $S$ defines a map $\psi$ from $S$ to conjugacy classes of elements of $G$. Moreover, the order of any element in the class $\psi(\eta)$ for $\eta \in S$ is $b_{P}+1$ for any point $P \in X$ with $f(P)=\eta$. \label{StoGbr}
\end{prop}

\begin{proof}
Fix some $\eta \in S$, and let $\gamma:[0,1] \to S$ be a simple smooth positively oriented closed path that contains $\eta$ in its interior. We assume that no branching image in $S$ that is distinct from $\eta$ is in the closure of the interior of $\gamma$. Consider a continuous map $\tilde{\gamma}:[0,1] \to X$ with $f\circ\tilde{\gamma}=\gamma$ (i.e., a \emph{lift} of $\gamma$ to $X$). Such a lift is determined by the starting point $\tilde{\gamma}(0)$. As $\tilde{\gamma}(1)$ also lies over $\gamma(1)=\gamma(0)$ and $f$ is Galois, there exists a unique $\sigma \in G$ such that $\tilde{\gamma}(1)=\sigma\big(\tilde{\gamma}(0)\big)$ (the uniqueness follows from the assumption that $\tilde{\gamma}(0)$ cannot be a branch point). It is clear that $\sigma$ does not change by continuous deformations of $\gamma$, with its chosen lift $\tilde{\gamma}$, as long as the deformation does not go over a branching image. However, we may change the choice of the starting point $\tilde{\gamma}(0)$ to another pre-image of $\gamma(0)$, which is $\tau\big(\tilde{\gamma}(0)\big)$ for a unique $\tau \in G$. As doing so would replace $\sigma$ by its conjugate $\tau\sigma\tau^{-1}$, the conjugacy class of $\sigma$ is a well-defined object depending only on $\eta$.

It remains to determine the order of $\sigma$. Take $P \in X$ with $f(P)=\eta$, choose coordinates around $P$ and $\eta$ in which $f$ becomes $t \mapsto t^{b_{P}+1}$, and take $\gamma$ to be a small circle around $\eta$ in this coordinate. Then a lift $\tilde{\gamma}$ becomes a $(b_{P}+1)$st root of the circle $\gamma$, which is $\frac{1}{b_{P}+1}$ of a circle. Applying this operation again, but starting with $\tilde{\gamma}(1)=\sigma\big(\tilde{\gamma}(0)\big)$, we get again the element $\sigma\sigma\sigma^{-1}=\sigma$ of $G$, and the new end-point is the image of $\tilde{\gamma}(0)$ under $\sigma^{2}$, after covering $\frac{2}{b_{P}+1}$ of a circle. It follows that for any integer $k$, the $k$th power of $\sigma$ is trivial in $G$ if and only if $k$ iterations of our operation complete an integral multiple of a circle (so that the end point of the combined lifts coincides with the starting point). The order of $\sigma$ (hence of any other element in its conjugacy class) is therefore $b_{P}+1$. This completes the proof of the proposition.
\end{proof}
The signature from \cite{[R]} includes, apart from the genus of $S$, precisely the conjugacy classes of the cyclic subgroups of $G$ generated by representatives for the conjugacy classes $\psi(\eta)$ for those $\eta \in S$ for which $\psi(\eta)$ is non-trivial (i.e., for branching images $\eta$).

The symbol $C$ will henceforth always stand for a conjugacy class in the group $G$ (also as a summation index). We shall denote by $o(C)$ the order $o(\sigma)$ of one (hence any) element $\sigma \in C$. We shall allow ourselves the abuse of notation for writing just $Id_{X}$ for the trivial conjugacy class $\{Id_{X}\}$ in $G$ (the trivial element is $Id_{X}$ since this is the identity element of the group $Aut(X)$ by definition). In addition, when the cover $f:X \to S$ is given and $C$ is non-trivial, we denote by $r_{C}$ the (finite) number of points $\eta \in S$ with $\psi(\eta)=C$. Proposition \ref{StoGbr} now combines with the Riemann--Hurwitz formula to produce the following result.
\begin{cor}
When $f:X \to S$ is Galois with Galois group $G$, the genus $g_{X}$ is $1+n(g_{S}-1)+\sum_{C \neq Id_{X}}\frac{nr_{C}}{2o(C)}\big(o(C)-1\big)$. \label{Galgen}
\end{cor}
For the proof, recall from Proposition \ref{StoGbr} that over a point $\eta \in S$ with $\psi(\eta)=C$ there exist $\frac{n}{o(C)}$ pre-images $P \in X$ with $b_{P}=o(C)-1$. In the notation and conventions of \cite{[R]}, Corollary \ref{Galgen} is just Equation (2.1) of that reference.

\smallskip

A particular case in which several results of this paper become simpler and more complete is the case of \emph{Abelian} covers, especially where $S=\mathbb{P}^{1}(\mathbb{C})$. In this case we shall replace every conjugacy class $C$ by the single element $\sigma \in G$ that it contains in all our notation. Recall that presenting a Galois group as a direct product corresponds to considering the cover $f:X \to S$ as a fibered product of Riemann surfaces over $S$. In particular, it follows from the Structure Theorem for Finite Abelian Groups that an Abelian cover $f:X \to S$, with the Galois group $G$ presented as the product of the cyclic groups $H_{l}$, $1 \leq l \leq q$, is the fibered product of $q$ cyclic covers of $S$. Such a cyclic cover, of some degree $m$, can be presented (e.g., by Hilbert's Theorem 90) as the normalization of the subset of $S\times\mathbb{P}^{1}(\mathbb{C})$ cut by the equation $w^{m}=F$ for $w\in\mathbb{P}^{1}(\mathbb{C})$ and $F\in\mathbb{C}(S)^{\times}$. This cover is irreducible if and only if $F$ is not a $d$th power in $\mathbb{C}(S)$ for any divisor $d>1$ of $m$. Equivalently, for any power $e$, the function $F^{e/\gcd\{m,e\}}$ is not a $\frac{m}{\gcd\{m,e\}}$th power in $\mathbb{C}(S)$ unless $m$ divides $e$ (for otherwise $w^{e}$ would be in $f^{*}\mathbb{C}(S)$, a situation that cannot occur if $w$ generates a cyclic cover of $S$ of degree $m$ and $m$ does not divide $e$). Therefore a general Abelian cover $X$ of $S$, with $G$ as above, can be presented as the normalization of the algebraic set
\begin{equation}
\big\{(\eta,w_{1},\ldots,w_{q}) \in S\times\mathbb{P}^{1}(\mathbb{C})^{q}\big|w_{l}^{m_{l}}=F_{l}(\eta),\ 1 \leq l \leq q\big\}, \label{fibprod}
\end{equation}
where $|H_{l}|=m_{l}$ and with $F_{l}\in\mathbb{C}(S)^{\times}$ for every $1 \leq l \leq q$. The irreducibility of the fibered product $X$ is characterized by the following condition. For a $q$-tuple $(e_{l})_{l=1}^{q}$ of integers, set $\beta=\mathrm{lcm}\big\{\frac{m_{l}}{\gcd\{m_{l},e_{l}\}}\big|1 \leq l \leq q\big\}$. Then if $\beta>1$ (which is equivalent to some $e_{l}$ not being divisible by the corresponding $m_{l}$), the function $\prod_{l=1}^{q}F_{l}^{e_{l}\beta/m_{l}}$ is not a $\beta$th power in $\mathbb{C}(S)$ (this condition clearly generalizes the one from the cyclic case, since in this case $\beta=\frac{m_{l}}{\gcd\{m_{l},e\}}$ and $\frac{e_{l}\beta}{m_{l}}=\frac{e_{l}}{\gcd\{m_{l},e\}}$).

\smallskip

We would like to see how the map $\psi$ from Proposition \ref{StoGbr} looks like when $G$ is presented as the product of the cyclic groups $H_{l}$. For any natural number $N$ we denote the primitive $N$th root of unity $e^{2\pi i/N}$ by $\zeta_{N}$. In our presentation of $X$, the group $H_{l}$ is generated by the automorphism $\tau_{l}$ sending the coordinate $w_{l}$ from Equation \eqref{fibprod} to $\zeta_{m_{l}}w_{l}$ and leaving the other coordinates invariant. A general element $\sigma \in G$ has a unique presentation as $\prod_{l=1}^{q}\tau_{l}^{\alpha_{l}}$, where for any $1 \leq l \leq q$ the power $\alpha_{l}$ is taken from $\mathbb{Z}/m_{l}\mathbb{Z}$ and is determined by the equality $\sigma(w_{l})=\zeta_{m_{l}}^{\alpha_{l}}w_{l}$. Note that $\tau_{l}$ depends not only on the structure of $X$ as a fibered product of cyclic covers of $S$, but also on the choice of the generator $w_{l}$. The map $\psi$ from Proposition \ref{StoGbr}, with values in the Abelian group $G$ itself, now takes the following explicit form.
\begin{prop}
For $\eta \in S$ and an index $1 \leq l \leq q$, consider the order $\alpha_{l}=\mathrm{ord}_{\eta}F_{l}$ of the function $F_{l}\in\mathbb{C}(S)^{\times}$ at $\eta$. Then $\psi(\eta)$ equals $\prod_{l=1}^{q}\tau_{l}^{\alpha_{l}}$. \label{psiZn}
\end{prop}

\begin{proof}
Take a local coordinate $u$ for $S$ around $\eta$ with $u(\eta)=0$, and let $P \in X$ be a pre-image of $\eta$ in $X$. The equation for $w_{l}$ around $P$ becomes $w_{l}^{m_{l}}=u^{\alpha_{l}}\phi(u)$ with $\phi$ a holomorphic function on a neighborhood of 0 with $\phi(0)\neq0$. Take the path $\gamma:[0,1] \to S$ defined by $\gamma(t)=u^{-1}(\varepsilon e^{2\pi it})$ for small enough $\varepsilon>0$, and consider a lift $\tilde{\gamma}$ of $\gamma$ to $X$ with $\tilde{\gamma}(0)$ lying near $P$. Substituting shows that the composition $w_{l}\circ\tilde{\gamma}$ takes $t$ to $e^{2\pi i\alpha_{l}t/m_{l}}\mu(\varepsilon e^{2\pi it})$, where $\mu$ is a holomorphic function in the neighborhood of 0 that satisfies $\mu(u)^{m_{l}}=\varepsilon^{\alpha_{l}}\phi(u)$. Indeed, we can define $\mu(u)$ as $\varepsilon^{\alpha_{l}/m_{l}}e^{\log\phi(u)/m_{l}}$ for some branch of $\log\phi(u)$, which can be defined holomorphically since $\phi(0)\neq0$, and we choose the branch according the values of $w_{l}$ in the neighborhood of $P$. Since $\mu(\varepsilon e^{2\pi it})$ attains the same value $\mu(\varepsilon)$ for both $t=0$ and $t=1$, we find that the value of $w_{l}$ at $\tilde{\gamma}(1)$ is $\zeta_{m_{l}}^{\alpha_{l}}$ times its value on $\tilde{\gamma}(0)$. Applying this argument for all $1 \leq l \leq q$ identifies $\psi(\eta)$ with the required element of $G$. This proves the proposition.
\end{proof}

\smallskip

When $S$ is the Riemann sphere $\mathbb{P}^{1}(\mathbb{C})$, with a natural coordinate $z$, the map $f$ yields the canonical meromorphic function $f^{*}z$ on $X$. The cyclic covers involved are $Z_{m_{l}}$ curves, and thus one can study meromorphic functions on $X$ in terms of the theory of such curves treated in, e.g., \cite{[FZ]} and \cite{[Z]}. This also puts some of the results of this paper in the context required for the sequel \cite{[KZ]}.

A remark about the comparison with \cite{[Z]} is in order here. Recall that in this reference $w_{l}$ was normalized such that the function $F_{l}$ appearing in the $Z_{m_{l}}$ equation is a monic polynomial in $f^{*}z$ having no roots of order $m_{l}$ or more. We shall not require this property here (though see the remark after Proposition \ref{charZn}). On the other hand, we shall assume that the order of $F_{l}(z)$ at $\infty$ is divisible by $m_{l}$ for every $l$, to avoid branching over $\infty$ (see Propositions \ref{psiZn} and \ref{StoGbr}). This condition is easily obtained by composing $f$ with an automorphism of $\mathbb{P}^{1}(\mathbb{C})$ if necessary. Back to the general setting, we shall later require a point $\nu \in S$ on which we shall concentrate all the non-normalized parts of divisors, and it will be more convenient to assume that $\nu$ is not a branching image. The natural choice in case $S=\mathbb{P}^{1}(\mathbb{C})$ would be $\nu=\infty$, and this is the reason why we are looking for presentations with no branching over that point.

\section{Normalization of Invariant Divisors \label{NormForm}}

Let $S^{1}$ be the \emph{circle group} $\big\{z\in\mathbb{C}\big||z|=1\big\}$, and let $G$ be a finite group. We denote the \emph{dual group} $Hom(G,S^{1})$, elements of which are known as \emph{characters} of $G$, by $\widehat{G}$, and its trivial element by $\mathbf{1}$. The dual group $\widehat{G}$ coincides with that of its Abelianization $G^{ab}=G/[G,G]$, and is therefore isomorphic to $G^{ab}$, though not canonically. Since any element $\chi\in\widehat{G}$ is constant on conjugacy classes in $G$, the value $\chi(C)$ is well-defined for any conjugacy class $C$ in $G$.

Assume now that $G \subseteq Aut(X)$ is finite, and take $\chi\in\widehat{G}$. The set of functions $h\in\mathbb{C}(X)$ satisfying $h\circ\sigma=\chi(\sigma) \cdot h$ for every $\sigma \in G$ is a vector space over $\mathbb{C}(S)$, which we denote by $\mathbb{C}(X)_{\chi}$. The direct sum $\bigoplus_{\chi\in\widehat{G}}\mathbb{C}(X)_{\chi}$ is the space of those functions on which the commutator subgroup $[G,G]$ of $G$ acts trivially, so that it equals all of $\mathbb{C}(X)$ if and only if $G$ is Abelian.

\smallskip

We shall denote the trivial element of the group $\mathrm{Div}(X)$ of divisors on $X$ by $0_{X}$, for distinguishing it from both $0\in\mathbb{Z}\subseteq\mathbb{C}$ and the trivial element $0_{S}$ of $\mathrm{Div}(S)$ for another Riemann surface $S$. We recall that two divisors $\Delta$ and $\Xi$ on $X$ are called \emph{linearly equivalent} if $\Xi-\Delta$ is the divisor $\mathrm{div}(h)$ of some function $h\in\mathbb{C}(X)^{\times}$. Having $G \subseteq Aut(X)$ as part of the structure, we call two divisors \emph{invariantly linearly equivalent} if $\Xi-\Delta$ is the divisor of a \emph{$G$-invariant} function from $\mathbb{C}(X)^{\times}$. On the other hand, the action of $G$ on $\mathrm{Div}(X)$ allows one to consider \emph{$G$-invariant} divisors. The next result compares the notions of linear equivalence and invariant linear equivalence for such divisors.
\begin{lem}
Let $\Delta$ be a $G$-invariant divisor on $X$, and let $\Xi$ be a divisor on $X$ that is linearly equivalent to $\Delta$. Then $\Xi$ is $G$-invariant if and only if the difference $\Xi-\Delta$, which is a principal divisor, is the divisor of a non-zero function that lies in $\mathbb{C}(X)_{\chi}$ for some character $\chi\in\widehat{G}$. \label{invlineq}
\end{lem}

\begin{proof}
One direction is easy: Since for any non-zero element $h\in\mathbb{C}(X)_{\chi}$ the function $h\circ\sigma$ is a multiple of $h$, it has the same divisor as $h$. Hence $\mathrm{div}(h)$ is a $G$-invariant divisor, and adding it to the $G$-invariant divisor $\Delta$ yields a $G$-invariant divisor. The other direction reduces to the statement that any meromorphic function $h\in\mathbb{C}(X)^{\times}$ whose divisor is $G$-invariant must be in $\mathbb{C}(X)_{\chi}$ for some character $\chi$ of $G$. And indeed, since principal divisors on compact Riemann surfaces determine functions up to scalar multiplication, the $G$-invariance of $\mathrm{div}(h)$ implies that for every $\sigma \in G$, the function $h\circ\sigma$ is a scalar multiple $c_{\sigma}h$ of $h$ for some non-zero scalar $c_{\sigma}\in\mathbb{C}$. It is now clear that the map $\chi$ sending $\sigma \in G$ to $c_{\sigma}$ is multiplicative, and as $G$ is finite, the numbers $c_{\sigma}=\chi(\sigma)$ must be roots of unity and hence contained in $S^{1}$. It follows that $\chi$ is in $\widehat{G}$ and $h\in\mathbb{C}(X)_{\chi}$. This completes the proof of the lemma.
\end{proof}

Since invariant linear equivalence is the equivalence relation arising from divisors of non-zero functions from $\mathbb{C}(X)_{\mathbf{1}}$, it follows from Lemma \ref{invlineq} that for any linear equivalence class, the set of $G$-equivariant divisors in that class decomposes (if it is not empty) as $n^{ab}$ invariant linear equivalence classes, where $n^{ab}$ is the order of $G^{ab}$ hence of $\widehat{G}$. In fact, the set of these classes is a free orbit of $\widehat{G}$ (see the proof of Lemma \ref{actCXrho} below).

\smallskip

The following result will be useful for normalizing certain divisors and functions, but it is also interesting in its own right.
\begin{lem}
Take $P \in X$, $\chi\in\widehat{G}$, and $0 \neq h\in\mathbb{C}(X)_{\chi}$, and set $\eta=f(P) \in S$ and $C=\psi(\eta) \subseteq G$, of order $o(C)$. Then $\chi(C)$ equals $\zeta_{o(C)}^{\mathrm{ord}_{P}h}$. \label{charandord}
\end{lem}

\begin{proof}
Consider a small positively oriented closed path $\gamma:[0,1] \to S$ as in Proposition \ref{StoGbr}, and assume that neither a zero nor a pole of $h$ in $X$ maps to the closure of the interior of $\gamma$, except perhaps pre-images of $\eta$. The proof of Proposition \ref{StoGbr} shows that a lift $\tilde{\gamma}$ of $\gamma$ covers $\frac{1}{o(C)}$ of a closed path around $P$, and the concatenation of the images of this lift under the $o(C)$ different powers of the element $\sigma \in C$ arising from the choice of $P \in f^{-1}(\eta)$ yields a closed path around $P$ in $X$. As the only zero or pole of $h$ that is possibly contained in the interior of this path is $P$, we find that $\mathrm{ord}_{P}h$ can be evaluated as $\frac{1}{2\pi i}\sum_{k=0}^{o(C)-1}\int_{\sigma^{k}\circ\tilde{\gamma}}\frac{dh}{h}$.

We now invoke the fact that $h\in\mathbb{C}(X)_{\chi}$. Since composing $h$ with $\sigma^{k}$ multiplies it by the scalar $\chi(\sigma^{k})$ and the quotient $\frac{dh}{h}$ is invariant under this operation, our expression for $\mathrm{ord}_{P}h$ becomes just $\frac{o(C)}{2\pi i}\int_{\tilde{\gamma}}\frac{dh}{h}$. But the latter integral is $\log h\big(\tilde{\gamma}(1)\big)-\log h\big(\tilde{\gamma}(0)\big)$, and the fact that $\tilde{\gamma}(1)=\sigma\big(\tilde{\gamma}(0)\big)$ and $h\in\mathbb{C}(X)_{\chi}$ implies that the latter difference is a logarithm of $\chi(\sigma)$ (or equivalently of $\chi(C)$). In total, $\mathrm{ord}_{P}h$ is $\frac{o(C)}{2\pi i}$ times a logarithm of $\chi(C)$, and exponentiating gives that $\chi(C)$ is $e^{2\pi i\mathrm{ord}_{P}h/o(C)}=\zeta_{o(C)}^{\mathrm{ord}_{P}h}$ as desired. This proves the lemma.
\end{proof}

\smallskip

Choose a point $\nu \in S$, and fix it once and for all. The example that the reader should bear in mind is $S=\mathbb{P}^{1}(\mathbb{C})$ and $\nu=\infty$. We would like to normalize $G$-invariant divisors on $X$, with respect to invariant linear equivalence, at all the points in $X$ that are not pre-images of $\nu$. This is simple if $S$ has genus 0, but in general we shall need a preliminary result, as well as some additional notation. For any divisor $\Delta$ on a compact Riemann surface $X$ we denote by $v_{P}(\Delta)$ the multiplicity in which a point $P \in X$ appears in a divisor $\Delta$ on $X$ (i.e., $v_{P}$ is the \emph{valuation} on $\mathrm{Div}(X)$ and on $\mathbb{C}(X)$ that is associated with $P$), and we recall that a divisor $\Delta$ with $v_{P}(\Delta)\geq0$ for every $P \in X$ is called \emph{positive}, an assertion that we denote by $\Delta\geq0_{X}$. A divisor that is not positive is called \emph{non-positive}. In addition we denote  by $L(-\Delta)$, following \cite{[FK]} and others, the space of functions $\phi\in\mathbb{C}(X)$ that either vanish identically or satisfy $\mathrm{ord}_{P}\phi\geq-v_{P}(\Delta)$ for every $P \in X$, and its (finite) dimension by $r(-\Delta)$.

We now prove the following lemma.
\begin{lem}
Every divisor on $S$ is linearly equivalent to a unique divisor of the form $\Upsilon-t\nu$ with $t\in\mathbb{Z}$, where $\Upsilon$ is a positive divisor on $S$ not containing $\nu$ in its support and such that $r(-\Upsilon)=1$. \label{divSinfty}
\end{lem}

\begin{proof}
Take $\Gamma$ to be any divisor on $S$. The dimension $r(-\Gamma-p\nu)$ is non-zero for large enough $p$ (by the Riemann--Roch Theorem), so that there is a meromorphic function $F$ on $S$ such that $\Upsilon=\mathrm{div}(F)+\Gamma+p\nu$ is a positive divisor. This divisor is clearly linearly equivalent to $\Gamma+p\nu$. It therefore suffices to consider divisors of the form $\Upsilon-p\nu$ with $\Upsilon\geq0_{S}$, and we may assume that $\Upsilon$ does not contain $\nu$ in its support (otherwise just cancel it). Now, the space $L(-\Upsilon)$ contains the constant functions (since $\Upsilon$ is positive), and if $r(-\Upsilon)\geq2$ then it also contains a function $F$ vanishing at $\nu$. But then $\mathrm{div}(F)$ is of the form $\nu+\Sigma-\Upsilon$ for some positive divisor $\Sigma$ of smaller degree, so that $\Upsilon-p\nu$ is equivalent to $\Sigma-(p-1)\nu$. Each iteration of this process replaces $\Upsilon$ by a positive divisor of smaller degree, so that the process must terminate after finitely many steps. This shows that every divisor is linearly equivalent to a difference $\Upsilon-t\nu$ with $\Upsilon$ having the desired properties. Moreover, if $\Upsilon-t\nu$ is linearly equivalent to some other divisor of that form, say $\Sigma-p\nu$ with the same properties, then assuming without loss of generality that $p \leq t$ we find that $\Sigma+(t-p)\nu-\Upsilon$ is principal. But as the function yielding this divisor comes from $L(-\Upsilon)$, it has to be a constant function, whose divisor is trivial. It follows that $t=p$ and $\Sigma=\Upsilon$, and uniqueness is also established. This proves the lemma.
\end{proof}
Recalling that positive divisors $\Upsilon$ with $r(-\Upsilon)=1$ are called \emph{non-special}, and that we require that $\nu$ is not in the support of $\Upsilon$ in Lemma \ref{divSinfty}, we shall call divisors satisfying these two properties \emph{non-special divisors on $S\setminus\{\nu\}$} (this is an abuse of terminology, since non-special divisors are defined only on compact Riemann surfaces, but we shall use it nonetheless). Note that non-specialty here implies, in particular, positivity.

Taking back the Galois cover $f:X \to S$ with Galois group $G$ into account, we can now obtain a normalization for $G$-invariant divisors on $X$ (depending on the choice of $\nu \in S$).
\begin{prop}
Let $\Xi$ be a divisor on $X$ that it invariant under the Galois group $G$ of the cover $f:X \to S$. Then there exist a unique non-special divisor $\Upsilon$ on $S\setminus\{\nu\}$ and a unique divisor $\Delta$ on $X$ that is invariantly linearly equivalent to $\Xi$ with the following property: Given a point $\nu\neq\eta \in S$, consider any pre-image $P$ of $\eta$ in $X$, and set $C=\psi(\eta)$. Then the multiplicity $v_{P}(\Delta)$ satisfies the inequalities $o(C)v_{\eta}(\Upsilon) \leq v_{P}(\Delta)<o(C)\big(v_{\eta}(\Upsilon)+1\big)$. \label{norminvdiv}
\end{prop}

\begin{proof}
The divisors that are invariantly linearly equivalent to $\Xi$ are of the form $\Delta=\Xi-\mathrm{div}(h)$ for non-zero $h\in\mathbb{C}(X)_{\mathbf{1}}=f^{*}\mathbb{C}(S)^{\times}$. The $G$-invariance of $\Xi$ allows us to write it in a unique manner as $f^{*}\Gamma+\sum_{\eta \in S}\sum_{P \in f^{-1}(\eta)}l_{\eta}P$, where $\Gamma\in\mathrm{Div}(S)$ and for any $\eta \in S$, with $C=\psi(\eta)$, we have $0 \leq l_{\eta}<o(C)$ (hence the sum is finite since only points $\eta$ with non-trivial $\psi$-images may contribute to it). Now, if $h=f^{*}F$ for $F\in\mathbb{C}(S)^{\times}$ then replacing $\Xi$ by $\Delta=\Xi-\mathrm{div}(h)$ is clearly the same as replacing $\Gamma$ by $\Gamma-\mathrm{div}(F)$ inside the argument of $f^{*}$, and $\Delta$ satisfies the required conditions if and only if $\Gamma-\mathrm{div}(F)$ is the sum of a multiple of $\nu$ and a non-special divisor $\Upsilon$ on $S\setminus\{\nu\}$ (by the restrictions on the multipliers $l_{\eta}$). The existence and uniqueness of $\Delta$ are therefore consequences of the existence and uniqueness of $\Upsilon$ established in Lemma \ref{divSinfty}. This proves the proposition.
\end{proof}
We call the divisor $\Delta$ from Proposition \ref{norminvdiv} the \emph{normalized representative} of the invariant linear equivalence class of $\Xi$, or just \emph{normalized}. It would be more convenient to assume in Proposition \ref{norminvdiv} that $\nu$ is not a branching image of $f$. Then we have a good normalization at all the branch points. Indeed, otherwise $o(C)>1$ for $C=\psi(\nu)$, the order at a pre-image of $\nu$ will have a fixed residue modulo $o(C)$, and we lose information on $\Xi$ by not considering such pre-images.

\smallskip

For any non-trivial conjugacy class $C \subseteq G$ we denote the $r_{C}$ values $\eta \in S\setminus\{\nu\}$ with $\psi(\eta)=C$ by $\eta_{C,j}$ with $1 \leq j \leq r_{C}$. If $G$ is Abelian and $C=\{\sigma\}$ then the number of points is $r_{\sigma}$, and the points themselves are denoted by $\eta_{\sigma,j}$. In the genus 0 case of $S=\mathbb{P}^{1}(\mathbb{C})$ and $\nu=\infty$, these points $\eta_{C,j}$ are just complex numbers. If one replaces the notation $\eta_{C,j}$ by $\lambda_{C,j}$ in this case, Propositions \ref{StoGbr} and \ref{psiZn} show how this generalizes the notation $\lambda_{\alpha,i}$ of \cite{[Z]} for the branching values on a single $Z_{n}$ curve. In the more general Abelian case, presented as a fibered product of $Z_{m}$ curves as in Equation \eqref{fibprod}, our notation deals with the branching in all of the $Z_{m}$ curves together. It is independent of the presentation of $X$ as a fibered product of $Z_{m}$ curves and of generators appearing in $Z_{m}$-equations. Note that we assume that $\nu$ is not a branching image, so that the restriction $\eta\neq\nu$ does not affect $r_{C}$ for any non-trivial $C$. Over $\eta_{C,j}$ there are $\frac{n}{o(C)}$ points, which we denote by $P_{C,j,\upsilon}$ with $1\leq\upsilon\leq\frac{n}{o(C)}$ (or $P_{\sigma,j,\upsilon}$ if $G$ is Abelian). For any point $\eta \in S$ other than the branching images $\eta_{C,j}$ we shall not require a notation for its $n$ pre-images in $X$. On the other hand, if $S=\mathbb{P}^{1}(\mathbb{C})$ then the $n$ points on $X$ lying over $\nu=\infty\in\mathbb{P}^{1}(\mathbb{C})$ will be denoted by $\infty_{\upsilon}$ with $1\leq\upsilon \leq n$. These are the poles of $f^{*}z$, and $f$ is assumed to be non-branched at them.

If $x$ is any rational (or real) number then $\lfloor x \rfloor$ denotes the largest integer $n$ satisfying $n \leq x$, and the fractional part $x-\lfloor x \rfloor$ of $x$ will be denoted by $\{x\}$. For any $\chi\in\widehat{G}$ and any conjugacy class $C \subseteq G$ we define
\begin{equation}
u_{\chi,C}=o(C)\cdot\big\{\tfrac{\log\chi(C)}{2\pi i}\big\}\qquad\mathrm{and}\qquad\textstyle{t_{\chi}=\sum_{C \neq Id_{X}}\frac{r_{C}u_{\chi,C}}{o(C)}} \label{uchiCtchi}
\end{equation}
(so that $u_{\chi,C}$ to be the unique integer in the range $0 \leq u_{\chi,C}<o(C)$ satisfying $\chi(C)=\zeta_{o(C)}^{u_{\chi,C}}$). If $G$ is Abelian and $C$ is the singleton $\{\sigma\}$ then we write $u_{\chi,\sigma}$ for the expression $u_{\chi,C}$ from Equation \eqref{uchiCtchi}, and the sum defining $t_{\chi}$ in that equation is over non-trivial elements of $G$. A simple consequence of the definitions in Equation \eqref{uchiCtchi} is the following one.
\begin{lem}
For any $C$ and $\chi$, the number $u_{\overline{\chi},C}$ equals $o(C)-u_{\chi,C}$ if $u_{\chi,C}>0$ and vanishes otherwise. In addition, we have $t_{\chi}+t_{\overline{\chi}}=\sum_{\{C|\chi(C)\neq1\}}r_{C}$. \label{tuchichibar}
\end{lem}

\begin{proof}
The first assertion follows easily from the fact that $\overline{\chi}(C)=\zeta_{o(C)}^{-u_{\chi,C}}$ and the definition in Equation \eqref{uchiCtchi}. The second assertion is an immediate consequence of the first one. This proves the lemma.
\end{proof}

We can now determine the divisors of normalized functions in the spaces $\mathbb{C}(X)_{\chi}$ for $\chi\in\widehat{G}$. Recall that $\deg\Gamma$ denotes the degree of the divisor $\Gamma$.
\begin{prop}
The number $t_{\chi}$ from Equation \eqref{uchiCtchi} is a non-negative integer for every $\chi\in\widehat{G}$. For every such $\chi$ there exists a non-zero function $h_{\chi}$, unique up to scalar multiples, that spans $\mathbb{C}(X)_{\chi}$ over $\mathbb{C}(S)$ and has the normalized divisor \[f^{*}\big(\Upsilon_{\chi}-(\deg\Upsilon_{\chi}+t_{\chi})\nu\big)+\sum_{C \neq Id_{X}}\sum_{j=1}^{r_{C}}\sum_{\upsilon=1}^{\frac{n}{o(C)}}u_{\chi,C}P_{C,j,\upsilon},\] where $\Upsilon_{\chi}$ is a non-special divisor on $S\setminus\{\nu\}$. Finally, the equalities $t_{\chi}=0$ and $\Upsilon_{\chi}=0_{S}$, or equivalently $t_{\chi}=0$ and $\Upsilon_{\chi}-\deg\Upsilon_{\chi}\cdot\nu$ is principal, occur simultaneously if and only if $\chi=\mathbf{1}$. \label{divgens}
\end{prop}

\begin{proof}
As the space $\mathbb{C}(X)_{\chi}$ is 1-dimensional over $\mathbb{C}(S)$, divisors of non-zero functions there form a single full invariant linear equivalence class of $G$-invariant divisors. The existence and uniqueness of $h_{\chi}$ up to scalars, as well as of $\Upsilon_{\chi}$, are therefore consequences of Proposition \ref{norminvdiv}. Lemma \ref{charandord} shows that for any non-zero $h\in\mathbb{C}(X)_{\chi}$, the order $\mathrm{ord}_{P_{C,j,\upsilon}}h$ has to be congruent modulo $o(C)$ to the number $u_{\chi,C}$ from Equation \eqref{uchiCtchi}. It follows that for $\eta=\eta_{C,j}$, the number denoted by $l_{\eta}$ in the proof of Proposition \ref{norminvdiv} equals $u_{\chi,C}$. This shows that our expression for $\mathrm{div}(h_{\chi})$ is indeed the required one, up to multiples of $\nu$. We denote the multiple of $\nu$ in the argument of $f^{*}$ in $\mathrm{div}(h_{\chi})$ by $-(\deg\Upsilon_{\chi}+\tilde{t}_{\chi})$ for some integer $\tilde{t}_{\chi}$, and recall that $\mathrm{div}(h_{\chi})$ must have degree 0. Now, $f^{*}$ is known to multiply the degrees of divisors by $n$, and $\Upsilon_{\chi}-(\deg\Upsilon_{\chi})\nu$ has degree 0. This yields the vanishing of the difference $\sum_{C \neq Id_{X}}\frac{nr_{C}u_{\chi,C}}{o(C)}-n\tilde{t}_{\chi}$, from which we deduce that $t_{\chi}=\tilde{t}_{\chi}\in\mathbb{Z}$ (by Equation \eqref{uchiCtchi}), and non-negativity is also clear.

For the last assertion, one direction is immediate: Non-zero constant functions lie in $\mathbb{C}(X)_{\mathbf{1}}$ and their (trivial) divisor is normalized, so that indeed $\Upsilon_{\mathbf{1}}=0_{S}$ and $t_{\mathbf{1}}=0$ (the latter equality also follows immediately from Equation \eqref{uchiCtchi}, since $u_{\mathbf{1},C}=0$ for any $C$). On the other hand, if $t_{\chi}=0$ then we already know that $\mathrm{div}(h_{\chi})$ reduces to $f^{*}\big(\Upsilon_{\chi}-(\deg\Upsilon_{\chi})\nu\big)$. Hence if $\Upsilon_{\chi}=0_{S}$ then $\mathrm{div}(h_{\chi})=0$, the function $h_{\chi}$ is a constant from $\mathbb{C}(X)_{\mathbf{1}}$, and $\chi=\mathbf{1}$. Finally, if $\Upsilon_{\chi}-\deg\Upsilon_{\chi}\cdot\nu$ is principal then the function having this divisor lies in $L(-\Upsilon_{\chi})$, which must be a constant by the non-specialty of $\Upsilon_{\chi}$. As this implies that $\Upsilon_{\chi}=\deg\Upsilon_{\chi}\cdot\nu$, and $\Upsilon_{\chi}$ does not contain $\nu$ in its support, this condition is indeed equivalent to $\Upsilon_{\chi}=0_{S}$. This completes the proof of the proposition.
\end{proof}

\smallskip

We would like to relate these results, in the Abelian case, to the description of $X$ as a fibered product of cyclic covers of the quotient curve $S$ . The fibered product structure from Equation \eqref{fibprod} makes it clear that the functions $\prod_{l=1}^{q}w_{l}^{e_{l}}$ with $0 \leq e_{l}<m_{l}$ for every $1 \leq l \leq q$ form a basis for $\mathbb{C}(X)$ over $f^{*}\mathbb{C}(S)$. We shall denote the function associated with $E=(e_{l})_{l=1}^{q}$ by $w^{E}$.
\begin{prop}
Given $E=(e_{l})_{l=1}^{q}\in\mathbb{Z}^{q}$, the space $\mathbb{C}(S)w^{E}$ depends only on the image of $E$ in $\prod_{l=1}^{q}(\mathbb{Z}/m_{l}\mathbb{Z})$. Moreover, the decomposition of $\mathbb{C}(X)$ as $\bigoplus_{E\in\prod_{l=1}^{q}(\mathbb{Z}/m_{l}\mathbb{Z})}\mathbb{C}(S)w^{E}$ coincides with its decomposition as $\bigoplus_{\chi\in\widehat{G}}\mathbb{C}(X)_{\chi}$. \label{charZn}
\end{prop}

\begin{proof}
The first assertion follows from the fact that altering $e_{l}$ by a multiple of $m_{l}$ multiplies $w_{l}^{e_{l}}$ by a power of the function $F_{l}\in\mathbb{C}(S)^{\times}$. For the second one we recall that $\tau_{l}$ multiplies $w_{l}$ by $\zeta_{m_{l}}$ and leaves elements of $f^{*}\mathbb{C}(S)$ and the other $w_{k}$s invariant. The action of that automorphism on $f^{*}F \cdot w^{E}$, where $F$ is an arbitrary function in $\mathbb{C}(S)$, thus multiplies it by $\zeta_{m_{l}}^{e_{l}}$. The second assertion now follows from the fact that a character of $G$ is determined by its values on the generators $\tau_{l}$, $1 \leq l \leq q$, while the image of $\tau_{l}$ under a character can be taken arbitrarily from the powers of $\zeta_{m_{l}}$ independently of the images of the other generators. This proves the proposition.
\end{proof}

The condition for the irreducibility of $X$ from Equation \eqref{fibprod} can now be explained. The number $\beta$ associated with the product $w^{E}$ is the order of the corresponding character $\chi$ from Proposition \ref{charZn}. If $\prod_{l=1}^{q}F_{l}^{e_{l}\beta/m_{l}}$ is a $\beta$th power then $w^{E}$ would belong to $\mathbb{C}(S)$, which cannot be the case if $\beta>1$ (i.e., if $\chi\neq\mathbf{1}$) since $G$ operates on $X$ without a kernel as a subgroup of $Aut(X)$.

We remark that taking $w_{l}$ to be the function $h_{\chi}$ from Proposition \ref{divgens} associated with the appropriate character from Proposition \ref{charZn} would produce (up to scalars) the equation that is normalized in the sense of \cite{[FZ]} and \cite{[Z]}. On the other hand, even after doing so, the normalized function from Proposition \ref{divgens} that is associated with another character, corresponding to some $w^{E}$ for $E=(e_{l})_{l=1}^{q}$ with $0 \leq e_{l}<m_{l}$ for every $l$, will in general still not be $w^{E}$ itself.

\smallskip

In the special case where $S$ has genus 0, i.e., $S=\mathbb{P}^{1}(\mathbb{C})$ and $\nu=\infty$ is not a branching image, Lemma \ref{divSinfty} and Propositions \ref{norminvdiv} and \ref{divgens} take the following simpler and more explicit form (Proposition \ref{charZn} remains the same).
\begin{cor}
Every divisor on $S=\mathbb{P}^{1}(\mathbb{C})$ is linearly equivalent to the multiple of $\infty$ having the same degree. If $f:X\to\mathbb{P}^{1}(\mathbb{C})$ is a Galois cover then any $G$-invariant divisor $\Delta$ is invariantly linearly equivalent to a unique divisor involving only the points $P_{C,j,\upsilon}$ with multiplicities between 0 and $o(C)-1$ (depending on $C$ and $j$ but not on $\upsilon$) and the poles $\infty_{\upsilon}$, $1\leq\upsilon \leq n$ of $f^{*}z$ (all with the same multiplicity). The divisor of the function $h_{\chi}$ from Proposition \ref{norminvdiv} is \[\sum_{C \neq Id_{X}}\sum_{j=1}^{r_{C}}\sum_{\upsilon=1}^{\frac{n}{o(C)}}u_{\chi,C}P_{C,j,\upsilon}-t_{\chi}\sum_{\upsilon=1}^{n}\infty_{\upsilon},\] and $t_{\chi}$ is a strictly positive integer for any $\mathbf{1}\neq\chi\in\widehat{G}$.
\label{P1Cnorm}
\end{cor}
For the proof, recall that every point $\eta\in\mathbb{C}$ is linearly equivalent to $\infty$ via the divisor of the function $z-\eta$. The rest of the proof uses the same arguments from the general case, in which some instances of $f^{*}\nu=f^{*}\infty$ are replaced by $\sum_{\upsilon=1}^{n}\infty_{\upsilon}$ (for the last assertion, about $t_{\chi}$, recall that $\Upsilon_{\chi}=0_{S}$ for every $\chi$ in this case). Another point of view on the first assertion of Corollary \ref{P1Cnorm} is via the Riemann--Roch Theorem, which implies that the condition $r(-\Upsilon)=1$ can hold for a positive divisor $\Upsilon$ on $S$ only if $\deg\Upsilon \leq g_{S}$, indeed leaving only the trivial divisor as a possibility for $\Upsilon$ in case $g_{S}=0$. The positivity of the numbers $t_{\chi}$ for $\chi\neq\mathbf{1}$ in the last assertion of Corollary \ref{P1Cnorm} is related, in the Abelian case, to $X$ being irreducible. We also remark that in this case, when $g_{S}=0$, $G$ is Abelian, and the $w_{l}$s are normalized, if $\chi\in\widehat{G}$ is associated with the function $w^{E}$ from Proposition \ref{charZn}, the function $h_{\chi}$ is just $w^{E}$ divided by a polynomial having roots in those branching values in which $w^{E}$ has zeros of too large orders.

\section{Function Spaces And Invariant Divisors \label{LDeltaInv}}

Let $G$ is a finite subgroup of $Aut(X)$, take a character $\chi\in\widehat{G}$, and consider again the space $L(-\Delta)$ for a divisor $\Delta$ on $X$. The intersection $L(-\Delta)\cap\mathbb{C}(X)_{\chi}$ will be denoted by $L(-\Delta)_{\chi}$, and its dimension by $r_{\chi}(-\Delta)$. Here we investigate these dimensions for $G$-invariant divisors $\Delta$.

Recall that if $\Xi$ and $\Delta$ are linearly equivalent divisors then $r(-\Delta)=r(-\Xi)$, since if $\Delta=\mathrm{div}(h)+\Xi$ for some $h\in\mathbb{C}(X)^{\times}$ then multiplication by $h$ defines an isomorphism from $L(-\Delta)$ onto $L(-\Xi)$. Moreover, this isomorphism is ``canonical up to scalars'', since the function $h$ is determined by $\mathrm{div}(h)$ up to scalars. Lemma \ref{invlineq} allows us to describe the action on the components associated with characters of $\widehat{G}$ as follows.
\begin{lem}
If the two $G$-invariant divisors $\Delta$ and $\Xi$ are linearly equivalent, then there exists a canonical element $\rho\in\widehat{G}$ such that $r_{\chi}(-\Delta)=r_{\chi\rho}(-\Xi)$ for every $\chi\in\widehat{G}$. The canonical isomorphism takes each $L(-\Delta)_{\chi}$ onto $L(-\Xi)_{\chi}$ (in case the spaces are non-trivial) if and only if $\Delta$ and $\Xi$ are invariantly linearly equivalent. \label{actCXrho}
\end{lem}

\begin{proof}
Lemma \ref{invlineq} shows that if $\Delta$ and $\Xi$ are linearly equivalent then there exists some non-zero function $h$, lying in one of the spaces $\mathbb{C}(X)_{\rho}$ with $\rho\in\widehat{G}$, such that $\Delta=\mathrm{div}(h)+\Xi$. Moreover, since $h$ is unique up to scalars, $\rho$ depends only on $\Delta$ and $\Xi$. Now, since $G$ operates on $h$ via $\rho$, the isomorphism $L(-\Delta) \to L(-\Xi)$ defined by multiplication by $h$ clearly takes $L(-\Delta)_{\chi}$ onto $L(-\Xi)_{\chi\rho}$ for any $\chi\in\widehat{G}$. Taking dimensions yields the first assertion. The second assertion now follows from the fact that $\rho=\mathbf{1}$ if and only if $\Delta$ and $\Xi$ are invariantly linearly equivalent (where non-triviality is required since there is only one map between the 0 spaces). This proves the lemma.
\end{proof}

We shall also be using the following lemma.
\begin{lem}
If a positive, $G$-invariant divisor $\Delta$ satisfies $r_{\mathbf{1}}(-\Delta)=1$ then $\Delta$ is normalized, and if $\Upsilon$ is the associated non-special divisor on $S\setminus\{\nu\}$ and $\xi \in X$ is any pre-image of $\nu$ then $v_{\xi}(\Delta) \leq g_{S}-\deg\Upsilon$. The assertion about normalization holds also for non-positive $G$-invariant divisors $\Delta$ satisfying $r_{\mathbf{1}}(-\Delta)=0$, provided that $\Delta+f^{*}\nu\geq0_{X}$. \label{nonspres}
\end{lem}

\begin{proof}
The proof of Proposition \ref{norminvdiv} shows that any $G$-invariant divisor $\Delta$ on $X$ can be written as $f^{*}\Gamma$ for $\Gamma\in\mathrm{Div}(S)$ plus a divisor involving only the points $P_{C,j,\upsilon}$ with multiplicities $0 \leq l_{C,j} \leq o(C)-1$ (independently of $\upsilon$). Recall that for every such $\Gamma$ and every point $P \in X$ with $\eta=f(P)$ and $C=\psi(\eta)$ we have $v_{P}(f^{*}\Gamma)=o(C)v_{\eta}(\Gamma)$, and similarly $\mathrm{ord}_{P}f^{*}F=o(C)\mathrm{ord}_{\eta}F$ for every $F\in\mathbb{C}(S)^{\times}$. We deduce that $\Delta$ is positive if and only if $\Gamma$ is positive, and since $\mathbb{C}(X)_{\mathbf{1}}=f^{*}\mathbb{C}(S)$ we also obtain the equality $L(-\Delta)_{\mathbf{1}}=f^{*}L(-\Gamma)$. Assuming now that $\Delta\geq0_{X}$ and $r_{\mathbf{1}}(-\Delta)=1$, we get $\Gamma\geq0_{S}$ and $r(-\Gamma)=1$, implying that the divisor $\Upsilon=\Gamma-v_{\nu}(\Gamma)\cdot\nu$ is positive as well, with $v_{\nu}(\Upsilon)=0$. Moreover, since $\Gamma=\Upsilon+v_{\nu}(\Gamma)\cdot\nu\geq\Upsilon\geq0_{S}$ we deduce that $\Upsilon$ is a non-special divisor on $S\setminus\{\nu\}$, so that $\Delta$ is normalized by Proposition \ref{norminvdiv}. In addition, since $\nu$ is assumed not to be a branching image we obtain the equality $v_{\xi}(\Delta)=v_{\nu}(\Gamma)=\deg\Gamma-\deg\Upsilon$ for every pre-image $\xi$ of $\nu$ (as above). But as $\Gamma\geq0_{S}$ can be non-special on $S$ only if $\deg\Gamma \leq g_{S}$ (by the Riemann--Roch Theorem), the bound on $v_{\xi}(\Delta)$ is proved as well.

Now, if $\Delta$ is not positive but $f^{*}\nu+\Delta\geq0_{X}$ then $\Gamma$ must be of the form $\Upsilon-\nu$ for some positive divisor $\Upsilon$ on $S$ with $v_{\nu}(\Upsilon)=0$. The same argument now shows that $r(\nu-\Upsilon)=0$, and as it differs from $r(-\Upsilon)$ by at most 1 and $L(-\Upsilon)$ contains the constant functions (since $\Upsilon\geq0_{S}$), we again deduce that $\Upsilon$ is a non-special divisor on $S\setminus\{\nu\}$. Proposition \ref{norminvdiv} thus establishes the assertion about normalization also in this case. This proves the lemma.
\end{proof}
Note that in the first case considered in Lemma \ref{nonspres} the divisor $\Delta$ is normalized with respect to any base point $\nu \in S$, while in the second case the dependence on $\nu$ appears in condition $\Delta+f^{*}\nu\geq0_{X}$.

Here there is a difference, in the form of the possible divisors satisfying the conditions of Lemma \ref{nonspres}, between the case where $g_{S}=0$ and the case of positive $g_{S}$. In the latter case every divisor of the form $\Delta=f^{*}\eta$ for $\eta \in S$ (and there are infinitely many such divisors, no two of them being invariantly linearly equivalent) satisfies $r_{\mathbf{1}}(-\Delta)=1$. On the other hand, when $S=\mathbb{P}^{1}(\mathbb{C})$ (and $\nu=\infty$) we have a much nicer result.
\begin{cor}
In the case of a Galois cover $f:X\to\mathbb{P}^{1}(\mathbb{C})$, all the positive divisors from Lemma \ref{nonspres} are supported on the branch points of $f$. There are finitely many divisors of both types considered in that lemma (i.e., positive divisors versus non-positive divisors $\Delta$ with $\Delta+f^{*}\infty\geq0_{X}$). \label{brsup}
\end{cor}

\begin{proof}
Corollary \ref{P1Cnorm} shows that in this case there are no non-trivial non-special divisors on $\mathbb{P}^{1}(\mathbb{C})\setminus\{\infty\}$. Moreover, the bound on $v_{\xi}(\Delta)$ in Lemma \ref{nonspres} reduces here to 0, so that only the branch points may appear in the positive divisors there. In addition, they appear with bounded multiplicities: The multiplicity of $P_{C,j,\upsilon}$ in such a divisor is between 0 and $o(C)-1$. The same assertion holds for the non-positive divisors, except for the poles $\infty_{\upsilon}$, $1\leq\upsilon \leq n$ of $f^{*}z$ that now appear with multiplicity $-1$. Since only finitely many points may appear, and with bounded multiplicities, the finiteness of the number of possible divisors also follows. This proves the corollary.
\end{proof}

Applying Corollary \ref{brsup} to non-special positive divisors on $Z_{n}$ curves explains where the condition that the divisors be supported on the branch points of fully ramified $Z_{n}$ curves, appearing in, e.g., \cite{[FZ]}, \cite{[K1]}, \cite{[K2]}, and \cite{[Z]}, comes from. Indeed, this condition is equivalent to invariance under the cyclic Galois group of the $Z_{n}$ cover,  since in the fully ramified $Z_{n}$ curve case every branch point alone is already $G$-invariant. It also follows that Lemma \ref{nonspres} reduces, in the fully ramified cyclic case, to Lemma 1.5 of \cite{[Z]}.

\smallskip

Lemma \ref{nonspres} allows us to restrict attention, for many questions, to normalized $G$-invariant divisors. We recall that the branch points on $X$ are the points $P_{C,j,\upsilon}$ with $C$ a non-trivial conjugacy class in $G$, $1 \leq j \leq r_{C}$, and $1\leq\upsilon\leq\frac{n}{o(C)}$, and every such point maps via $f$ to $\eta_{C,j} \in S\setminus\{\nu\}$. An explicit expression of a normalized $G$-invariant divisor is given in terms of sets of points, as follows. For every nontrivial $C$ we partition the $r_{C}$ elements $\eta_{C,j}$, $1 \leq j \leq r_{C}$ of $S$ into $o(C)$ sets $B_{C,i}$, $0 \leq i<o(C)$. For any such set $B_{C,i}$ and any integer $e$, we denote by $eB_{C,i}$ the divisor on $X$ that is obtained as the sum of the $e$th multiples of all the pre-images of elements of $B_{C,i}$ in $X$ (i.e., $eB_{C,i}$ is the sum $\sum_{\{j|\eta_{C,j} \in B_{C,i}\}}\sum_{\upsilon=1}^{n/o(C)}eP_{C,j,\upsilon}$). In addition, it will turn out useful to consider, for a character $\chi\in\widehat{G}$ and a non-trivial conjugacy class $C \subseteq G$, the set
\begin{equation}
A_{C,\chi}=\bigcup_{i=0}^{u_{\chi,C}-1}B_{C,i}. \label{ACchi}
\end{equation}
As usual, when $G$ is Abelian we denote the sets in the partition by $B_{\sigma,i}$ with $Id_{X}\neq\sigma \in G$, and the sets from Equation \eqref{ACchi} by $A_{\sigma,\chi}$.

A general normalized $G$-invariant divisor $\Delta$ on $X$ is based on such a partition, on a non-special divisor $\Upsilon$ on $S\setminus\{\nu\}$, and on an integer $p$. The divisor $\Delta$ with these parameters is given, in the general case and in the particular genus 0 case of $S=\mathbb{P}^{1}(\mathbb{C})$ and $\nu=\infty$, by
\begin{equation}
f^{*}\big(\Upsilon+(p-\deg\Upsilon)\cdot\nu\big)+\sum_{C \neq Id_{X}}\sum_{i=0}^{o(C)-1}\big(o(C)-1-i\big)B_{C,i} \label{normDelgen}
\end{equation}
and
\begin{equation}
\sum_{C \neq Id_{X}}\sum_{i=0}^{o(C)-1}\big(o(C)-1-i\big)B_{C,i}+p\sum_{\upsilon=1}^{n}\infty_{\upsilon} \label{normDelP1C}
\end{equation}
respectively (since in the latter case $\Upsilon=0_{S}$ and we have substituted the multiple of $f^{*}\infty$). Wherever we use Equation \eqref{normDelgen} in what follows, the divisor $\Upsilon$ should be understood to be a non-special divisor on $S\setminus\{\nu\}$. The degree $\deg\Delta$ of the divisor $\Delta$ from both Equations \eqref{normDelgen} and \eqref{normDelP1C} equals
\begin{equation}
\sum_{C \neq Id_{X}}\sum_{i=0}^{o(C)-1}\frac{n\big(o(C)-1-i\big)}{o(C)}|B_{C,i}|+np, \label{degDelta}
\end{equation}
where here and throughout $|Y|$ stands for the cardinality of the finite set $Y$. Note that while not all the branch points necessarily appear in $\Delta$, the sets $B_{C,i}$, $0 \leq i<o(C)$ must always form a full partition of the points $\eta_{C,j}$, $1 \leq j \leq r_{C}$. Those points $\eta_{C,j}$ whose pre-images in $X$ do not appear in $\Delta$ in Equation \eqref{normDelgen} or \eqref{normDelP1C} are precisely those that lie in $B_{C,o(C)-1}$. The divisors from Lemma \ref{nonspres} correspond to $p=0$ and to $p=-1$ in these equations.

Note that there is a difference between Equation \eqref{normDelP1C} for the genus 0 case here and the notation from \cite{[FZ]} and \cite{[Z]}, in that here we work with partitions of the $f$-images of the branch points rather than the branch points themselves. While this makes the presentation of $\Delta$ using these sets a bit less straightforward (a more direct way for expressing $\Delta$ would be by using the pointwise pre-images of the sets $B_{C,i}$), we keep the cardinalities of the sets $B_{C,i}$ arbitrary with $\sum_{i=0}^{o(C)-1}|B_{C,i}|=r_{C}$ (while the subset $f^{-1}(B_{C,i})$ of $X$ would have cardinality $\frac{n}{o(C)}|B_{C,i}|$). In \cite{[FZ]} and \cite{[Z]} every $\sigma \in G$ with $r_{\sigma}>0$ was assumed to satisfy $o(\sigma)=n$ (this is full ramification), explaining why this difference is not visible in these references.

\smallskip

We can now prove the following assertion about the dimensions of the spaces $L(-\Delta)_{\chi}$ for normalized $G$-invariant divisors $\Delta$ and characters $\chi\in\widehat{G}$.
\begin{prop}
Take a normalized $G$-invariant divisor $\Delta$ on $X$ written as in Equation \eqref{normDelgen}, and a character $\chi\in\widehat{G}$. Then $r_{\chi}(-\Delta)$ equals \[r\bigg((\deg\Upsilon_{\chi}+\deg\Upsilon+t_{\chi}-p)\cdot\nu-\Upsilon_{\chi}-\Upsilon-\sum_{C \neq Id_{X}}\sum_{j \in A_{C,\chi}}\eta_{C,j}\bigg),\] where $t_{\chi}$ is the number from Equation \eqref{uchiCtchi}, $\Upsilon_{\chi}$ is the divisor appearing in Proposition \ref{divgens}, and $A_{C,\chi}$ is the set from Equation \eqref{ACchi}. \label{L1/Deltachi}
\end{prop}

\begin{proof}
The proof of Lemma \ref{actCXrho} shows that division by the function $h_{\chi}$ from Proposition \ref{divgens} takes the space $L(-\Delta)_{\chi}$ isomorphically onto $L\big(-\Delta-\mathrm{div}(h_{\chi})\big)_{\mathbf{1}}$, where the divisor $\Delta+\mathrm{div}(h_{\chi})$ is also $G$-invariant by the easy direction of Lemma \ref{invlineq}. The proof of Lemma \ref{nonspres} implies that if we decompose the latter divisor as the sum of a divisor $f^{*}\Gamma$ for some $\Gamma\in\mathrm{Div}(S)$ plus some branch points $P_{C,j,\upsilon}$ with non-negative multiplicities smaller than $o(C)$ then the required dimension $r_{\mathbf{1}}\big(-\Delta-\mathrm{div}(h_{\chi})\big)$ would be the same as $r(-\Gamma)$. The parts $f^{*}\big(\Upsilon+(p-\deg\Upsilon)\cdot\nu)$ of $\Delta$ and $f^{*}\big(\Upsilon_{\chi}-(\deg\Upsilon_{\chi}+t_{\chi})\cdot\nu)$ of $\mathrm{div}(h_{\chi})$ combine to the part of $f^{*}\Gamma$ involving $\Upsilon$, $\Upsilon_{\chi}$, and the multiple of $\nu$. The remaining expressions are just the sums $\sum_{\upsilon=1}^{n/o(C)}\big(o(C)-1-i+u_{\chi,C}\big)P_{C,j,\upsilon}$ for a non-trivial conjugacy class $C$ in $G$ and some index $1 \leq j \leq r_{C}$ such that $\eta_{C,j}$ lies in the set $B_{C,i}$. This remains a normalized expression if $i \geq u_{\chi,C}$, but otherwise it is the sum of $f^{*}\eta_{C,j}$ and a normalized expression. As the case where this product is not normalized occurs precisely when $i<u_{\chi,C}$, i.e., when $j$ is in the set $A_{C,\chi}$ from Equation \eqref{ACchi}, these are the $\eta_{C,j}$s which enter $\Gamma$ as well. This proves the proposition.
\end{proof}

Once again, in the quotient genus 0 case we have simpler results, as well as a more explicit description of the spaces $L(-\Delta)_{\chi}$ themselves. For this we adopt the following notation from \cite{[FZ]} and \cite{[Z]}. Given any integer $d\geq-1$, we denote by $\mathcal{P}_{\leq d}(z)$ the $(d+1)$-dimensional vector space of polynomials of degree not exceeding $d$ in $z$ (with complex coefficients). The 0 space is written here as $\mathcal{P}_{\leq-1}(z)$ to keep the dimension as $d+1$ also in this case. These spaces appear naturally in our setting of normalized divisors on $S=\mathbb{P}^{1}(\mathbb{C})$ since for $d\geq-1$ we have $L(-d\cdot\infty)=\mathcal{P}_{\leq d}(z)$ (and $L(-d\cdot\infty)=\{0\}$ for $d\leq-1$): Indeed, functions from $\mathbb{C}(S)=\mathbb{C}(z)$ in that space cannot have poles in finite points, so they must be polynomials in $z$, and the assertion now follows from the fact that the order of a polynomial $p(z)$ at $\infty$ is minus the degree of $p$.

When $f:X\to\mathbb{P}^{1}(\mathbb{C})$ is a Galois cover, we write $\mathcal{P}_{\leq d}(f^{*}z)$ for the appropriate space of polynomials in $f^{*}z$ (as a subspace of $\mathbb{C}(X)_{\mathbf{1}}=f^{*}\mathbb{C}(S)=\mathbb{C}(f^{*}z)$ inside $\mathbb{C}(X)$). Proposition \ref{L1/Deltachi} and its proof then have the following consequence.
\begin{cor}
If $S=\mathbb{P}^{1}(\mathbb{C})$, $\nu=\infty$, and $\Delta$ is as in Equation \eqref{normDelP1C} then the space $L(-\Delta)_{\chi}$ is $h_{\chi}\mathcal{P}_{d(\chi)}(f^{*}z)\big/\prod_{C \neq Id_{X}}\prod_{j \in A_{C,\chi}}(f^{*}z-\eta_{C,j})$, where the value of $d(\chi)$ is $p+\sum_{C \neq Id_{X}}|A_{C,\chi}|-t_{\chi}$ (or $-1$ if the latter number is negative). The dimension $r_{\chi}(-\Delta)$ is $\max\big\{0,p+1+\sum_{C \neq Id_{X}}|A_{C,\chi}|-t_{\chi}\big\}$. \label{funspP1C}
\end{cor}

\begin{proof}
We apply the proof of Proposition \ref{L1/Deltachi}, recalling that $\Upsilon=\Upsilon_{\chi}=0_{S}$ in this case, and that each of the points $\eta_{C,j}\in\mathbb{C}$ is equivalent to $\infty$ via the function $z-\eta_{C,j}$ (see Corollary \ref{P1Cnorm}). The space in question is thus $f^{*}L\big(-d(\chi)\cdot\infty\big)$ multiplied by the function $h_{\chi}\big/\prod_{C \neq Id_{X}}\prod_{j \in A_{C,\chi}}(f^{*}z-\lambda_{C,j})$, and its dimension is $r\big(-d(\chi)\cdot\infty\big)$. The first assertion is then a consequence of the structure of these spaces, and the second one immediately follows since the dimension of $\mathcal{P}_{\leq d}(z)$ is $d+1$ for any $d\geq-1$. This proves the corollary.
\end{proof}

In particular, by setting $\chi=\mathbf{1}$ we get $u_{\mathbf{1},C}=0$ for every $C$ in Equation \eqref{uchiCtchi}, hence also $t_{\mathbf{1}}=0$. In addition we get $A_{C,\mathbf{1}}=\emptyset$ for every $C$ in Equation \eqref{ACchi}, and $h_{\mathbf{1}}$ is a constant function in Proposition \ref{divgens}. Hence Corollary \ref{funspP1C} reduces to the assertion that if $\Delta$ is given in Equation \eqref{normDelP1C} then $L(-\Delta)_{\mathbf{1}}$ is just $\mathcal{P}_{\leq d}(f^{*}z)$ for $d=\max\{-1,p\}$ (this is the required space $f^{*}L(-p\cdot\infty)$), of dimension $\max\{0,p+1\}$.

\section{Abelian Covers of $\mathbb{P}^{1}(\mathbb{C})$ \label{AbCovP1C}}

Consider now the case where the finite subgroup $G \subseteq X$ is Abelian. In this case the direct sum $\bigoplus_{\chi\in\widehat{G}}\mathbb{C}(X)_{\chi}$ is the full space $\mathbb{C}(X)$. We also have the following generalization of the part of Proposition 1.2 of \cite{[Z]} involving functions.
\begin{lem}
If $G$ is Abelian and $\Delta$ is a $G$-invariant divisor then $L(-\Delta)$ is the direct sum $\bigoplus_{\chi\in\widehat{G}}L(-\Delta)_{\chi}$ and $r(-\Delta)=\sum_{\chi\in\widehat{G}}r_{\chi}(-\Delta)$. \label{decom}
\end{lem}

\begin{proof}
For an arbitrary divisor $\Delta$, if $f \in L(-\Delta)$ and $\sigma \in G$ then $\sigma(f)$ lies in $L\big(-\sigma(\Delta)\big)$. This shows that if $\Delta$ is $G$-invariant then $L(-\Delta)$ is a representation space of $G$. The two statements thus follow from the decomposition theorem for representations of Abelian groups. This proves the lemma.
\end{proof}
We note that if $G$ is not necessarily Abelian and $\Delta$ is $G$-invariant then $\bigoplus_{\chi\in\widehat{G}}L(-\Delta)_{\chi}$ is the space of functions in $L(-\Delta)$ on which the commutator subgroup $[G,G]$ operates trivially.

Assume now that $G$ is the Galois group of the Abelian cover $f:X \to S$.
\begin{prop}
Let $\Delta$ be a normalized $G$-invariant divisor on $X$, presented as in Equation \eqref{normDelgen} (but with every index $C$ replaced by $\sigma$). For every $\chi\in\widehat{G}$ consider the number $t_{\chi}$ from Equation \eqref{uchiCtchi}, the divisor $\Upsilon_{\chi}$ from Proposition \ref{divgens}, and the sets $A_{\sigma,\chi}$ with $Id_{X}\neq\sigma \in G$ from Equation \eqref{ACchi}. Then the dimension
$r(-\Delta)$ equals
\[\sum_{\chi\in\widehat{G}}r\bigg(\big(\deg(\Upsilon_{\chi}+\Upsilon)+t_{\chi}-p\big)\cdot\nu-\Upsilon_{\chi}-\Upsilon-\sum_{\sigma \neq Id_{X}}\sum_{j \in A_{\sigma,\chi}}\eta_{\sigma,j}\bigg).\] \label{r1/Delta}
\end{prop}

\begin{proof}
Lemma \ref{decom} shows that $r(-\Delta)=\sum_{\chi\in\widehat{G}}r_{\chi}(-\Delta)$, and the dimensions $\{r_{\chi}(-\Delta)\}_{\chi\in\widehat{G}}$ are evaluated in Proposition \ref{L1/Deltachi}. This proves the proposition.
\end{proof}

Let us see how for the trivial divisor $\Delta=0_{X}$ (which is clearly normalized), we recover the usual decomposition of $L(0_{X})=\mathbb{C}$ as the direct sum of $L(0_{X})_{\mathbf{1}}=\mathbb{C}$ and $n-1$ zero spaces. In the notation of Equation \eqref{normDelgen} we get $p=0$ and $\Upsilon=0_{S}$, as well as $\eta_{\sigma,j} \in B_{\sigma,o(\sigma)-1}$ and $u_{\chi,\sigma}<o(\sigma)$ for every $\chi$ and $\sigma$. The sets $A_{\sigma,\chi}$ from Equation \eqref{ACchi} are therefore all empty. Hence we obtain the sum of the numbers $r\big((\deg\Upsilon_{\chi}+t_{\chi})\cdot\nu-\Upsilon_{\chi}\big)$, which for $\chi=\mathbf{1}$ equals just $r(0_{S})=1$. On the other hand, Proposition \ref{divgens} implies that if $\chi\neq\mathbf{1}$ then either $t_{\chi}>0$ and we get the (vanishing) dimension associated with a divisor of negative degree, or $t_{\chi}=0$ and $\Upsilon_{\chi}$ is a non-trivial divisor not linearly equivalent to $\deg\Upsilon_{\chi}\cdot\nu$, so that again this dimension vanishes. By introducing the \emph{Kronecker delta symbol} $\delta_{\alpha,\beta}$, which equals 1 when the two objects (numbers, characters, representations, etc.) $\alpha$ and $\beta$ coincide and vanishes otherwise, we can summarize this paragraph more succinctly in the equality $r_{\chi}(0_{X})=\delta_{\chi,\mathbf{1}}$.

For genus 0 we thus deduce the following result.
\begin{cor}
For a normalized $G$-invariant divisor $\Delta$ on an Abelian cover $X$ of $\mathbb{P}^{1}(\mathbb{C})$, we have $r(-\Delta)=\sum_{\chi\in\widehat{G}}\max\big\{0,p+1+\sum_{\sigma \neq Id_{X}}|A_{\sigma,\chi}|-t_{\chi}\big\}$. \label{totdimP1C}
\end{cor}

\begin{proof}
We apply the same argument appearing in the proof of Proposition \ref{r1/Delta}, combined with the results of Corollary \ref{funspP1C}. This proves the corollary.
\end{proof}

The evaluation for the trivial divisor appearing above is in correspondence with the fact that in the genus 0 case the number $\max\big\{0,1-t_{\chi}\big\}$ equals 1 for $\chi=\mathbf{1}$ and vanishes otherwise when $g_{S}=0$.

\smallskip

In the general theory of theta characteristics, non-special positive divisors of degree $g_{X}$ on $X$ are very important. Finding all these divisors in our setting is difficult in general (e.g., because of the divisors $\Upsilon_{\chi}$ and $\Upsilon$, on which we have no control), but when $S$ has genus 0 we can prove the following generalization of Theorem 1.6 of \cite{[Z]}.
\begin{thm}
The $G$-invariant positive divisors $\Delta$ of degree $g_{X}$ on the Abelian cover $X$ of $\mathbb{P}^{1}(\mathbb{C})$ that are non-special are precisely the normalized divisors whose presentation in Equation \eqref{normDelP1C} have the following properties: The parameter $p$ vanishes, and the cardinality condition $\sum_{\sigma \neq Id_{X}}\sum_{i=0}^{u_{\chi,\sigma}-1}|B_{\sigma,i}|=t_{\chi}-1$ holds for every $\mathbf{1}\neq\chi\in\widehat{G}$, where $t_{\chi}$ is the number from Equation \eqref{uchiCtchi}. There are finitely many such divisors. \label{intnonsp}
\end{thm}
Note that the indices of the sets from Equation \eqref{normDelP1C} are written again as $\sigma$ rather than $C$, since $G$ is Abelian.

\begin{proof}
Lemma \ref{nonspres} allows us to restrict attention to normalized divisors with $p=0$, and Corollary \ref{brsup} establishes the finiteness of the number of such divisors (since $S=\mathbb{P}^{1}(\mathbb{C})$ here). Corollary \ref{totdimP1C} implies that in this case $\Delta$ is non-special if and only if the numbers $\max\big\{0,1+\sum_{\sigma \neq Id_{X}}|A_{\sigma,\chi}|-t_{\chi}\big\}$ from Corollary \ref{funspP1C} sum to 1. For $\chi=\mathbf{1}$ we get precisely 1 (since $t_{\mathbf{1}}=0$ and the set $A_{\sigma,\mathbf{1}}$ from Equation \eqref{ACchi} is empty for every $\sigma$), and the corresponding space $L(-\Delta)_{\mathbf{1}}$ consists of the constant functions. Hence non-specialty is equivalent to the condition $\sum_{\sigma \neq Id_{X}}|A_{\sigma,\chi}| \leq t_{\chi}-1$ for every $\chi\neq\mathbf{1}$. As Equation \eqref{ACchi} presents $A_{\sigma,\chi}$ as the appropriate disjoint union for every $\sigma \neq Id_{X}$ and $\chi\neq\mathbf{1}$, the left hand side here coincides with the left hand side of the desired cardinality condition associated with $\chi$.

Now, the sum over $\chi$ of the right hand side is $1-n+\sum_{\sigma \neq Id_{X}}\sum_{\chi\neq\mathbf{1}}\frac{r_{\sigma}u_{\chi,\sigma}}{o(\sigma)}$ (using Equation \eqref{uchiCtchi} and since $-1$ from each $\mathbf{1}\neq\chi\in\widehat{G}$ sum to $1-n$), and in the latter sum we can also include the character $\chi=\mathbf{1}$ since Equation \eqref{uchiCtchi} immediately implies that $u_{\mathbf{1},\sigma}=0$ for every $\sigma$. Fix now $Id_{X}\neq\sigma \in G$, and since every power of $\zeta_{o(\sigma)}$ is attained as $\chi(\sigma)$ for $\frac{n}{o(\sigma)}$ characters $\chi$, we have the equality $u_{\chi,\sigma}=u$ precisely $\frac{n}{o(\sigma)}$ times for any $0 \leq u<o(\sigma)$. The inner sum over $\chi$ is therefore $\frac{r_{\sigma}}{o(\sigma)}$ times $\frac{n}{o(\sigma)}\cdot\frac{o(\sigma)(o(\sigma)-1)}{2}$, so that the total sum $1-n+\sum_{\sigma \neq Id_{X}}\frac{nr_{\sigma}}{2o(\sigma)}\big(o(\sigma)-1\big)$ coincides with the expression for $g_{X}$ obtained from Corollary \ref{Galgen} in the case of Abelian $G$ and $g_{S}=0$. On the other hand, on the left hand side the cardinality of each set $B_{\sigma,i}$ is counted precisely once for each $\chi\in\widehat{G}$ with $u_{\overline{\chi},\sigma}>i$ (the omission of $\chi=\mathbf{1}$ does not affect this assertion, since we have seen that $0=u_{\mathbf{1},\sigma} \leq i$ for any $i$ and any $\sigma$), and the number of such characters is $\frac{n}{o(\sigma)}$ times the cardinality $o(\sigma)-1-i$ of the set of numbers $0 \leq u<o(\sigma)$ that are larger than $i$. The resulting sum is thus the expression for $\deg\Delta$ in Equation \eqref{degDelta} (recall that $p=0$), implying that the sum of our inequalities becomes $\deg\Delta \leq g_{X}$. By our assumption on $\deg\Delta$, we are now in a situation where summing non-strict inequalities, all pointing to the same direction, yields an equality. As this occurs if and only if all the inequalities were equalities to begin with, all of our inequalities are equalities as desired. This completes the proof of the theorem.
\end{proof}

\smallskip

We also prove the following result, involving non-positive divisors, generalizing those divisors from \cite{[FZ]} and \cite{[Z]} that in our additive notation are written as positive divisors of degree $g_{X}+n-1$ minus $\sum_{\upsilon=1}^{n}\infty_{\upsilon}$ (these references use a multiplicative notation). These are the divisors that turn out to be the most appropriate for stating and proving Thomae formulae in \cite{[KZ]}.
\begin{thm}
For a $G$-invariant divisor $\Xi$ on an Abelian cover $X$ of $\mathbb{P}^{1}(\mathbb{C})$, with $\deg\Xi=g_{X}-1$, the condition $r(-\Xi)=0$ holds if and only if $\Xi$ is invariantly linearly equivalent to a normalized divisor $\Delta$ whose presentation in Equation \eqref{normDelP1C} is with $p=-1$ and such that for every $\chi\in\widehat{G}$ the sum $\sum_{\sigma \neq Id_{X}}\sum_{i=0}^{u_{\chi,\sigma}-1}|B_{\sigma,i}|$ yields the number $t_{\chi}$ from Equation \eqref{uchiCtchi}. \label{eqnointg-1}
\end{thm}

\begin{proof}
Since linear equivalence (hence in particular invariant linear equivalence) leaves the dimension $r(-\Xi)$ invariant, Proposition \ref{norminvdiv} allows us to restrict attention to the unique normalized divisor $\Delta$ in the invariant linear equivalence class of $\Xi$. We write this divisor as in Equation \eqref{normDelP1C}, and note that since $\Delta$ cannot be positive if $r(-\Delta)=0$ (for otherwise $L(-\Delta)$ would contain the constant functions), the parameter $p$ must be negative. It would thus be convenient to decompose the expression from Equation \eqref{normDelP1C}, and write $\Delta$ as $\Sigma-|p|\sum_{\upsilon=1}^{n}\infty_{\upsilon}$ where $\Sigma$ is the positive divisor obtained from the sets $B_{\sigma,i}$ (again with an index $\sigma$ and not $C$), without the poles of $f^{*}z$. It now easily follows (e.g., from Equation \eqref{degDelta} and the known degree $g_{X}-1$ of $\Delta$) that $\deg\Sigma=g_{X}+n|p|-1$.

At this point we follow the proof of Theorem \ref{intnonsp}. We apply Corollary \ref{totdimP1C} to see that the equality $r(-\Delta)=0$ is equivalent to the vanishing of the number $\max\big\{0,1-|p|+\sum_{\sigma \neq Id_{X}}|A_{\sigma,\chi}|-t_{\chi}\big\}$ from Corollary \ref{funspP1C} for every $\chi\in\widehat{G}$, or equivalently to the inequality $\sum_{\sigma \neq Id_{X}}|A_{\sigma,\chi}| \geq t_{\chi}+|p|-1$. Here we take the sum over all $\chi\in\widehat{G}$ (now including $\chi=\mathbf{1}$), where the left hand side yields $\deg\Sigma$ as in the proof of Theorem \ref{intnonsp} (the fact that now $\chi=\mathbf{1}$ is included makes no difference on this side, since we have seen that $A_{\sigma,\mathbf{1}}$ is empty for every $\sigma \neq Id_{X}$). On the right hand side we get the sum of $\sum_{\chi\in\widehat{G}}|p|=n|p|$, the expression $\sum_{\mathbf{1}\neq\chi\in\widehat{G}}(t_{\chi}-1)$ which was evaluated as $g_{X}$ in the proof of Theorem \ref{intnonsp}, and the extra term $t_{\mathbf{1}}-1=-1$. Altogether we get $\deg\Sigma \leq g_{X}+n|p|-1$, which was assumed to be an equality, implying again that all the inequalities from above must be equalities. In particular, for $\chi=\mathbf{1}$ we have the equality between $\sum_{Id_{X}\neq\sigma \subseteq G}|A_{\sigma,\chi}|=0$ and $t_{\mathbf{1}}+|p|-1=|p|-1$, so that $p=-1$ (a fact that also follows from the inequality $0=\sum_{\sigma \neq Id_{X}}|A_{\sigma,\mathbf{1}}| \geq t_{\mathbf{1}}+|p|-1$, with $t_{\mathbf{1}}=0$, since $|p|\geq1$). Substituting this value into the remaining equalities yields the desired cardinality conditions. This proves the theorem.
\end{proof}

\section{Differentials and $q$-Differentials \label{Diffs}}

This section begins the investigation of the representations of the Galois group $G$ of a Galois cover $f:X \to S$ on differentials on $X$, and more generally on $q$-differentials (see \cite{[FK]}). We call the divisor of any $q$-differential on $X$ \emph{$q$-canonical}. Because all the $q$-canonical divisors are linearly equivalent to one another (for the same $q$), we shall consider a simple $q$-differential on $X$, and deduce the assertions about all the other ones using the results for functions above. We denote the space of \emph{meromorphic} $q$-differentials on $X$ by $\Omega^{q}(X)$. Following \cite{[FK]} we define, for a divisor $\Delta$, the vector space $\Omega^{q}(\Delta)$ to be the one consisting of 0 and of those non-zero $q$-differentials $\omega\in\Omega^{q}(X)$ that satisfy $\mathrm{ord}_{P}\omega \geq v_{P}(\Delta)$ for every $P \in X$. This space has finite dimension, which is denoted by $i^{q}(\Delta)$. In case $q=1$ the superscript $q$ will simply be omitted. As a subgroup $G \subseteq Aut(X)$ operates linearly also on $\Omega^{q}(X)$, the spaces $\Omega^{q}(X)_{\chi}$ for $\chi\in\widehat{G}$ are defined similarly to $\mathbb{C}(X)_{\chi}$. The intersection with $\Omega^{q}(\Delta)$ produces the space denoted by $\Omega^{q}(\Delta)_{\chi}$, of (finite) dimension $i^{q}_{\chi}(\Delta)$.

The analysis of $G$-invariant $q$-canonical divisors is based on the following lemma. Let $f:X \to S$ be a general Galois cover, with Galois group $G$, and we denote by $\Omega^{q}(S)$ the space of meromorphic $q$-differentials on $S$.
\begin{lem}
The space $\Omega^{q}(X)_{\chi}$ is non-trivial. Given a point $P \in X$, with $f(P)=\eta \in S$ and $C=\psi(\eta) \subseteq G$, and a $q$-differential $0\neq\omega\in\Omega^{q}(X)_{\mathbf{1}}$, we have $\mathrm{ord}_{P}\omega\equiv-q\big(\mathrm{mod\ }o(C)\big)$. Moreover, if $\Delta$, $\Xi$ and $\rho$ are as in Lemma \ref{actCXrho} then the spaces $\Omega^{q}(\Xi)_{\chi}$ and $\Omega^{q}(\Delta)_{\chi\rho}$ are canonoically isomorphic. \label{dif-1}
\end{lem}

\begin{proof}
It is clear that $\Omega^{q}(X)_{\mathbf{1}}$ is $f^{*}\Omega^{q}(S)$, hence of dimension 1 over $\mathbb{C}(S)$, and via multiplication by a element of $\mathbb{C}(X)_{\chi}$ we obtain the space $\Omega^{q}(X)_{\chi}$ and its non-triviality. Writing our element $\omega\in\Omega^{q}(X)_{\mathbf{1}}$ as the pullback of a $q$-differential on $S$ and expanding it in local coordinates easily proves the assertion about the orders. It is also clear that for any $\Xi\in\mathrm{Div}(X)$ and any element $\chi\in\widehat{G}$, multiplication by $\omega\in\Omega^{q}(X)_{\mathbf{1}}$ takes the space $L(\Xi-\mathrm{div}(\omega))_{\chi}$ isomorphically onto $\Omega^{q}(\Xi)_{\chi}$, and that this map commutes with the isomorphism from Lemma \ref{actCXrho} to yield the desired one (just note that the divisors here are with a positive sign, so that multiplication by $h$ goes in the other direction). This proves the lemma.
\end{proof}

The fact that $f^{*}\Omega^{q}(S)$ covers all of $\Omega^{q}(X)_{\mathbf{1}}$ is related to the fact that all the divisors of elements of the latter space must be \emph{invariantly} linearly equivalent to one another, so that the difference between two such divisors is of the form $\mathrm{div}(f^{*}F)$ for $F\in\mathbb{C}(S)^{\times}$. The latter assertion generalizes as follows.
\begin{lem}
Every $0\neq\omega\in\Omega^{q}(X)$ for which $\mathrm{div}(\omega)$ is $G$-invariant lies in $\Omega^{q}(X)_{\chi}$ for some $\chi\in\widehat{G}$. In addition, for $P$, $\eta$, and $C$ as in Lemma \ref{dif-1} we have $\mathrm{ord}_{P}\omega\equiv-q-u_{\overline{\chi},C}\big(\mathrm{mod\ }o(C)\big)$. If $G$ is Abelian then the direct sum $\bigoplus_{\chi\in\widehat{G}}\Omega^{q}(X)_{\chi}$ is the whole space $\Omega^{q}(X)$, and we have $i^{q}(\Delta)=\sum_{\chi\in\widehat{G}}i^{q}_{\chi}(\Delta)$ wherever $\Delta$ is a $G$-invariant divisor on $X$. \label{chardiff}
\end{lem}

\begin{proof}
The fact that all the $q$-canonical divisors on $X$ are linearly equivalent combines with Lemma \ref{invlineq} to establish the first assertion. Now, if $\omega\in\Omega^{q}(X)_{\chi}$ then write it as $h_{\chi}\cdot\frac{\omega}{h_{\chi}}$ with $\frac{\omega}{h_{\chi}}\in\Omega^{q}(X)_{\mathbf{1}}$, and the second assertion now follows from Proposition \ref{divgens} and Lemmas \ref{tuchichibar} and \ref{dif-1}. The third assertion is proved just like Lemma \ref{decom}. This proves the lemma.
\end{proof}
The case $q=1$ in the last assertion of Lemma \ref{chardiff} generalizes the other part of Proposition 1.2 of \cite{[Z]} to this setting.

\smallskip

We can now construct a generator $\omega_{\chi,q}$ of $\Omega^{q}(X)_{\chi}$ (up to scalars) for every character $\chi\in\widehat{G}$. Given such $q$ and  $\chi$ and a conjugacy class $C \subseteq G$, we write
\begin{equation}
q\big(o(C)-1\big)-u_{\overline{\chi},C}=\alpha_{C,\overline{\chi}}^{q}o(C)+\beta_{C,\overline{\chi}}^{q}\quad\mathrm{with}\quad\alpha_{C,\overline{\chi}}^{q}\in\mathbb{Z}\ \mathrm{and}\ 0\leq\beta_{C,\overline{\chi}}^{q}<o(C). \label{multqdifs}
\end{equation}
Note that $\alpha_{Id_{X},\chi}^{q}=\beta_{Id_{X},\chi}^{q}=0$ for every $q$ and $\chi$, so that these numbers are interesting only for non-trivial classes. Let $s_{\overline{\chi},q}$ be the maximal number $s$ such that $i^{q}\Big(\Upsilon_{\overline{\chi}}+s\nu-\sum_{C \neq Id_{X}}\sum_{j=1}^{r_{C}}\alpha_{C,\overline{\chi}}^{q}\eta_{C,j}\Big)\geq1$. Hence by setting $s=s_{\overline{\chi},q}$ the latter dimension is 1, so let $\varpi_{\chi,q}$ be any non-zero $q$-differential in the associated space. This determines $\varpi_{\chi,q}$ up to scalar multiples, and we get that
\begin{equation}
\mathrm{div}(\varpi_{\chi,q})=\Upsilon_{\overline{\chi}}+s_{\overline{\chi},q}\nu+\widetilde{\Upsilon}_{\chi,q}-\sum_{C \neq Id_{X}}\sum_{j=1}^{r_{C}}\alpha_{C,\overline{\chi}}^{q}\eta_{C,j} \label{varpichiq}
\end{equation}
for some positive divisor $\widetilde{\Upsilon}_{\chi,q}$. Since the degree of any $q$-canonical divisor on $S$ must be $q(2g_{S}-2)$, it follows immediately from Equation \eqref{varpichiq} that
\begin{equation}
\deg\widetilde{\Upsilon}_{\chi,q}=q(2g_{S}-2)-s_{\overline{\chi},q}-\deg\Upsilon_{\overline{\chi}}+\sum_{C \neq Id_{X}}r_{C}\alpha_{C,\overline{\chi}}^{q}. \label{degUpschiq}
\end{equation}
As with the spaces, when $q=1$ the index will be omitted (so that $\varpi_{\chi}$ and $\widetilde{\Upsilon}_{\chi}$ stand for $\varpi_{\chi,1}$ and $\widetilde{\Upsilon}_{\chi,1}$ respectively).
\begin{prop}
The space $\Omega^{q}(X)_{\chi}$ with $\chi\in\widehat{G}$ is spanned over $\mathbb{C}(S)$ by the $q$-differential $\omega_{\chi,q}=f^{*}(\varpi_{\chi,q})/h_{\overline{\chi}}$, where $h_{\overline{\chi}}$ is the function from Proposition \ref{divgens}. This $q$-differential has the normalized $q$-canonical divisor \[f^{*}\!\bigg[\widetilde{\Upsilon}_{\chi,q}+\!\bigg(q(2g_{S}-2)+t_{\overline{\chi}}-\deg\widetilde{\Upsilon}_{\chi,q}+\!\!\sum_{C \neq Id_{X}}\!\!r_{C}\alpha_{C,\overline{\chi}}^{q}\bigg)\!\cdot\nu\bigg]\!+\!\sum_{C \neq Id_{X}}\sum_{j=1}^{r_{C}}\sum_{\upsilon=1}^{\frac{n}{o(C)}}\beta_{C,\overline{\chi}}^{q}P_{C,j,\upsilon},\] where $t_{\overline{\chi}}$ is given in Equation \eqref{uchiCtchi}. In particular, the (normalized) divisor of the generator $\omega_{\chi}=\omega_{\chi,1}$ of
$\Omega(X)_{\chi}$ is \[f^{*}\big(\widetilde{\Upsilon}_{\chi}+(2g_{S}-2+t_{\overline{\chi}}-\deg\widetilde{\Upsilon}_{\chi})\cdot\nu\big)+\sum_{C \neq Id_{X}}\sum_{j=1}^{r_{C}}\sum_{\upsilon=1}^{\frac{n}{o(C)}}\big(o(C)-1-u_{\overline{\chi},C}\big)P_{C,j,\upsilon}.\] \label{normdif}
\end{prop}

\begin{proof}
The proof of Lemma \ref{chardiff} shows that $\omega_{\chi,q}$ is a non-zero element of $\Omega^{q}(X)_{\chi}$, hence it spans this space over $\mathbb{C}(S)$. Equation \eqref{varpichiq} and the proof of Lemma \ref{dif-1} now imply that $\mathrm{div}(f^{*}\varpi_{\chi,q})$ is $f^{*}\big(\Upsilon_{\overline{\chi}}+s_{\overline{\chi},q}\nu+\widetilde{\Upsilon}_{\chi,q}-\sum_{C \neq Id_{X}}\sum_{j=1}^{r_{C}}\alpha_{C,\overline{\chi}}^{q}\eta_{C,j}\big)$ plus the terms $q(o(C)-1)P_{C,j,\upsilon}$ for every $C$, $j$, and $\upsilon$. When subtracting the divisor $\mathrm{div}(h_{\overline{\chi}})$ given in Proposition \ref{divgens}, the two instances of $f^{*}\Upsilon_{\overline{\chi}}$ cancel, the terms involving multiples of $P_{C,j,\upsilon}$ (including the multiple of $f^{*}\eta_{C,j}$) directly produce $\beta_{C,\overline{\chi}}^{q}P_{C,j,\upsilon}$ by Equation \eqref{multqdifs}, and the remaining terms combine to the expression $f^{*}\big(\widetilde{\Upsilon}_{\chi,q}+(s_{\overline{\chi},q}+\deg\Upsilon_{\overline{\chi}}+t_{\overline{\chi}})\cdot\nu\big)$. Equation \eqref{degUpschiq} thus shows that the multiplicity of $f^{*}\nu$ here is indeed the asserted one, and we need to verify that this divisor is normalized. Proposition \ref{norminvdiv} then reduces us to verifying that $\widetilde{\Upsilon}_{\chi,q}$ is a non-special divisor on $S\setminus\{\nu\}$. Positivity is clear, and the fact that $\nu$ is not contained in the support of that divisor follows from the maximality of $s_{\overline{\chi},q}$ (since otherwise $\varpi_{\chi,q}$ would be in the space with $(s_{\overline{\chi},q}+1)\nu$). But this maximality also proves that $r(-\widetilde{\Upsilon}_{\chi,q})=1$ in the same manner: The proof of Lemma \ref{divSinfty} shows that otherwise we could replace $\widetilde{\Upsilon}_{\chi,q}$ by a linearly equivalent integral divisor containing $\nu$ in its support, thus producing a non-zero element of the space with $(s_{\overline{\chi},q}+1)\nu$. Hence $\mathrm{div}(\omega_{\chi,q})$ is normalized, establishing the result for general $q$. The assertion for $q=1$ easily follows since it is clear from Equation \eqref{multqdifs} that $\alpha_{C,\overline{\chi}}^{1}=0$ and $\beta_{C,\overline{\chi}}^{1}=o(C)-1-u_{\overline{\chi},C}$ for every $\chi$ and $C$. This completes the proof of the proposition.
\end{proof}

\smallskip

We now turn to proving the following analogue of Proposition \ref{L1/Deltachi} for differentials. We remark that one can use the same argument for proving a similar assertion for $q$-differentials with arbitrary $q$. However, as working in that generality makes the notation much more complicated, we content ourselves with the case $q=1$ here.
\begin{prop}
Take a normalized $G$-invariant divisor $\Delta$ written as in Equation \eqref{normDelgen} and a character $\chi\in\widehat{G}$, and let $t_{\chi}$ be as above and $A_{C,\chi}$ be the sets from Equation \eqref{ACchi}. Then $i_{\chi}(\Delta)$ coincides with \[i\bigg(\Upsilon-\Upsilon_{\chi}+(p+t_{\chi}-\deg\Upsilon+\deg\Upsilon_{\chi})\cdot\nu-\sum_{C \subseteq G\setminus\ker\chi}\sum_{j \not\in A_{C,\overline{\chi}}}\eta_{C,j}\bigg).\] \label{Omega1/Deltachi}
\end{prop}

\begin{proof}
As in the proof of Proposition \ref{L1/Deltachi}, we apply the arguments from Lemmas \ref{actCXrho} and \ref{dif-1} in order to show that $i_{\chi}(\Delta)$ equals $i_{\mathbf{1}}(\Delta-\mathrm{div}(h_{\chi}))$. Denote the asserted argument of $i$ here by $\Gamma$, and set $\Xi$ to be $\Delta-\mathrm{div}(h_{\chi})$. Proposition \ref{divgens} and Equation \eqref{normDelgen} evaluate $\Xi$ as $f^{*}\Gamma$ plus some multiples of the points $P_{C,j,\upsilon}$. If $\eta_{C,j} \in B_{C,i}$ then this multiplicity is $o(C)-1-i-u_{\chi,C}$, and in case $j \not\in A_{C,\overline{\chi}}$ we have to add $o(C)$. Note that by Equation \eqref{ACchi} we add $o(C)$ if and only if $\chi(C)\neq1$ and $i \geq u_{\overline{\chi},C}$, so that by Lemma \ref{tuchichibar} this means $i \geq o(C)-u_{\chi,C}$. This means (when one checks the cases $u_{\chi,C}=0$ and $u_{\chi,C}>0$ separately) that we add $o(C)$ precisely when the multiplicity in question is negative, so that the total multiplicity of $P_{C,j,\upsilon}$ that we investigate is between 0 and $o(C)-1$. Hence we are in the setting where $\Xi\in\mathrm{Div}(X)$ is of the form $f^{*}\Gamma+\sum_{C \neq Id_{X}}\sum_{j=1}^{r_{C}}\sum_{\upsilon=1}^{n/o(C)}l_{C,j}P_{C,j,\upsilon}$ with $0 \leq l_{C,j} \leq o(C)-1$. But Lemma \ref{dif-1} combines with the proof of Lemma \ref{nonspres} to show that in this case an element $f^{*}\omega\in\Omega(X)_{\mathbf{1}}$ with $\omega\in\Omega(S)$ lies in $\Omega(\Xi)$ (hence in $\Omega(\Xi)_{\mathbf{1}}$) if and only if $\omega\in\Omega(\Gamma)$. This proves the proposition.
\end{proof}
For arbitrary $q$-differentials, the proof of Proposition \ref{Omega1/Deltachi} identifies $i^{q}_{\chi}(\Delta)$ (or $i^{q}_{\mathbf{1}}(\Delta-\mathrm{div}(h_{\chi}))$) with the number $i^{q}(\Sigma)$ for a divisor $\Sigma\in\mathrm{Div}(S)$ for which $v_{\eta_{C,j}}(\Sigma)$ is the sum of $v_{\eta_{C,j}}(\Upsilon)$ for $\Upsilon$ from Equation \eqref{normDelgen}, $\alpha_{C,\overline{\chi}}^{q}$ from Equation \eqref{multqdifs}, and a number from $\{0,1\}$ that depends on the relation between the index $i$ for which $\eta_{C,j} \in B_{C,i}$ and the parameter $\beta_{C,\overline{\chi}}^{q}$ from Equation \eqref{multqdifs}. We do not write the general formula, since it becomes cumbersome without additional, ad-hoc notation extending Equation \eqref{ACchi}. An alternative point of view to the case $q=1$ from Proposition \ref{Omega1/Deltachi} is acquired by showing, via evaluating the required divisors, that the space $\Omega(\Delta)_{\chi}$ itself consists of those differentials that are obtained from $\omega_{\chi}$ via multiplication by functions from
\begin{equation}
f^{*}L\bigg(\Upsilon-\widetilde{\Upsilon}_{\chi}-(2g_{S}-2+t_{\overline{\chi}}-p+\deg\Upsilon-\deg\widetilde{\Upsilon}_{\chi})\cdot\nu+\sum_{C \neq Id_{X}}\sum_{j \in A_{C,\overline{\chi}}}\eta_{C,j}\bigg). \label{OmegaDeltaChi}
\end{equation}
In particular, we get that $\Omega(\Delta)_{\mathbf{1}}$ is just $f^{*}\Omega\big(\Upsilon+(p-\deg\Upsilon)\cdot\nu\big)$. Equation \eqref{OmegaDeltaChi} also extends to general $q$, but with a much more complicated formula, which does not simplify much also for $\chi=\mathbf{1}$.

\smallskip

As in the case of functions, Proposition \ref{Omega1/Deltachi} and Equation \eqref{OmegaDeltaChi} produce more complete results also about differentials in case the Galois group $G$ of the cover $f:X \to S$ is Abelian.
\begin{cor}
If $G$ is Abelian and $\Delta$ is a $G$-invariant divisor with the usual parameters then
$i(\Delta)$ is \[\sum_{\chi\in\widehat{G}}i\bigg(\Upsilon-\Upsilon_{\chi}+(p+t_{\chi}-\deg\Upsilon+\deg\Upsilon_{\chi})\cdot\nu-\sum_{C \subseteq G\setminus\ker\chi}\sum_{j \not\in A_{C,\overline{\chi}}}\eta_{C,j}\bigg),\]
which also equals \[\sum_{\chi\in\widehat{G}}r\bigg(\Upsilon-\widetilde{\Upsilon}_{\chi}-(2g_{S}-2+t_{\overline{\chi}}-p+\deg\Upsilon-\deg\widetilde{\Upsilon}_{\chi})\cdot\nu+\sum_{C \neq Id_{X}}\sum_{j \in A_{C,\overline{\chi}}}\eta_{C,j}\bigg).\] \label{diffGab}
\end{cor}

\begin{proof}
The first expression is obtained by substituting the expressions from Proposition \ref{Omega1/Deltachi} into the case $q=1$ of the last assertion of Lemma \ref{chardiff}. For the second expression we do the same using the dimensions of the spaces from Equation \eqref{OmegaDeltaChi}. This proves the corollary.
\end{proof}

In the case where $S=\mathbb{P}^{1}(\mathbb{C})$ and $\nu=\infty$ we would expect the canonical choice $dz$ for a differential on $S$, as well as its $f^{*}$-image on $X$, to play a role (the latter is the differential of the meromorphic function $f^{*}z$, but we shall stick with the notation $f^{*}dz$). Indeed, in this case the results of Propositions \ref{normdif} and \ref{Omega1/Deltachi} and Corollary \ref{diffGab} take the following form.
\begin{cor}
If $f:X\to\mathbb{P}^{1}(\mathbb{C})$ is a Galois cover with Galois group $G$ and $\chi\in\widehat{G}$ then the normalized generator $\omega_{\chi,q}$ for $\Omega^{q}(X)_{\chi}$ over $\mathbb{C}(S)=\mathbb{C}(z)$ has the divisor
\[\sum_{C \neq Id_{X}}\sum_{j=1}^{r_{C}}\sum_{\upsilon=1}^{\frac{n}{o(C)}}\beta_{C,\overline{\chi}}^{q}P_{C,j,\upsilon}+
\bigg(t_{\overline{\chi}}-2q+\sum_{C \neq Id_{X}}r_{C}\alpha_{C,\overline{\chi}}^{q}\bigg)\sum_{\upsilon=1}^{n}\infty_{\upsilon}.\]
For $q=1$ we have that $\omega_{\chi}=\omega_{\chi,1}$ is $f^{*}dz/h_{\overline{\chi}}$, with divisor \[\sum_{C \neq Id_{X}}\sum_{j=1}^{r_{C}}\sum_{\upsilon=1}^{\frac{n}{o(C)}}\big(o(C)-1-u_{\overline{\chi},C}\big)P_{C,j,\upsilon}+(t_{\overline{\chi}}-2)\sum_{\upsilon=1}^{n}\infty_{\upsilon}.\] If $\Delta$ is a divisor as in Proposition \ref{Omega1/Deltachi} then the dimension $i_{\chi}(\Delta)$ is the number $\max\{0,t_{\overline{\chi}}-\sum_{C \neq Id_{X}}|A_{C,\overline{\chi}}|-p-1\}$ in the notation of that proposition, and if $G$ is Abelian then $i(\Delta)$ equals $\sum_{\chi\in\widehat{G}}\max\{0,t_{\overline{\chi}}-\sum_{C \neq Id_{X}}|A_{C,\overline{\chi}}|-p-1\}$. \label{P1Cdiff}
\end{cor}

\begin{proof}
For the first assertion we argue as in Proposition \ref{normdif}, recalling from Corollary \ref{P1Cnorm} that $\Upsilon_{\overline{\chi}}=\widetilde{\Upsilon}_{\chi}=0_{S}$ when $g_{S}=0$, and substituting $f^{*}\nu=f^{*}\infty$ as $\sum_{\upsilon=1}^{n}\infty_{\upsilon}$. The formula for $\mathrm{div}(\omega_{\chi})$ follows in the same way (by what we know about $\alpha_{C,\overline{\chi}}^{1}$ and $\beta_{C,\overline{\chi}}^{1}$), or alternatively by combining Corollary \ref{P1Cnorm} and the proof of Lemma \ref{dif-1}. Proving the formula for that differential itself amounts to showing that $\varpi_{\chi}=\varpi_{\chi,1}=dz$ for every $\chi\in\widehat{G}$. But this follows immediately from the definition of $\varpi_{\chi}$, since $\Upsilon_{\overline{\chi}}=0_{S}$, $\alpha_{C,\overline{\chi}}^{1}=0$ in Equation \eqref{multqdifs}, and the space $\Omega(-2\cdot\infty)$ is 1-dimensional on $\mathbb{P}^{1}(\mathbb{C})$, spanned by $dz$. Since $\Omega(-1\cdot\infty)=\{0\}$, this also shows that $s_{\overline{\chi}}=-2$ for every such $\chi$.

For the third assertion we follow the proof of Proposition \ref{Omega1/Deltachi}, combined with Corollary \ref{P1Cnorm} once more. Once again we substitute $\Upsilon=\Upsilon_{\chi}=0_{S}$, and since we can replace each point $\eta_{C,j}$ by the linearly equivalent point $\infty$, the dimension in question coincides with $i\big((p+t_{\chi})\cdot\infty^{}-\sum_{C \subseteq G\setminus\ker\chi}\sum_{j \not\in A_{C,\overline{\chi}}}\infty\big)$. We now observe that if we subtract $\sum_{C \neq Id_{X}}|A_{C,\overline{\chi}}|$ from the total multiplicity of $\infty$ here then the terms with $\chi(C)=1$ have no effect (since we have seen that $A_{C,\overline{\chi}}$ is empty then), and the remaining terms involving the cardinalities $|A_{C,\overline{\chi}}|$ combine to $\sum_{C \subseteq G\setminus\ker\chi}r_{C}$. But Lemma \ref{tuchichibar} shows that the latter expression is $t_{\chi}+t_{\overline{\chi}}$, so that we are interested in $i\big[\big(p+\sum_{C \neq Id_{X}}|A_{C,\overline{\chi}}|-t_{\overline{\chi}}\big)\cdot\infty\big]$. This dimension can be evaluated either through dividing the relevant differentials by $dz$ or via the Riemann--Roch Theorem, using our knowledge about the spaces $L(-d\cdot\infty)$. This establishes the third assertion, and the fourth one then follows, as in Corollary \ref{diffGab}, from the last assertion of Lemma \ref{chardiff}. This proves the corollary.
\end{proof}

Corollary \ref{diffGab} and the last two assertions of Corollary \ref{P1Cdiff} also have extensions to arbitrary $q$, with the same properties as with Proposition \ref{Omega1/Deltachi}. Let us now demonstrate the complications arising in this extension by considering the first assertion of Corollary \ref{P1Cdiff} with $\chi=\mathbf{1}$. The $q$-differential $\varpi_{\mathbf{1},q}$ on $S=\mathbb{P}^{1}(\mathbb{C})$ is no longer $(dz)^{q}$, but $(dz)^{q}$ divided by the terms $(z-\eta_{C,j})^{\alpha_{C,\mathbf{1}}^{q}}$ for each $C$ and $j$, and the coefficient $\alpha_{C,\mathbf{1}}^{q}=\big\lfloor q-\frac{q}{o(C)} \big\rfloor$ from Equation \eqref{multqdifs} vanishes only for $q\in\{0,1\}$. On the other hand, the second assertion of Corollary \ref{P1Cdiff} implies that $\omega_{\mathbf{1}}=\omega_{\mathbf{1},1}$ is just $f^{*}dz$, and its divisor is already evaluated implicitly in the proof of that corollary. We also remark that taking $p\geq0$ in the latter two statements of Corollary \ref{P1Cdiff} generalizes Corollary 1.4 of \cite{[Z]} to this setting.

In fact, combining the argument proving Equation \eqref{OmegaDeltaChi} with Corollary \ref{P1Cnorm} and the proof of Corollary \ref{P1Cdiff} proves the following result: Set $\tilde{d}(\chi)$ to be the number $\max\{-1,t_{\overline{\chi}}-\sum_{C \neq Id_{X}}|A_{C,\overline{\chi}}|-p-2\}$, and then the space $\Omega(\Delta)_{\chi}$ is just $\prod_{C \neq Id_{X}}\prod_{\{j|\lambda_{C,j} \in A_{C,\overline{\chi}}\}}(f^{*}z-\eta_{C,j})\cdot\mathcal{P}_{\leq\tilde{d}(\chi)}(f^{*}z)\omega_{\chi}$ when $S=\mathbb{P}^{1}(\mathbb{C})$. In particular, the space $\Omega(\Delta)_{\mathbf{1}}$ is then $\mathcal{P}_{\leq\tilde{d}(\mathbf{1})}(z)dz$ with $\tilde{d}(\mathbf{1})=\max\{-1,-p-2\}$, of dimension $\max\{0,-p-1\}$. As non-special positive divisors $\Delta$ of degree $g_{X}$ and non-positive divisors $\Delta$ of degree $g_{X}-1$ that are not linearly equivalent to any positive divisor are both described by the condition $i(\Delta)=0$, the case $q=1$ of Corollary \ref{P1Cdiff} can be used to give an alternative proof for Theorems \ref{intnonsp} and \ref{eqnointg-1}. The fact that only $p=-1$ has to be considered in such a proof for Theorem \ref{eqnointg-1}, visible from the formula $\max\{0,-p-1\}$ for $i_{\mathbf{1}}(\Delta)$ and the negativity of $p$ from $r(-\Delta)=0$, corresponds to the fact that when $p\leq-2$ the differential $f^{*}dz$ lies in $\Omega(\Delta)_{\mathbf{1}}\subseteq\Omega(\Delta)$.

\smallskip

In the case where $X$ is a $Z_{n}$ curve, the elements $\sigma \in G$ are the powers $\tau^{\alpha}$ of $\tau$ with $\alpha\in\mathbb{Z}/n\mathbb{Z}$, and elements $\chi\in\widehat{G}$ are all powers $\phi^{-k}$ of the generator $\phi$ of $\widehat{G}$ sending $\tau$ to $\zeta_{n}$. The associated differential $\omega_{\chi}=\omega_{\chi,1}$ is the one denoted by $\omega_{k}$ in \cite{[Z]} (this is true for $\chi=\phi^{-k}$ since $w^{k}$ is in the denominator in that reference), and the combination $\alpha k-ns_{\alpha,k}$ appearing with a negative sign in Equation (2) of \cite{[Z]} is just our $u_{\overline{\chi},\sigma}$ with these $\chi$ and $\sigma$. Since in the fully ramified case considered in \cite{[Z]} all the elements $\sigma$ for which $r_{\sigma}$ is positive are powers of $\tau$ with exponents $\alpha\in\mathbb{Z}/n\mathbb{Z}$ that are co-prime to $n$, the order $o(\sigma)$ is always $n$. Therefore the second assertion of Corollary \ref{P1Cdiff} extends the validity of Equation (2) of \cite{[Z]} to this much more general setting. We note that the evaluations from the proof of Theorems \ref{intnonsp} and \ref{eqnointg-1} imply that $\mathrm{div}(\omega_{\chi})$ has the required degree $2g_{X}-2$ for any $\chi\in\widehat{G}$ wherever $X$ is an Abelian cover of $\mathbb{P}^{1}(\mathbb{C})$.

\section{Representations on Spaces of $q$-Differentials \label{RepQDiff}}

This section applies the results of the previous section for ($G$-invariant) divisors on $X$ that are obtained as $f^{*}$-images of divisors on $S$. For \emph{positive} such divisors we obtain a generalization of the Eichler trace formula from \cite{[FK]}, as well as of the Chevalley--Weil formula. We begin with a useful identity involving the numbers from Equations \eqref{uchiCtchi} and \eqref{multqdifs}.
\begin{lem}
For any $q\in\mathbb{Z}$, conjugacy class $C$ in $G$, and character $\chi\in\widehat{G}$, the sum $\frac{u_{\overline{\chi},C}}{o(C)}+\alpha_{C,\overline{\chi}}^{q}$ equals $(q-1)\big(1-\frac{1}{o(C)}\big)+\big\{\frac{q-1-u_{\chi,C}}{o(C)}\big\}$. \label{altexp}
\end{lem}

\begin{proof}
Recall from Equation \eqref{multqdifs} that $\alpha_{C,\overline{\chi}}^{q}$ is $\big\lfloor q-\frac{q+u_{\overline{\chi},C}}{o(C)} \big\rfloor$, which we write as $q-\frac{q+u_{\overline{\chi},C}}{o(C)}-\big\{q-\frac{q+u_{\overline{\chi},C}}{o(C)}\big\}$. After adding $\frac{u_{\overline{\chi},C}}{o(C)}$ the expression outside the fractional part becomes just $q\big(1-\frac{1}{o(C)}\big)$, and inside the fractional part we can omit the integer $q$ and replace $-\frac{u_{\overline{\chi},C}}{o(C)}$ by $\frac{u_{\chi,C}}{o(C)}$ (see Lemma \ref{tuchichibar}). It therefore suffices to show that by adding $\big\{\frac{q-1-u_{\chi,C}}{o(C)}\big\}$ to the resulting fractional part $\big\{\frac{u_{\chi,C}-q}{o(C)}\big\}$ we get $1-\frac{1}{o(C)}$. But these two fractional parts are of the form $\frac{a}{o(C)}$ and $\frac{b}{o(C)}$ for some integers $a$ and $b$, with both integers lying between 0 and $o(C)-1$ (the quotients are fractional parts), and their sum must be in $-\frac{1}{o(C)}+\mathbb{Z}$ since the sum of the arguments of these fractional parts is $-\frac{1}{o(C)}$. As this can only happen when the two fractional parts sum to $1-\frac{1}{o(C)}$, this proves the lemma.
\end{proof}

We can now present the decompositions of the spaces of differentials associated with $f^{*}$-images of divisors on $S$. We shall later be focusing on parameters $g_{X}$ and $q$ satisfying the \emph{positivity condition} $(g_{X}-1)(q-1)\geq0$, or explicitly either $g_{X}\geq2$ and $q\geq1$, or $g_{X}=1$ and $q$ is arbitrary, or $g_{X}=0$ and $q\leq1$. In case $q\neq1$ and $g_{X}\neq1$ (i.e., if the inequality in the positivity condition is strict), we say that $g_{X}$ and $q$ satisfy the \emph{strict positivity condition}.
\begin{prop}
For any $\chi\in\widehat{G}$, any $q\in\mathbb{Z}$, and any $\Gamma\in\mathrm{Div}(S)$ we have \[\Omega^{q}(-f^{*}\Gamma)_{\chi}=f^{*}L\bigg[\!\bigg(\!\!\deg\widetilde{\Upsilon}_{\chi,q}-q(2g_{S}-2)-t_{\overline{\chi}}-\sum_{C \neq Id_{X}}r_{C}\alpha_{C,\overline{\chi}}^{q}\bigg)\cdot\nu-\Gamma-\widetilde{\Upsilon}_{\chi,q}\bigg]\!\cdot\omega_{\chi,q},\] where $t_{\overline{\chi}}$, $\alpha_{C,\overline{\chi}}^{q}$, and $\omega_{\chi,q}$ are defined in Equation \eqref{uchiCtchi}, Equation \eqref{multqdifs}, and Proposition \ref{normdif} respectively. If $\Gamma\geq0_{S}$ and $g_{X}$ and $q$ satisfy the positivity condition, then the dimension $i^{q}_{\chi}(-f^{*}\Gamma)$ of the latter space is \[(2q-1)(g_{S}-1)+\deg\Gamma+\sum_{C \neq Id_{X}}r_{C}\big[(q-1)\big(1-\tfrac{1}{o(C)}\big)+\big\{\tfrac{q-1-u_{\chi,C}}{o(C)}\big\}\big]+\delta\] for some $\delta\in\{0,1\}$. Moreover, we have $\delta=0$ if $\Gamma\neq0_{S}$ or if $g_{X}$ and $q$ satisfy the strict positivity condition. On the other hand, when $q=1$ and $\Gamma=0_{S}$ we have $\delta=\delta_{\chi,\mathbf{1}}$. Finally, if $G$ is Abelian then $\Omega^{q}(-f^{*}\Gamma)$ is the direct sum of all these spaces, and in case $\Gamma$ is positive and $g_{X}$ and $q$ satisfy the positivity condition, the dimension $i^{q}(-f^{*}\Gamma)$ is the sum of the asserted dimensions. \label{dimchardif}
\end{prop}

\begin{proof}
For the first assertion we follow the proof of Proposition \ref{Omega1/Deltachi} and Equation \eqref{OmegaDeltaChi}, which we can do without additional notation since all the sets $A_{C,\chi}$ from Equation \eqref{ACchi} (as well as their generalization for arbitrary $q$) are empty. Indeed, the (normalized) divisor of $\omega_{\chi,q}$ from Proposition \ref{normdif} is presented as the $f^{*}$-image of some divisor $\Sigma$ on $S$ plus branch points with small multiplicities, and therefore an element of $\Omega^{q}(X)_{\chi}$, which must be of the form $f^{*}F\cdot\omega_{\chi,q}$ for some $F\in\mathbb{C}(S)$, lies in $\Omega^{q}(-f^{*}\Gamma)_{\chi}$ if and only if $F \in L(-\Sigma-\Gamma)$. By substituting the explicit value of $\Sigma$ from Proposition \ref{normdif}, the first assertion follows.

In order to evaluate the dimension $i^{q}_{\chi}(-f^{*}\Gamma)$, which we have determined as $r(-\Sigma-\Gamma)$, we apply the Riemann--Roch Theorem. When $\deg(\Sigma+\Gamma)\geq2g_{S}-2$ the latter dimension is $1-g_{S}+\deg(\Sigma+\Gamma)$, except when $\Sigma+\Gamma$ is a canonical divisor (of degree precisely $2g_{S}-2$), where we have to add 1 to that value. We now assume that $\Gamma\geq0_{S}$ and $g_{X}$ and $q$ satisfy the positivity condition, and in order to be able to apply the above argument we compare $\deg(\Sigma+\Gamma)$ with $2g_{S}-2$. By expanding $t_{\overline{\chi}}$ as in its definition in Equation \eqref{uchiCtchi} and applying Lemma \ref{altexp}, we find that $\deg(\Sigma+\Gamma)-(2g_{S}-2)$ is
\[(q-1)\Bigg[2g_{S}-2+\sum_{C \neq Id_{X}}r_{C}\big(1-\tfrac{1}{o(C)}\big)\Bigg]+\sum_{C \neq Id_{X}}r_{C}\big\{\tfrac{q-1-u_{\chi,C}}{o(C)}\big\}+\deg\Gamma.\] But Corollary \ref{Galgen} implies that the first term here is just $\frac{q-1}{n}(2g_{X}-2)$, so that the latter expression is non-negative by our assumption on $\Gamma$, $q$, and $g_{X}$. The formula for $i^{q}_{\chi}(-f^{*}\Gamma)$ is therefore established up to the value of $\delta\in\{0,1\}$. In addition, when $\Gamma\neq0_{S}$ or when the positivity condition is strict we get $\deg(\Sigma+\Gamma)>2g_{S}-2$, so that indeed $\delta=0$ in these cases.

It remains to consider the case with $\Gamma=0_{S}$ and $q=1$, where the remaining sum over $C$ gives just $t_{\overline{\chi}}$ (this also follows from a direct application of Equation \eqref{OmegaDeltaChi}, with $A_{C,\overline{\chi}}=\emptyset$ for every $C$, and with the parameters $p=0$ and $\Upsilon=0_{S}$ associated with $\Delta=f^{*}\Gamma=0_{X}$). Hence $\delta=0$ if $t_{\overline{\chi}}>0$, and when $t_{\overline{\chi}}=0$ we have to examine the canonicity of $\Sigma+\Gamma=\Sigma=\widetilde{\Upsilon}_{\chi}+(2g_{S}-2-\deg\widetilde{\Upsilon}_{\chi})\cdot\nu$. Now, this divisor is canonical if and only if by subtracting $\mathrm{div}(\varpi_{\chi})$ we get a principal divisor, and recalling the value of the latter divisor from Equation \eqref{varpichiq} we find that $\delta=1$ if and only $t_{\overline{\chi}}=0$ and $\Upsilon_{\overline{\chi}}-\deg\Upsilon_{\overline{\chi}}\cdot\nu$ is principal. But this was seen in the last assertion of Proposition \ref{divgens} to be equivalent to $\Upsilon_{\overline{\chi}}=0_{S}$ and to $\chi=\mathbf{1}$, which determines $\delta$ as the required value when $\Gamma=0_{S}$ and $q=1$. The last two assertions follow, as usual, from the last statement in Lemma \ref{chardiff}. This completes the proof of the proposition.
\end{proof}
The result for $\Gamma=0_{S}$ and $q=1$ in Proposition \ref{dimchardif} amounts to $i_{\chi}(0_{X})$ being $g_{S}-1+t_{\overline{\chi}}$ for $\chi\neq\mathbf{1}$ and just $g_{S}$ for $\chi=\mathbf{1}$, the latter value being the dimension of $f^{*}\Omega(0_{S})$. The analysis of the remaining case, with $g_{X}=1$ and $\Gamma=0_{S}$ (and arbitrary $q$), is completed as follows.
\begin{prop}
Take $g_{X}=1$ and $\Gamma=0_{S}$ in Proposition \ref{dimchardif}. Then the parameter $\delta$ is 1 for precisely one character $\chi_{q}$. This character is $\mathbf{1}$ if and only if $q$ is congruent to 1 modulo the least common multiple of all the numbers $o(C)$ for which $r_{C}>0$, a situation which occurs for every $q$ if and only if $g_{S}=1$. \label{gX1}
\end{prop}

\begin{proof}
Since the canonical divisors on $X$ are principal (hence the normalized one is trivial), we know that $i^{q}(0_{X})=1$ for every $q\in\mathbb{Z}$, and $\Omega^{q}(0_{X})$ is spanned by the $q$th power $\omega_{X}^{q}$ of a fixed generator $\omega_{X}$ of $\Omega(0_{X})=\Omega^{1}(0_{X})$. This already determines the structure of $\bigoplus_{q=-\infty}^{\infty}\Omega^{q}(0_{X})$ as a $G$-module: $G$ must act on $\Omega(0_{X})$ via some character $\chi_{X}$, so that $i_{\chi}^{q}(0_{X})$ is 1 if $\chi=\chi_{X}^{q}$ and 0 otherwise. Moreover, since $i^{q}(0_{X})=1$ we trivially have $\Omega^{q}(0_{X})=\bigoplus_{\chi\in\widehat{G}}\Omega^{q}(0_{X})_{\chi}$, also without $G$ being Abelian (compare with the last assertion of Lemma \ref{chardiff}). Note that the required character $\chi_{q}$ will not, in general, be the same as $\chi_{X}^{q}$.

We first consider the case in which $g_{S}=1$, where Corollary \ref{Galgen} implies that $r_{C}=0$ for every $C$ (i.e., the cover $f$ is unramified). Then the number from Proposition \ref{dimchardif} reduces to $\delta$ for every $\chi$, and since $f$ is unramified we have $\mathrm{div}(f^{*}\omega)=f^{*}\big(\mathrm{div}(\omega)\big)$ for any $\omega\in\Omega^{q}(S)$ for every $q$. It follows that if $\omega_{S}$ generates $\Omega(0_{S})$ then we have $\omega_{X}=f^{*}\omega_{S}\in\Omega(0_{X})_{\mathbf{1}}$, and the previous paragraph implies that $\chi_{q}=\chi_{X}^{q}=\mathbf{1}$ for every $q$ (this could alternatively be shown via the equality $\Omega^{q}(0_{X})_{\chi}=L(\deg\widetilde{\Upsilon}_{\chi,q}\cdot\nu-\widetilde{\Upsilon}_{\chi,q})\cdot\omega_{\chi,q}$ and some principality analysis, combined with the last assertion of Proposition \ref{divgens} since $t_{\overline{\chi}}=0$). This proves the assertion when $g_{S}=1$.

We thus turn to the case where $g_{S}=0$, and observe that if the kernel of $\chi\in\widehat{G}$ contains all those classes $C$ for which $r_{C}>0$ then $\chi=\mathbf{1}$ by the last assertion of Corollary \ref{P1Cnorm} (since $t_{\chi}=0$). Now, the proof of Proposition \ref{dimchardif} shows that the expression $\deg(\Sigma+\Gamma)-(2g_{S}-2)=\deg\Sigma+2$ reduces to $\sum_{C \neq Id_{X}}r_{C}\big\{\frac{q-1-u_{\chi,C}}{o(C)}\big\}$, and since every divisor of degree $-2$ on $S=\mathbb{P}^{1}(\mathbb{C})$ is canonical, we deduce that $\delta=1$ precisely when the latter sum over $C$ vanishes. Therefore $i^{q}_{\chi}(0_{X})$ equals $\sum_{C \neq Id_{X}}r_{C}\big\{\frac{q-1-u_{\chi,C}}{o(C)}\big\}-1$ when the sum is non-zero, and just 0 when it vanishes (i.e., when $\delta=1$). To determine when this happens, consider the case $q=1$ proved in Proposition \ref{dimchardif}, and deduce that when $g_{S}=0$ the dimension $i_{\chi}(0_{X})$ is $t_{\overline{\chi}}-1$ for $\chi\neq\mathbf{1}$ and $0=t_{\mathbf{1}}$ for $\chi=\mathbf{1}$. Hence $\sum_{\chi\in\widehat{G}}(t_{\chi}-1)=i(0_{X})-1=0$ and thus $\sum_{\chi\in\widehat{G}}t_{\chi}=|\widehat{G}|=|G^{ab}|=n^{ab}$.

We now claim that the map $\chi\mapsto\chi(C)$ surjects onto $\zeta_{o(C)}^{\mathbb{Z}/o(C)\mathbb{Z}}$ for every class $C$ for which $r_{C}>0$ (this is not a trivial assertion for general conjugacy classes in arbitrary finite groups, as the case of a non-trivial conjugacy class that is contained in $[G,G]$ shows). To see this, note that the image $\{\chi(C)|\chi\in\widehat{G}\}$ is a subgroup of the group of roots of unity of order $o(C)$, so that if it contains $d_{C}$ elements for some divisor $d_{C}$ of $o(C)$ then there are $\frac{n^{ab}}{d_{C}}$ characters $\chi$ for which $u_{\chi,C}=\frac{o(C)}{d_{C}}l$ for every $0 \leq l<d_{C}$. But then $\sum_{\chi\in\widehat{G}}t_{\chi}$, which was seen above to equal $n^{ab}$, is $\sum_{C \neq Id_{X}}\frac{n^{ab}r_{C}}{d_{C}}\sum_{l=0}^{d_{C}-1}\frac{o(C)l/d_{C}}{o(C)}=\frac{n^{ab}}{2}\sum_{C \neq Id_{X}}r_{C}\big(1-\frac{1}{d_{C}}\big)$, and we get $\sum_{C \neq Id_{X}}r_{C}\big(1-\frac{1}{d_{C}}\big)=2$. But Corollary \ref{Galgen} with our values of $g_{X}$ and $g_{S}$ yields $\sum_{C \neq Id_{X}}r_{C}\big(1-\frac{1}{o(C)}\big)=2$, so that $\sum_{C \neq Id_{X}}\frac{r_{C}}{o(C)}=\sum_{C \neq Id_{X}}\frac{r_{C}}{d_{C}}$. Since $d_{C} \leq o(C)$ for every $C$ we find that $d_{C}=o(C)$ for every $C$ with $r_{C}>0$, which establishes the surjectivity claim.

We conclude that wherever $r_{C}>0$, each number between 0 and $o(C)-1$ equals $u_{\chi,C}$ for $\frac{n^{ab}}{o(C)}$ characters $\chi$. But this implies that when we evaluate \[\sum_{\chi\in\widehat{G}}\bigg[\sum_{C \neq Id_{X}}r_{C}\bigg\{\frac{q-1-u_{\chi,C}}{o(C)}\bigg\}-1\bigg]=\sum_{C \neq Id_{X}}r_{C}\sum_{\chi\in\widehat{G}}\bigg\{\frac{q-1-u_{\chi,C}}{o(C)}\bigg\}-n^{ab},\] the inner sum over $\chi$ yields $\frac{n^{ab}}{2}\big(1-\frac{1}{o(C)}\big)$ for every $\chi$ (as in the previous paragraph) and the total expression vanishes. Recalling that $i^{q}_{\chi}(0_{X})$ equals $\sum_{C \neq Id_{X}}r_{C}\big\{\frac{q-1-u_{\chi,C}}{o(C)}\big\}-1+\delta$, we find that $i^{q}(0_{X})=\sum_{\chi\in\widehat{G}}i^{q}_{\chi}(0_{X})$ reduces to $\sum_{\chi\in\widehat{G}}\delta$. Since $i^{q}(0_{X})=1$, we deduce that $\delta=1$ for precisely one character, which we denote by $\chi_{q}$. We conclude that $\big\{\frac{q-1-u_{\chi,C}}{o(C)}\big\}=0$ for every $C$ with $r_{C}>0$ if and only if $\chi=\chi_{q}$, which implies that $\chi_{q}=\mathbf{1}$ if and only if $q$ satisfies the required congruences. As the equality from Corollary \ref{Galgen} implies the existence of a non-trivial class $C$ with $r_{C}>0$, not all these congruences are trivial, and there are values of $q$ for which $\chi_{q}\neq\mathbf{1}$ when $g_{S}=0$. This completes the proof of the proposition.
\end{proof}

Note that as the proof of Proposition \ref{gX1} shows, the parameter $\delta$ takes the value 1 when $g_{X}=1$ and $g_{S}=0$ precisely when the number from Proposition \ref{dimchardif} equals $\delta-1$, assuring that the dimension $i_{\chi}(0_{X})$ remains non-negative.

We remark that the condition $g_{X}=1$ in Proposition \ref{gX1} is rather restrictive on $G$. When $g_{S}=1$ we get an isogeny of elliptic curves, for which the Galois group is a finite Abelian group generated by at most two elements. On the other hand, if $g_{S}=0$ then the equality from Corollary \ref{Galgen} can be satisfied if there are 4 branching images with $\psi$-classes all of order 2, or if there are 3 branching images whose $\psi$-classes have orders $(3,3,3)$, $(2,4,4)$ or $(2,3,6)$ (these $\psi$-classes are not necessarily distinct), and in no other situation. Moreover, we have seen in the proof of Proposition \ref{gX1} that when $r_{C}>0$, each value $\zeta_{o(C)}^{l}$ is attained on $C$ by some character $\chi\in\widehat{G}$. We also know that the number $t_{\chi}$ must be integral for every $\chi$ (see Proposition \ref{divgens}), and in fact it must equal 0 for $\chi=\mathbf{1}$, 2 for $\chi=\chi_{X}$, and 1 for any other character $\chi$ (recall the value of the dimensions $i_{\chi}(0_{X})$ and their sum $i(0_{X})$ with these values of $g_{X}$ and $g_{S}$). In fact, the definition of $t_{\chi}$ in Equation \eqref{uchiCtchi} and the equality from Corollary \ref{Galgen} then imply that $u_{\chi_{X},C}=o(C)-1$ for every class $C$ with $r_{C}>0$. Knowing the numbers $u_{\chi,C}$ for every class $C$ with $r_{C}>0$ determines $\chi$ (the proof of Proposition \ref{gX1} shows that if $g_{X}=1$, $g_{S}=0$, and $\chi(C)=1$ wherever $r_{C}>0$, then $\chi=\mathbf{1}$), so that $\widehat{G}$ must be embedded as a subgroup of $\prod_{\{C|r_{C}>0\}}\zeta_{o(C)}^{\mathbb{Z}/o(C)\mathbb{Z}}$.

This allows us to deduce that $G$ is a group whose character group $\widehat{G}$ is one of the following finite list. In the case with 4 branching images with $\psi$-images of order 2 (not necessarily distinct) we know that $\chi_{X}$ has order 2, and there might be (in case enough classes are distinct) additional characters, which attain $-1$ on two classes and 1 on two others (for the integrality of $t_{\chi}$). Hence in this case $\widehat{G}$ is a vector space over the field of two elements, of dimension between 1 and 3. If $f$ has 3 branching images (not necessarily distinct) whose $\psi$-images are of order 3 then $\chi_{X}$ has order 3, and in case all the 3 classes are distinct $\widehat{G}$ may also contain the characters attaining the three different possible values on the three different classes. Hence $\widehat{G}$ is then a vector space over the field of three elements, whose dimension is either 1 or 2. In the case of branching images with $\psi$-orders 2, 4, and 4, the order of $\chi_{X}$ is 4, and there are two possibilities: Either $\chi_{X}$ generates $\widehat{G}$, or, in case the two classes of order 4 are distinct, $\widehat{G}$ might be of order 8 (the additional characters either attain $i$ and $-i$ on the two classes of order 4 and 1 on the third class, or attain $-1$ on two of the classes and 1 on the remaining one). Finally, for classes of orders 2, 3, and 6 denote by $C_{6}$ the class for which $o(C_{6})=6$ and $r_{C_{6}}>0$ (i.e., $r_{C_{6}}=1$), and if $\chi\in\widehat{G}$ is such that $\chi(C_{6})$ is of order 6 then the integrality of $t_{\chi}$ implies that the numbers $u_{\chi,C}$ for the other two classes with $r_{C}>0$ is uniquely determined. Therefore $\chi_{X}$ generates $\widehat{G}$, and it is of order 6. Examples of all these branching situations can be obtained using $Z_{n}$-curves (where $G$ and $\widehat{G}$ are cyclic): The $Z_{n}$-curves $w^{2}=\prod_{i=1}^{4}(z-\lambda_{i})$, $w^{3}=\prod_{i=1}^{3}(z-\lambda_{i})$, $w^{4}=(z-\lambda_{1})(z-\lambda_{2})(z-\lambda_{3})^{2}$, and $w^{6}=(z-\lambda_{1})(z-\lambda_{2})^{2}(z-\lambda_{3})^{3}$ (with the different $\lambda_{i}$s in the same equation being distinct complex numbers) all have genus 1 with no branching over $\infty$.

\smallskip

Writing $\chi_{q}=\mathbf{1}$ for $q=1$ also when $g_{X}\neq1$, we can use Proposition \ref{gX1} to express the number $\delta$ from Proposition \ref{dimchardif} in general using Kronecker delta symbols: Indeed, $\delta$ is the product of $\delta_{\Gamma,0_{S}}$, $\delta_{(g_{X}-1)(q-1),0}$, and $\delta_{\chi,\chi_{q}}$.

As always, the form of our results becomes simpler when $S=\mathbb{P}^{1}(\mathbb{C})$. As the assertion remains unaffected when $\Gamma$ is replaced by a linearly equivalent divisor, if suffices to present the consequence of Propositions \ref{dimchardif} and \ref{gX1} for this case under the assumption that $\Gamma=p\cdot\infty$ (so that $f^{*}\Gamma$ is normalized).
\begin{cor}
If $f:X\to\mathbb{P}^{1}(\mathbb{C})$ is a Galois cover with Galois group $G$, $p$ and $q$ are in integers from $\mathbb{Z}$, and $\chi$ is a character from $\widehat{G}$, then set $d_{p,q,\chi}$ to be the number $p+\sum_{C \neq Id_{X}}r_{C}\big\{\frac{q-1-u_{\chi,C}}{o(C)}\big\}+\frac{q-1}{n}(2g_{X}-2)-2$. Then the space $\Omega^{q}\big(-p\sum_{\upsilon=1}^{n}\infty_{\upsilon}\big)_{\chi}$ is $\mathcal{P}_{\leq d_{p,q,\chi}}(f^{*}z)\omega_{\chi,q}$ in case $d_{p,q,\chi}\geq-1$, of dimension $d_{p,q,\chi}+1$, and is trivial otherwise. For $q=1$ this space reduces to $\mathcal{P}_{\leq t_{\overline{\chi}}+p-2}(z)\omega_{\chi}$ (or 0). If $g_{X}$ and $q$ satisfy the positivity condition and $p\geq0$ (i.e., $\Gamma\geq0_{S}$) then $d_{p,q,\chi}\geq-1$ if and only if the parameter $\delta$ from Proposition \ref{dimchardif} vanishes (in particular this is always the case when $p=\deg\Gamma>0$). For Abelian $G$ the space $\Omega^{q}\big(-p\sum_{\upsilon=1}^{n}\infty_{\upsilon}\big)$ and its dimension $i^{q}\big(-p\sum_{\upsilon=1}^{n}\infty_{\upsilon}\big)$ are the usual direct sum and sum respectively. \label{dimdifP1C}
\end{cor}

\begin{proof}
As in the proof of Corollary \ref{funspP1C}, the first three assertions follow from the proofs of Propositions \ref{dimchardif} and \ref{gX1}, together with the additional information given in Corollaries \ref{P1Cnorm} and \ref{P1Cdiff} for this case (recall also that when $g_{S}=0$ the equality $\deg(\Sigma+\Gamma)=2g_{S}-2=-2$ implies the canonicity of this divisor). The last assertion, about Abelian covers, is proved as in Proposition \ref{Omega1/Deltachi}, or as in Corollaries \ref{totdimP1C} and \ref{diffGab}. This proves the corollary.
\end{proof}
The case with $q=1$ and $p=0$ in the last assertion in Corollary \ref{dimdifP1C} generalizes Proposition 1.3 of \cite{[Z]} to the case of general Abelian covers of $\mathbb{P}^{1}(\mathbb{C})$.

We recall again that for an Abelian cover $X$ of $\mathbb{P}^{1}(\mathbb{C})$ and $p>0$ the total dimension $i\big(-p\sum_{\upsilon=1}^{n}\infty_{\upsilon}\big)$ equals the sum $\sum_{\chi\in\widehat{G}}(t_{\overline{\chi}}-1+p)$, while the fact that for $g_{S}=p=0$ the term $i_{\mathbf{1}}(0_{X})$ associated with $\chi=\mathbf{1}$ vanishes allows us to write the dimension $g_{X}$ as $\sum_{\mathbf{1}\neq\chi\in\widehat{G}}(t_{\overline{\chi}}-1)$. For $p<0$ we shall have to begin distinguishing between the characters $\chi$ according to whether $t_{\overline{\chi}}\geq|p|+1$ or not. In fact, recalling that all the assertions about the case with $g_{S}=0$ (namely Corollaries \ref{P1Cnorm}, \ref{funspP1C}, \ref{totdimP1C}, \ref{P1Cdiff}, and \ref{dimdifP1C}) were proven without invoking the Riemann--Roch Theorem, we can prove that theorem for $G$-integral divisors on Abelian covers of $\mathbb{P}^{1}(\mathbb{C})$ directly from our results. Indeed, by Proposition \ref{norminvdiv} it suffices to prove the assertion for normalized divisors, and for such a divisor consider the difference between the expression for $i(\Delta)$ from Corollary \ref{P1Cdiff} from $r(-\Delta)$ given in Corollary \ref{totdimP1C}. Observing that $x=\max\{0,x\}-\max\{0,-x\}$ for any number $x$, this difference is the sum $\sum_{\chi\in\widehat{G}}\big(p+1+\sum_{Id_{X}\neq\sigma \subseteq G}|A_{\sigma,\chi}|-t_{\chi}\big)$, and the proofs of Theorems \ref{intnonsp} and \ref{eqnointg-1} evaluate the latter sum as the required value $\deg\Delta+1-g_{X}$.

\smallskip

Since evaluating the trace of a given element $\tau \in G$ depends only on the behavior of the space in question under the cyclic group generated by $\tau$, we can deduce a generalization of the Eichler Trace Formula (stated and proved as, e.g., the theorem in Subsection V.2.9 of \cite{[FK]}). For this, and for the generalization of the Chevalley--Weil formula to follow, we shall need the following identity.
\begin{lem}
For any integer $d\geq1$ and any invertible number $y\neq1$, the sum $\sum_{l=0}^{d-1}ly^{l+1}$ equals $\frac{1-y^{d}}{(1-y^{-1})^{2}}+\frac{dy^{d}}{1-y^{-1}}$. In particular, if $y$ is a $d$th root of unity then the latter expression is just $\frac{d}{1-\overline{y}}$. \label{transform}
\end{lem}

\begin{proof}
We work by induction on $d$, where for $d=1$ both sides clearly vanish (since $\frac{1-y^{1}}{(1-y^{-1})^{2}}=-\frac{1 \cdot y^{1}}{1-y^{-1}}$). If the assertion holds for $d$, the one for $d+1$ would follow if we show that adding $dy^{d+1}$ to $\frac{1-y^{d}}{(1-y^{-1})^{2}}+\frac{dy^{d}}{1-y^{-1}}$ yields the next expression $\frac{1-y^{d+1}}{(1-y^{-1})^{2}}+\frac{(d+1)y^{d+1}}{1-y^{-1}}$. But the difference $\frac{1-y^{d+1}}{(1-y^{-1})^{2}}-\frac{1-y^{d}}{(1-y^{-1})^{2}}$ is $\frac{y^{d}-y^{d+1}}{(1-y^{-1})^{2}}=-\frac{y^{d+1}}{1-y^{-1}}$, which cancels with $\frac{y^{d+1}}{1-y^{-1}}$, and the difference between the remaining term $\frac{dy^{d+1}}{1-y^{-1}}$ and $\frac{dy^{d}}{1-y^{-1}}$ is indeed $\frac{dy^{d+1}-dy^{d}}{1-y^{-1}}=dy^{d+1}$. This proves the first assertion, and after substituting $y^{d}=1$ and $y^{-1}=\overline{y}$ for $y$ a $d$th root of unity we deduce the second assertion as well. This proves the lemma.
\end{proof}

We can now prove our version of the Eichler Trace Formula. Consider a compact Riemann surface $X$ of genus $g_{X}$, with some element $Id_{X}\neq\tau \in Aut(X)$ generating a cyclic subgroup $H$ of $Aut(X)$, of some finite order $d=o(\tau)$. Let $Y$ be the quotient Riemann surface $H \backslash X$, of some genus $g_{Y}$, with $\pi:X \to Y$ the projection, and consider a positive divisor $\Sigma$ on $Y$ and an integer $q\in\mathbb{Z}$ such that $g_{X}$ and $q$ satisfy the positivity condition. Given a fixed point $P$ of $\tau$, set $\pi(P)=\eta$, and let $\chi_{P}$ denote the character taking the generator $\psi(\eta)$ of $H$ to $\zeta_{d}$. With this notation, the formula is as follows.
\begin{prop}
The trace of the action of $\tau$ on the space $\Omega^{q}(-\pi^{*}\Sigma)$ is the sum $\sum_{\{P \in X|\tau(P)=P\}}\frac{\chi_{P}(\tau)^{q}}{1-\chi_{P}(\tau)}+\delta_{\Sigma,0_{Y}}\cdot\delta_{(g_{X}-1)(q-1),0}\cdot\chi_{q}(\tau)$. Here $\chi_{q}$ is defined in Proposition \ref{gX1} when $g_{X}=1$, and is just $\mathbf{1}$ for $q=1$ and general $g_{X}$.  \label{tracesigma}
\end{prop}

\begin{proof}
If $\phi$ is the character of $H$ that takes $\tau$ to $\zeta_{d}$ then $\widehat{H}$ consists of the powers $\phi^{\alpha}$ with $\alpha\in\mathbb{Z}/d\mathbb{Z}$, and given a representation space $V$ of $H$, we know that $\tau$ operates on the subspace $V_{\chi}$ for $\chi=\phi^{\alpha}$ as multiplication by $\zeta_{d}^{\alpha}$. As for the spaces $\Omega^{q}(-f^{*}\Gamma)$ under consideration, Proposition \ref{dimchardif} (completed by Proposition \ref{gX1}) provides us with the dimension of the space associated with every such $\chi$, together with the fact that their direct sum is the whole space (since $H$ is Abelian). The proof of that proposition shows that the dimension of the space in question is the sum of $2g_{Y}-2+\frac{q-1}{d}(2g_{X}-2)+\deg\Gamma$, the expression $\delta_{\Gamma,0_{S}}\cdot\delta_{(g_{X}-1)(q-1),0}\cdot\delta_{\chi,\chi_{q}}$ (see Proposition \ref{gX1}), and a sum of fractional parts. The trace is obtained by multiplying this dimension by $\zeta_{d}^{\alpha}$ and then taking the sum over $\alpha$. When doing so the contributions of the constants cancel (since  $\sum_{\alpha\in\mathbb{Z}/d\mathbb{Z}}\zeta_{d}^{\alpha}$ vanishes for $d>1$), $\delta_{\Gamma,0_{S}}\cdot\delta_{(g_{X}-1)(q-1),0}\cdot\delta_{\chi,\chi_{q}}$ adds the required term involving $\delta$-symbols, and it remains to consider the fractional parts. The non-trivial elements in $H$ are $\sigma=\tau^{\beta}$ for $0\neq\beta\in\mathbb{Z}/d\mathbb{Z}$, and as for each $\chi=\phi^{\alpha}$ the value $u_{\chi,\sigma}$ is congruent to $\frac{\alpha\beta}{\gcd\{d,\beta\}}$ modulo $o(\sigma)=\frac{d}{\gcd\{d,\beta\}}$ (since $\chi(\sigma)=\zeta_{d}^{\alpha\beta}$, and we can cancel $\gcd\{d,\beta\}$ from the index and the power), we find that $r_{\tau^{\beta}}$ multiplies $\big\{\frac{\gcd\{d,\beta\}(q-1)-\alpha\beta}{d}\big\}$. Multiplying each such term by $\zeta_{d}^{\alpha}$ and summing over $\alpha$, we consider first those $\beta$ for which $\gcd\{d,\beta\}>1$. In this case we can decompose the sum over $\alpha$ to representatives for $\mathbb{Z}/\frac{d}{\gcd\{d,\beta\}}\mathbb{Z}$ in $\mathbb{Z}/d\mathbb{Z}$, and $\big\{\frac{\gcd\{d,\beta\}(q-1)-\alpha\beta}{d}\big\}$ is independent of the choice of representative. For every such representative we therefore get a constant ($\zeta_{d}$ raised to the power of that representative) times the sum $\sum_{\gamma}\zeta_{d}^{\gamma}$ over $\gamma\in\frac{d}{\gcd\{d,\beta\}}\mathbb{Z}/d\mathbb{Z}$. But this sum coincides with $\sum_{\varepsilon\in\mathbb{Z}/\gcd\{d,\beta\}\mathbb{Z}}\zeta_{\gcd\{d,\beta\}}^{\varepsilon}$, so that it vanishes since we assumed that $\gcd\{d,\beta\}>1$.

Therefore only the contributions of elements $\tau^{\beta}$ for $\beta$ co-prime to $d$ have to be considered. As for each such $\beta$ we have $o(\tau^{\beta})=d=|H|$, each point of $Y$ with such a $\psi$-image (which thus generates $H$) is the image of a single point of $X$, which is thus a fixed point of $H$ hence of $\tau$. The same argument shows that all the fixed points of $\tau$ in $X$ have images in $Y$ with such $\psi$-values. Now, let $\tilde{\beta}$ denote the inverse of $\beta$ modulo $d$, take a point $\eta \in Y$ for which $\psi(\eta)=\tau^{\beta}$ (there are $r_{\tau^{\beta}}$ such points), and let $P$ is the unique pre-image of $\eta$ in $X$. Since in this case the character $\phi^{\tilde{\beta}}$ takes $\psi(\eta)$ to $\zeta_{d}$, it equals $\chi_{P}$ by definition. Moreover, since $\gcd\{d,\beta\}=1$ now, the fractional part is just $\big\{\frac{q-1-\alpha\beta}{d}\big\}$, and for each $0 \leq l<d$ the quotient $\frac{l}{d}$ is obtained as this fractional part for precisely one value of $\alpha$, namely $\alpha=\tilde{\beta}(q-1-l)$. This summation index change implies that each such point $P$ contributes $\sum_{l=0}^{d-1}\frac{l}{d}\zeta_{d}^{\tilde{\beta}(q-1-l)}$ to the trace of $\tau$, in which we write $\zeta_{d}^{\tilde{\beta}}=\chi_{P}(\tau)$ and obtain $\frac{\chi_{P}(\tau)^{q}}{d}\sum_{l=0}^{d-1}l\cdot\overline{\chi_{P}(\tau)}^{l+1}$. As $\overline{\chi_{P}(\tau)}$ is a $d$th root of unity distinct from 1 (because it is primitive and $d=o(\tau)>1$), Lemma \ref{transform} transforms the latter expression to the desired one. This proves the proposition.
\end{proof}

Note that the sum over $l$ appearing in the proof of Proposition \ref{tracesigma} (as well as in the proof of Theorem \ref{CWgen} below) is closely related to the representation called the ramification module in Section 3 of \cite{[JK]}.

\smallskip

Knowing the traces of the action of every element of $G$ on $\Omega^{q}(-f^{*}\Gamma)$, we can now establish the following generalization of the Chevalley--Weil formula from \cite{[CW]} and \cite{[W]}. We denote the set of isomorphism classes of complex irreducible representations of $G$ by $Irr_{\mathbb{C}}(G)$, and for every representation $\rho$ in that set let $d_{\rho}$ be the dimension of its representation space and let $\chi_{\rho}$ be its character. For any conjugacy class $C$, the element $\rho(\sigma)$ for $\sigma \in C$ decomposes the representation space of $\rho$ into eigenspaces of eigenvalues $\zeta_{o(C)}^{\alpha}$ for $\alpha\in\mathbb{Z}/o(C)\mathbb{Z}$. The dimensions of these spaces are independent of the choice of $\sigma \in C$, and they are denoted by $N_{C,\alpha}^{\rho}$ for such $\rho$, $C$, and $\alpha$. The sum $\sum_{\alpha\in\mathbb{Z}/o(C)\mathbb{Z}}N_{C,\alpha}^{\rho}$ is therefore the full dimension $d_{\rho}$ of the representation space of $\rho$.
\begin{thm}
Let $f:X \to S$ be a Galois cover with Galois group $G$, and take a positive divisor $\Gamma$ on $S$ and an index $q$ such that $g_{X}$ and $q$ satisfy the positivity condition $(g_{X}-1)(q-1)\geq0$. Then a representation $\rho \in Irr_{\mathbb{C}}(G)$ appears in the representation of $G$ on the space $\Omega^{q}(-f^{*}\Gamma)$ with multiplicity \[d_{\rho}[(2q-1)(g_{S}-1)+\deg\Gamma]+
\sum_{C \neq Id_{X}}r_{C}\sum_{\alpha=0}^{o(C)-1}N_{C,\alpha}^{\rho}\big[(q-1)\big(1-\tfrac{1}{o(C)}\big)+\big\{\tfrac{q-1-\alpha}{o(C)}\big\}\big]+\delta,\] where $\delta$ is the product $\delta_{\Gamma,0_{S}}\cdot\delta_{(g_{X}-1)(q-1),0}\cdot\delta_{\rho,\chi_{q}}\in\{0,1\}$. \label{CWgen}
\end{thm}

\begin{proof}
We recall from representation theory that if $V$ is a finite dimensional vector space on which $G$ acts, then $\rho$ appears in the resulting representation with multiplicity $\frac{1}{n}\sum_{\tau \in G}\chi_{V}(\tau)\chi_{\rho}(\tau^{-1})$. But the traces $\chi_{V}(\tau)$ for $V=\Omega^{q}(-f^{*}\Gamma)$ are given, for $\tau \neq Id_{X}$, in Proposition \ref{tracesigma}, so we begin by evaluating the sum over those $\tau$. For each such $\tau$ set $Y$ and $\pi:X \to Y$ to be as in that proposition, denote the resulting map from $Y$ to $S$ by $p$ (so that $p\circ\pi=f$), and the divisor $\Sigma$ that we take is $p^{*}\Gamma$. Moreover, the multiplicity from Proposition \ref{tracesigma} depends on $\Sigma$ only via the factor $\delta_{\Sigma,0_{Y}}$, which we can write as $\delta_{\Gamma,0_{S}}$, not involving $\Sigma$ at all. In addition, when $\Gamma=0_{S}$ (i.e., $\Sigma=0_{Y}$) and the positivity condition is not strict, the character $\chi_{q}$ on $G$ here restricts to the character $\chi_{q}$ on $H$ (indeed, both are trivial for $q=1$ and come from the action of $Aut(X)$ when $g_{X}=1$). We may therefore apply Proposition \ref{tracesigma} directly, and then the terms involving $\delta$ combine to $\delta_{\Gamma,0_{S}}\cdot\delta_{(g_{X}-1)(q-1),0}$ times $\frac{1}{n}\sum_{Id_{X}\neq\tau \in G}\chi_{q}(\tau)\chi_{\rho}(\tau^{-1})$. The orthogonality of characters and the known values of $\chi_{q}$ and $\chi_{\rho}$ on $Id_{X}$ shows that the latter sum equals $\delta_{\rho,\chi_{q}}-\frac{d_{\rho}}{n}$, which produces the required product $\delta$ of Kronecker delta symbols, and the extra term $-\frac{d_{\rho}}{n}\delta_{\Gamma,0_{S}}\cdot\delta_{(g_{X}-1)(q-1),0}$.

In the sum of the remaining parts over $Id_{X}\neq\tau \in G$ we change the order of summation, and get $\frac{1}{n}\sum_{P \in X}\sum_{Id_{X}\neq\tau \in G_{P}}\frac{\chi_{P}(\tau)^{q}}{1-\chi_{P}(\tau)}\chi_{\rho}(\tau^{-1})$, where $G_{P}$ stands for the stabilizer of $P$ in $G$ and $\chi_{P}$ is the character from Proposition \ref{tracesigma}. Given $P \in X$, with $\eta=f(P)$, the stabilizer $G_{P}$ is cyclic, generated by an element of $C=\psi(\eta)$ that is determined by $P$. Since we consider only non-trivial elements of that stabilizer, we get contributions here only from branch points (so that $C \neq Id_{X}$ from now on), and the sum over $P$ is essentially finite. Moreover, the proof of Proposition \ref{tracesigma} shows that $\chi_{P}$ is a character of order $o(C)$ of the cyclic group $G_{P}$ of order $o(C)$, so that the numbers $\chi_{P}(\tau)$ over $Id_{X}\neq\tau \in G_{P}$ are all roots of unity of order $o(C)$ that are different from 1. This allows us to apply Lemma \ref{transform} in the opposite direction to the proof of Proposition \ref{tracesigma}, and replace the quotient involving $\chi_{P}(\tau)$ by $\sum_{l=0}^{o(C)-1}\frac{l}{o(C)}\chi_{P}(\tau)^{q-1-l}$.

Now, recall from the proof of Proposition \ref{StoGbr} that the choice of $P \in X$ lying over $\eta \in S$ determines a specific element $\sigma \in G$ that generates $G_{P}$ and lies in $C$. For such $P$ the elements $\tau$ over which we sum are $\tau=\sigma^{\beta}$ for $0\neq\beta\in\mathbb{Z}/o(C)\mathbb{Z}$, and then $\chi_{P}(\tau)=\zeta_{o(C)}^{\beta}$ by definition, and $\chi_{\rho}(\tau^{-1})$ is $\chi_{\rho}(\sigma^{-\beta})$. Replacing $P$ by another pre-image of $\eta$ sends $\sigma$ to a conjugate in $G$, thus leaving also the value $\chi_{\rho}(\sigma^{-\beta})$ invariant. Our expression thus depends only on $C=\psi(\eta)$, and since each of the $r_{C}$ points $\eta \in S$ with $\psi(\eta)=C$ has $\frac{n}{o(C)}$ pre-images in $X$, we may replace $\frac{1}{n}$ times the sum over $P$ by the sum over $C \neq Id_{X}$ with the respective coefficients $\frac{r_{C}}{o(C)}$. As for the value of $\chi_{\rho}(\sigma^{-\beta})$ for $\sigma \in C$ and $\beta\in\mathbb{Z}/o(C)\mathbb{Z}$, we recall the decomposition of the representation space of $\rho$ according to the action of $\sigma$ and evaluate this expression as $\sum_{\alpha\in\mathbb{Z}/o(C)\mathbb{Z}}N_{C,\alpha}^{\rho}\zeta_{o(C)}^{-\alpha\beta}$. Thus, for every $C$, $\alpha$, and $0 \leq l \leq o(C)-1$ we get $\frac{r_{C}}{o(C)}N_{C,\alpha}^{\rho}\frac{l}{o(C)}$ times $\sum_{0\neq\beta\in\mathbb{Z}/o(C)\mathbb{Z}}\zeta_{o(C)}^{\beta(q-1-l)}\zeta_{o(C)}^{-\alpha\beta}$.

Let us evaluate this expression. By adding and subtracting the term with $\beta=0$ and observing that $o(C)>1$ we deduce that the sum over $\beta$ is $o(C)-1$ when $l \equiv q-1-\alpha\big(\mathrm{mod\ }o(C)\big)$ (so that $\frac{l}{o(C)}=\big\{\frac{q-1-\alpha}{o(C)}\big\}$), and $-1$ otherwise. The terms with the fractional parts produce the corresponding required terms. For the remaining contributions we write $\sum_{l=0}^{o(C)-1}\frac{l}{o(C)}=\frac{o(C)-1}{2}$ and recall that $\sum_{\alpha\in\mathbb{Z}/o(C)\mathbb{Z}}N_{C,\alpha}^{\rho}=d_{\rho}$, and deduce that these terms combine to $-\sum_{C \neq Id_{X}}\frac{r_{C}d_{\rho}}{2}\big(1-\frac{1}{o(C)}\big)$.

It remains to evaluate the contribution of the trivial element $Id_{X} \in G$, which equals just $\frac{d_{\rho}}{n}$ times the dimension $i^{q}(-f^{*}\Gamma)$. In order to do this, first observe that when $\rho$ is a character $\chi$, with $d_{\rho}=1$, the number $N_{C,\alpha}^{\chi}$ for a non-trivial class $C$ equals 1 if $\alpha=u_{\chi,C}+o(C)\mathbb{Z}$ and vanishes otherwise. Hence we have proved that the multiplicity $i^{q}_{\chi}(-f^{*}\Gamma)$ with which $\chi$ appears in $\Omega^{q}(-f^{*}\Gamma)$ is \[\frac{1}{n}i^{q}(-f^{*}\Gamma)+\sum_{C \neq Id_{X}}r_{C}\big[\big\{\tfrac{q-1-u_{\chi,C}}{o(C)}\big\}-\tfrac{1}{2}\big(1-\tfrac{1}{o(C)}\big)\big]+\delta-\frac{1}{n}\delta_{\Gamma,0_{S}}\cdot\delta_{(g_{X}-1)(q-1),0}.\] Comparing this with the value for $i^{q}_{\chi}(-f^{*}\Gamma)$ appearing in Proposition \ref{dimchardif} (with the same $\delta$ by Proposition \ref{gX1}), we deduce that $i^{q}(-f^{*}\Gamma)$ is $\delta_{\Gamma,0_{S}}\cdot\delta_{(g_{X}-1)(q-1),0}$ plus the expression $n\deg\Gamma+n(2q-1)\big[(g_{S}-1)+\sum_{C \neq Id_{X}}\tfrac{r_{C}}{2}\big(1-\tfrac{1}{o(C)}\big)\big]$. When we add $\frac{d_{\rho}}{n}$ times this expression to our value for $\frac{1}{n}\sum_{Id_{X}\neq\tau \in G}\chi_{V}(\tau)\chi_{\rho}(\tau^{-1})$ for general $\rho$, the extra term $-\frac{d_{\rho}}{n}\delta_{\Gamma,0_{S}}\cdot\delta_{(g_{X}-1)(q-1),0}$ cancels, and we get the first required term (with $d_{\rho}$). After we expand $d_{\rho}$ back as $\sum_{\alpha\in\mathbb{Z}/o(C)\mathbb{Z}}N_{C,\alpha}^{\rho}$ in the remaining terms $r_{C}d_{\rho}(q-1)\big(1-\frac{1}{o(C)}\big)$ with $C \neq Id_{X}$, we obtain the full asserted expression. This completes the proof of the theorem.
\end{proof}

The proof of Theorem \ref{CWgen} also shows how for 1-dimensional representations $\rho=\chi\in\widehat{G}$ the result indeed agrees with the one of Proposition \ref{dimchardif}. Moreover, the dimension $i^{q}(-f^{*}\Gamma)$ considered in the proof of Theorem \ref{CWgen} can be expressed, by Corollary \ref{Galgen}, simply as $(2q-1)(g_{X}-1)+n\deg\Gamma+\delta_{\Gamma,0_{S}}\cdot\delta_{(g_{X}-1)(q-1),0}$. This result is also easily obtained by tools similar to those appearing in the proof of Proposition \ref{dimchardif} (i.e., applying the Riemann--Roch Theorem for divisors of degree at least $2g_{X}-2$ and evaluating the correction terms). A similar argument allows us to express the multiplicity from Theorem \ref{CWgen} to
\[d_{\rho}\big[g_{S}-1+\tfrac{q-1}{n}(2g_{X}-2)+\deg\Gamma\big]+\sum_{C \neq Id_{X}}r_{C}\sum_{\alpha=0}^{o(C)-1}N_{C,\alpha}^{\rho}\big\{\tfrac{q-1-\alpha}{o(C)}\big\}+\delta.\] We deduce that wherever $\Gamma$ is non-trivial, or the positivity condition is strict, adding another point to $\Gamma$ increases the representation $\Omega^{q}(-f^{*}\Gamma)$ of $G$ by a copy of the regular representation. Note that our result for $i^{q}(-f^{*}\Gamma)$ in case $\Gamma=0_{S}$ is in correspondence with Proposition III.5.2 of \cite{[FK]}, and it covers all the parts of that proposition for which $\Omega^{q}(0_{X})\neq\{0\}$, except for the case with $g_{X}\geq2$ and $q=0$ (but this case is dealt with in the remark following Proposition \ref{r1/Delta}). In particular, for $q=1$ (where $\delta_{(g_{X}-1)(q-1),0}=1$ and $\chi_{q}=\mathbf{1}$) the multiplicity that we obtain reduces to the expression \[d_{\rho}(g_{S}-1+\deg\Gamma)+\delta_{\Gamma,0_{S}}\delta_{\rho,\mathbf{1}}+\sum_{C \neq Id_{X}}r_{C}\sum_{\alpha=0}^{o(C)-1}N_{C,\alpha}^{\rho}\big\{\tfrac{-\alpha}{o(C)}\big\},\] the total dimension is $g_{X}-1+n\deg\Gamma+\delta_{\Gamma,0_{S}}$, and the latter number is just $g_{X}$ for trivial $\Gamma$. This generalizes the considerations appearing in the proofs of Theorems \ref{intnonsp} and \ref{eqnointg-1} for the Abelian case. We also remark that \cite{[VL]} has obtained a result very similar to Theorem \ref{CWgen}, with $q=0$ and arbitrary $g_{X}$, but provided that the degree of $\Gamma$ is large enough. It should be possible to deduce this result, and perhaps more general ones, in the same way that we proved Theorem \ref{CWgen}. We leave these questions for future research.

\section{Jacobian Decompositions \label{JacDecom}}

For a general compact Riemann surface $X$, of genus $g_{X}$, we denote by $J(X)$ its Jacobian. As a complex torus, it is the vector space $T_{0}J(X)=\mathrm{Hom}\big(\Omega(0_{X}),\mathbb{C}\big)$ divided by the lattice corresponding to the integral homology group $H_{1}(X,\mathbb{Z})$. The action of a finite subgroup $G$ of $Aut(X)$ on $J(X)$ expresses itself in the \emph{analytic representation} $\rho_{a}$ on $T_{0}J(X)$, as well as in the \emph{rational representation} $\rho_{r}$ on $H_{1}(X,\mathbb{Z})\otimes\mathbb{Q}=H_{1}(X,\mathbb{Q})$. Therefore $T_{0}J(X)$ decomposes according to $Irr_{\mathbb{C}}(G)$, while for $H_{1}(X,\mathbb{Q})$ we have a (typically coarser) decomposition according to the set $Irr_{\mathbb{Q}}(G)$ of representations of $G$ that are defined over $\mathbb{Q}$ and are irreducible in this sense. Combining these decompositions allows us to write $J(X)$ as the product (up to isogeny) of Abelian varieties associated with elements of $Irr_{\mathbb{Q}}(G)$ (see \cite{[LR]}). The paper \cite{[R]} determines the dimensions of these Abelian varieties, using the analysis of genera of intermediate curves and the invertibility of a certain matrix arising from the character table of $G$. A similar method is used in \cite{[JK]} for obtaining some analogues of Theorem \ref{CWgen}, with $q=0$ but with a large divisor $\Gamma$. Finding these dimensions is equivalent to determining the decomposition of $\rho_{r}$. In this section we show how our results easily imply the decomposition of $\rho_{a}$. As the complexification $\rho_{r}\otimes\mathbb{C}$ of $\rho_{r}$ is isomorphic to the direct sum of $\rho_{a}$ with its complex conjugate representation $\overline{\rho}_{a}$ (see, e.g., \cite{[BL]}), our decomposition also yields a finer decomposition of $\rho_{r}$.

\smallskip

Given $\rho \in Irr_{\mathbb{C}}(G)$, we denote by $\mathbb{K}_{\rho}$ the (cyclotomic) field generated over $\mathbb{Q}$ by the numbers $\chi_{\rho}(\sigma)$ with $\sigma \in G$, with Galois group $\Gamma_{\rho}$ over $\mathbb{Q}$, and set $k_{\rho}$ to be $|\Gamma_{\rho}|=[\mathbb{K}_{\rho}:\mathbb{Q}]$. Then the sum $\sum_{\tau\in\Gamma_{\rho}}(\tau\circ\chi_{\rho})$ takes values in $\mathbb{Q}$. We recall, however, that the character of the element $W \in Irr_{\mathbb{Q}}(G)$ whose complexification $W\otimes\mathbb{C}$ contains $\rho$ is not the aforementioned sum, but its multiple by an integer $m_{\rho}$ called the \emph{Schur index} of $\rho$ (see, e.g., Chapter 12 of \cite{[S]}). The Schur index $m_{\rho}$ is known to divide $d_{\rho}$ for every such $\rho$, so that in particular $m_{\chi}=1$ for $\chi\in\widehat{G}$. Moreover, the numbers $d_{\rho}$, $k_{\rho}$ and $m_{\rho}$ (as well as the group $\Gamma_{\rho}$) depend only on $W \in Irr_{\mathbb{Q}}(G)$ and not on the choice of $\rho \in Irr_{\mathbb{C}}(G)$ contained in $W\otimes\mathbb{C}$, hence they can be denoted with index $W$ instead of $\rho$ (recall though that the dimension of $W$ is not $d_{W}$, but rather the product $d_{W}k_{W}m_{W}$). For a conjugacy class $C \subseteq G$, we recall that $N_{C,0}^{\rho}$ is the dimension of the space of elements of the representation space of $\rho$ that are pointwise fixed by $\rho(\sigma)$ for an element $\sigma \in C$, and this dimension is independent of the choice of $\sigma \in C$.

The next proposition determines, using the latter notation, the number of times each element of $Irr_{\mathbb{C}}(G)$ appears in the analytic representation $\rho_{a}$ of $G$ on $T_{0}J(X)$, as well as the multiplicity of representations from $Irr_{\mathbb{Q}}(G)$ in the rational representation $\rho_{r}$ of $G$. In particular, it reproduces the main technical statement from \cite{[R]}, namely Theorem 5.10 of that reference.
\begin{prop}
The multiplicity with which an element $\rho \in Irr_{\mathbb{C}}(G)$ appears in $\rho_{a}$ is $d_{\rho}(g_{S}-1)+\delta_{\rho,\mathbf{1}}+\sum_{C \neq Id_{X}}r_{C}\sum_{\alpha=0}^{o(C)-1}N_{C,\alpha}^{\rho}\big\{\tfrac{\alpha}{o(C)}\big\}$. In particular, each character $\mathbf{1}\neq\chi\in\widehat{G}$ appears in $\rho_{a}$ precisely $g_{S}+t_{\chi}-1$ times (with $t_{\chi}$ defined in Equation \eqref{uchiCtchi}), while $\mathbf{1}$ appears there with multiplicity $g_{S}$. In $\rho_{r}\otimes\mathbb{C}$ the representation $\rho$ appears $d_{\rho}(2g_{S}-2)+2\delta_{\rho,\mathbf{1}}+\sum_{C \neq Id_{X}}r_{C}(d_{\rho}-N_{C,0}^{\rho})$ times. For characters this number is $2g_{S}-2+\sum_{\{C|\chi(C)\neq1\}}r_{C}$ if $\chi\neq\mathbf{1}$ and $2g_{S}$ for $\chi=\mathbf{1}$. The multiplicity with which an element $W \in Irr_{\mathbb{Q}}(G)$ is contained in $\rho_{r}$ is therefore $\frac{d_{W}}{m_{W}}(2g_{S}-2)+2\delta_{W,\mathbf{1}}+\sum_{C \neq Id_{X}}r_{C}\frac{d_{W}-N_{C,0}^{W}}{m_{W}}$, which for representations $W$ that are sum of Galois conjugate characters becomes $2g_{S}-2+\sum_{C \not\in\ker W}r_{C}$ if these characters are non-trivial and just $2g_{S}$ for $W=\mathbf{1}$. \label{multinJac}
\end{prop}

\begin{proof}
Since the operation of taking the dual commutes with finite direct sums of representations, the multiplicity to which $\rho$ appears in the action of $G$ on $T_{0}J(X)$ via $\rho_{a}$ is the multiplicity to which the representation $\rho^{*}$ dual to $\rho$ appears in the representation $\Omega(0_{X})$. Setting $q=1$ and $\Gamma=0_{S}$ in Theorem \ref{CWgen} (i.e., the case $q=1$ of the classical formula of Chevalley and Weil), and recalling that $\chi_{\rho^{*}}(\sigma)=\chi_{\rho}(\sigma^{-1})$ for each $\sigma \in G$ hence $N_{C,\alpha}^{\rho^{*}}=N_{C,-\alpha}^{\rho}$ for $\alpha\in\mathbb{Z}/o(C)\mathbb{Z}$, we obtain the value appearing in the first assertion. We have seen that if $\rho$ is a character $\chi\in\widehat{G}$ then $N_{C,\alpha}^{\rho}$ equals 1 for $\alpha=u_{\chi,C}+o(C)\mathbb{Z}$ and 0 for any other value of $\alpha$, which combines with the value of $t_{\chi}$ from Equation \eqref{uchiCtchi} to produce the assertion about characters in $\rho_{a}$.

For $\rho_{r}\otimes\mathbb{C}$ we add the multiplicity of $\rho$ in $\overline{\rho}_{a}=\Omega(0_{X})$ itself, from which the required assertion follows since $\big\{\frac{\alpha}{o(C)}\}+\big\{\frac{-\alpha}{o(C)}\}$ equals 0 if $\alpha=0\in\mathbb{Z}/o(C)\mathbb{Z}$ and 1 otherwise. For characters $\chi\in\widehat{G}$ it is clear that $N_{C,0}^{\chi}$ is 1 if $\chi(C)=1$ (i.e., if $C\subseteq\ker\chi$) and 0 otherwise, which establishes the assertion about characters in $\rho_{r}\otimes\mathbb{C}$ as well (equivalently, it follows from the results about characters in $\rho_{a}$ via Lemma \ref{tuchichibar}). The multiplicity of $\rho \in Irr_{\mathbb{C}}(G)$ (and in particular $\chi\in\widehat{G}$) in $\rho_{r}\otimes\mathbb{C}$ is thus invariant under the action of $\Gamma_{\rho}$ (the invariance of $N_{C,0}^{\rho}$ is easy as well), so that the direct sum over the Galois orbit of $\rho$ appears together with this multiplicity. The relation between $W$ and this direct sum thus implies the remaining assertions, using the observation that $m_{W}=1$ if $W$ is the direct sum of characters from $\widehat{G}$ (i.e., when $d_{W}=1$). This proves the proposition.
\end{proof}

Proposition \ref{multinJac} allows us to reproduce the basic assertions from Sections 1 and 2 of \cite{[LR]}, together with the main result (Theorem 5.12) of \cite{[R]}. We present these results here since we interpret them using generalized Prym varieties of cyclic covers of $S$ in Theorem \ref{PryminJac} below. We shall denote the ring of endomorphisms of an Abelian variety $A$ by $\mathrm{End}(A)$, and the algebra $\mathrm{End}(A)\otimes\mathbb{Q}$ by $\mathrm{End}_{0}(A)$. Note that the action of $G$ on $X$ (hence on $J(X)$) embeds the group ring $\mathbb{Z}[G]$ of $G$ over $\mathbb{Z}$ into $\mathrm{End}\big(J(X)\big)$, hence the group algebra $\mathbb{Q}[G]$ of $G$ over $\mathbb{Q}$ into $\mathrm{End}_{0}\big(J(X)\big)$.
\begin{prop}
The Jacobian $J(X)$ decomposes, up to isogeny, into the product over $W \in Irr_{\mathbb{Q}}(C)$ of Abelian varieties, where the Abelian variety $A_{W}$ associated with $W$ has dimension
$k_{W}d_{W}^{2}(g_{S}-1)+\delta_{W,\mathbf{1}}+\sum_{C \neq Id_{X}}\frac{k_{W}d_{W}}{2}r_{C}(d_{W}-N_{C,0}^{W})$. Moreover, $A_{W}$ is isogenous to the $\frac{d_{W}}{m_{W}}$th power of an Abelian variety $B_{W}$, the dimension of which is
$k_{W}d_{W}m_{W}(g_{S}-1)+\delta_{W,\mathbf{1}}+\sum_{C \neq Id_{X}}\frac{k_{W}m_{W}}{2}r_{C}(d_{W}-N_{C,0}^{W})$, and whose algebra of endomorphisms contains a division algebra of dimension $m_{W}^{2}$ over $\mathbb{K}_{W}$. The Abelian varieties $A_{W}$ are canonical as subvarieties of $J(X)$, while the $B_{W}$s are typically not. \label{Jacdecom}
\end{prop}

\begin{proof}
The $W$-isotypical sub-representation of $\rho_{r}$ and the direct sum over $\tau\in\Gamma_{\rho}$ of the $(\tau\circ\rho)$-isotypical sub-representations of $\rho_{a}$ (where $W\otimes\mathbb{C}$ is isomorphic to the direct sum over $\tau\in\Gamma_{\rho}$ of $m_{W}$ copies of $\tau\circ\rho$) are the images of $H_{1}(X,\mathbb{Q})$ (resp. of $T_{0}J(X)$) under the projector $p_{W}$ associated with $W$ in $\mathbb{Q}[G]\subseteq\mathrm{End}_{0}\big(J(X)\big)$. As an integral multiple of $p_{W}$ lies in $\mathbb{Z}[G]\subseteq\mathrm{End}\big(J(X)\big)$, the image $A_{W}$ of $J(X)$ under such an element is a well-defined canonical Abelian subvariety, which has $p_{W}\big(T_{0}J(X)\big)$ as its tangent space and $p_{W}\big(H_{1}(X,\mathbb{Q})\big)$ as the rational vector space arising from its lattice. Since the dimension of an Abelian variety is half the rank of its lattice, and the dimension over $\mathbb{Q}$ of the $W$-isotypical part of $\rho_{r}$ is $\dim W=d_{W}k_{W}m_{W}$ times the multiplicity from Proposition \ref{multinJac}, the formula for $\dim A_{W}$ is also established. The fact that this is a decomposition up to isogeny follows immediately from the construction (since the sum over $W \in Irr_{\mathbb{Q}}(G)$ of the projectors $p_{W}$ is $Id_{J(X)}$), and the first assertion is proved.

For the second assertion we observe that $\mathrm{End}_{0}(A_{W})=p_{W}\mathrm{End}_{0}\big(J(X)\big)$, and we recall from Section 12 of \cite{[S]} that the ideal $p_{W}\mathbb{Q}[G]$ of $\mathbb{Q}[G]\subseteq\mathrm{End}_{0}\big(J(X)\big)$ (which is a subring of $\mathrm{End}_{0}(A_{W})$) is isomorphic to a central simple algebra of the sort $\mathrm{M}_{d_{W}/m_{W}}(D_{W})$. Here $D_{W}$ is a division algebra of dimension $m_{W}^{2}$ over $\mathbb{K}_{W}$ and $\mathrm{M}_{l}(R)$ is the ring of $l \times l$ matrices over the ring $R$. Since the unit of $\mathrm{M}_{l}(R)$ can be presented as the sum of $l$ conjugate idempotents, an Abelian variety $A$ such that $\mathrm{M}_{l}(R)\subseteq\mathrm{End}_{0}(A)$ must be isogenous to the $l$th self-product of a subvariety $B$ with $R\subseteq\mathrm{End}_{0}(B)$, and $\dim A=l\cdot\dim B$ (but $B$, as a subvariety, depends on the choice of a matrix idempotent). Applying this to $A=A_{W}$ with $l=\frac{d_{W}}{m_{W}}$ establishes the second assertion (as well as the non-canonicity of $B_{W}$ unless $d_{W}=m_{W}$ and $B_{W}=A_{W}$). This proves the proposition.
\end{proof}

The positivity of the dimensions of the subvarieties $A_{W}$ and $B_{W}$ for all $W \in Irr_{\mathbb{Q}}(G)$ when $g_{S}\geq2$, which is Theorem 3.1 of \cite{[LR]}, is an immediate consequence of Proposition \ref{Jacdecom} (or equivalently of Theorem 5.12 of \cite{[R]}).

\smallskip

An equivalent way of formulating the part of the results of Proposition \ref{Jacdecom} concerning characters of $G$ is as follows. As characters of $G$ map $G$ onto finite (hence cyclic) subgroups of $S^{1}$, each such character becomes a faithful character of a cyclic quotient $Q=G/N$ of $G$. Moreover, if $\chi\in\widehat{G}$ is an embedding of $Q$ into $S^{1}$ then the other elements of $\widehat{G}$ that are embeddings of the same quotient $Q$ are $\chi$ raised to powers that are co-prime to $|Q|$, and they are also precisely the images of $\chi$ under the Galois group $\Gamma_{\chi}$. It follows that the cyclic quotients of $G$ are in one-to-one correspondence with representations $W \in Irr_{\mathbb{Q}}(G)$ with $d_{W}=1$ (hence also $m_{W}=1$), and hence if $W$ is associated with $Q$ in this way then $k_{W}=\varphi(|Q|)$, where $\varphi$ is Euler's totient function. Now, since $d_{W}=m_{W}=1$, the subvariety $B_{W}$ from Proposition \ref{Jacdecom} is a canonical subvariety of $J(X)$ (since it equals $A_{W}$), and we may denote it also by $B_{Q}$ for the cyclic subgroup $Q$ associated with $W$. For the dimension of $B_{Q}$ we recall that $d_{W}=1$, while $N_{C,0}^{W}$ equals 1 as well if and only if the elements of $C$ operate trivially via any character $\chi\in\widehat{G} \subseteq Irr_{\mathbb{C}}(G)$ that it contained in $W\otimes\mathbb{C}$. But this means that $C$ is contained in the kernel $N$ of the projection $G \to Q$. Since $W=\mathbf{1}$ if and only if $Q$ if the trivial quotient $G/G$, Proposition \ref{Jacdecom} shows that the dimension of $B_{Q}$ is $\varphi(|Q|)\big[g_{S}-1+\delta_{Q,G/G}+\sum_{C \not\subseteq N}\frac{r_{C}}{2}\big]$. Finally, since $\prod_{W \in Irr_{\mathbb{Q}}(G)}A_{W} \to J(X)$ is an isogeny, the map $\prod_{Q=G/N\ \mathrm{cyclic}}B_{Q} \to J(X)$ has finite kernel, and as we have $Irr_{\mathbb{C}}(G)=\widehat{G}$ for Abelian $G$, the latter map is surjective (hence an isogeny) in this case. We gather these results in the following proposition.
\begin{prop}
To every cyclic quotient $Q=G/N$ of $G$ corresponds a canonical subvariety $B_{Q}$ of $J(X)$, of dimension $\varphi(|Q|)\big[g_{S}-1+\delta_{Q,G/G}+\sum_{C \not\subseteq N}\frac{r_{C}}{2}\big]$. The map $\prod_{Q=G/N\ \mathrm{cyclic}}B_{Q} \to J(X)$ has finite kernel, and in case $G$ is Abelian it is a decomposition of $J(X)$ up to isogeny. \label{Jaccycquot}
\end{prop}

The subvariety $B_{G/G}$ from Proposition \ref{Jaccycquot} (or equivalently $B_{\mathbf{1}}=A_{\mathbf{1}}$ in the terminology of Proposition \ref{Jacdecom}), of dimension $g_{S}$, is just the $f^{*}$-image of the Jacobian $J(S)$ in $J(X)$. Its natural complement (e.g., with respect to the pairing defined by the polarization) is called, in \cite{[LR]} and others, the \emph{Prym variety} $P(X/S)$ of the cover $f:X \to S$. More generally, a Galois map $f:X \to S$ with Galois group $G$ admits an intermediate Riemann surface $Y_{H}$ for every subgroup $H$ of $G$, and we define the complement of the sum of the images of the Jacobians $J(Y_{H})$ for proper subgroups $H$ of $G$ to be the \emph{primitive Prym variety} $\widetilde{P}(X/S)$ of $X$ over $S$. In particular, while $P(S/S)$ is always a trivial variety, the primitive Prym variety $\widetilde{P}(S/S)$ is the Jacobian $J(S)$ itself (since the trivial Galois group has no proper subgroups). This concept might be complicated for general Galois covers, but we shall use it here only for \emph{cyclic} covers.

Proposition \ref{Jaccycquot} now has the following interpretation.
\begin{thm}
For a cyclic quotient $Q=G/N$ of $G$, let $Y_{Q}$ be the Riemann surface associated with the action of $N$ on $X$, so that $Y_{Q}$ is a cyclic cover of $S$ with Galois group $Q$. Considering the primitive Prym varieties $\widetilde{P}(Y_{Q}/S)$, the map $\prod_{Q=G/N\ \mathrm{cyclic}}\widetilde{P}(Y_{Q}/S) \to J(X)$ has finite kernel, and when $G$ is Abelian this map is an isogeny. Moreover, $\widetilde{P}(Y_{Q}/S)$ is non-trivial wherever $g_{Y_{Q}}\geq1$, except when $g_{Y_{Q}}=g_{S}=1$ and $Q \neq G/G$. \label{PryminJac}
\end{thm}

\begin{proof}
We first recall from Proposition \ref{Jacdecom} that the map $f^{*}$ from $J(S)$ onto its image $B_{\mathbf{1}}=A_{\mathbf{1}}$ in $J(X)$ is an isogeny (since both have the same dimension $g_{S}$). Therefore $J(Y_{Q})$ is isogenous to its image in $J(X)$ as well (as $X$ is a Galois cover of $Y_{Q}$), and hence the same assertion holds for its subvarieties, in particular the primitive Prym variety $\widetilde{P}(Y_{Q}/S)$. It thus suffices to investigate $J(Y_{Q})$ as a representation of $Q$. As $Q$ is cyclic, we deduce from Proposition \ref{Jacdecom} that precisely one element of $Irr_{\mathbb{Q}}(Q)$ is a faithful representation of $Q$, while the other elements have non-trivial kernels. Specifically, as a representation of $G$, the faithful representation in $Irr_{\mathbb{Q}}(Q)$ is the representation $W \in Irr_{\mathbb{Q}}(G)$ with $d_{W}=1$ that corresponds to $Q$, and the other representations in $Irr_{\mathbb{Q}}(Q)$ come from proper quotients of $Q$. It follows that the image of $\widetilde{P}(Y_{Q}/S) \subseteq J(Y_{Q})$ in $J(X)$ is the subvariety denoted by $B_{W}$ in Proposition \ref{Jacdecom} and by $B_{Q}$ in Proposition \ref{Jaccycquot}. The other parts of $J(Y_{Q})$ (or its image in $J(X)$) are, via the same argument, images of primitive Prym varieties of coarser quotients (either in $J(Y_{Q})$ or in $J(X)$). The first two required assertions therefore follow from Proposition \ref{Jaccycquot}.

For the third one we just have to verify the positivity of the associated dimension, which we can write as in Proposition \ref{Jaccycquot} for the cover $Y_{Q} \to S$. For the trivial quotient $G/G$ this is immediate: The dimension of $f^{*}J(S)$ is just $g_{S}$. On the other hand, when $Q$ is non-trivial and we consider the faithful representation the condition $C \not\subseteq N$ is redundant, and the formula for the dimension of $\widetilde{P}(Y_{Q}/S) \subseteq J(Y_{Q})$ becomes just $\varphi(|Q|)\big(g_{S}-1+\sum_{Id_{Y_{Q}} \neq y \in Q}\frac{r_{y}}{2}\big)$ (the sum taken on non-trivial elements $y \in Q$). But as Corollary \ref{Galgen} for this cover produces the equality $\frac{g_{Y_{Q}}-1}{|Q|}=g_{S}-1+\sum_{Id_{Y_{Q}} \neq y \in Q}\frac{r_{y}}{2}\big(1-\frac{1}{o(y)}\big)$, we can write this dimension as $\frac{\varphi(|Q|)}{|Q|}\big(g_{Y_{Q}}-1+\sum_{y}\frac{|Q|r_{y}}{2o(y)}\big)$. For $g_{Y_{Q}}\geq1$ this number is positive, unless $g_{Y_{Q}}=1$ and there is no branching (since any branching has a positive contribution to the sum over $y$), which by Corollary \ref{Galgen} is equivalent to $g_{S}$ being 1 as well. This proves the theorem.
\end{proof}

It is easy to see that in the cases not covered by the last assertion in Theorem \ref{PryminJac} the primitive Prym variety $\widetilde{P}(Y_{Q}/S)$ is trivial. Indeed, if $g_{Y_{Q}}=0$ then $J(Y_{Q})$ itself is trivial, and if $Y_{Q}$ is a non-trivial unbranched cover of $S$ and both $g_{Y_{Q}}$ and $g_{S}$ equal 1 then $J(Y_{Q})$ is the image of $J(S)$ there and $\widetilde{P}(Y_{Q}/S)$ is again trivial. This completes the analysis of non-triviality of primitive Prym varieties in cyclic covers. While we shall not investigate primitive Prym varieties for arbitrary covers, Theorem \ref{PryminJac} has the following immediate consequence regarding primitive Prym varieties for Abelian covers that are not cyclic.
\begin{cor}
Let $f:X \to S$ be an Abelian cover, and assume that its (Abelian) Galois group $G$ is not cyclic. Then $\widetilde{P}(X/S)$ is trivial. \label{Prymabnocyc}
\end{cor}

\begin{proof}
Theorem \ref{PryminJac} implies, for Abelian $G$, that $J(X)$ is covered by the images of Jacobians of its quotients that cover $S$ with cyclic Galois groups. Since $G$ is not cyclic, all of these Jacobians arise from proper intermediate Riemann surfaces. Hence $\widetilde{P}(X/S)$ is the the complement of all of $J(X)$ by definition, implying its triviality. This proves the corollary.
\end{proof}
Corollary \ref{Prymabnocyc} was so simple since the full structure of $J(X)$ was given in Theorem \ref{PryminJac} in the Abelian case. For non-Abelian groups one must examine how the intermediate Riemann surfaces (in particular those that are not Galois over $S$) interact with the subvarieties $A_{W}$ and $B_{W}$ for $W \in Irr_{\mathbb{Q}}(G)$ from Proposition \ref{Jacdecom}, specifically with those for which $d_{W}>1$. We leave this question for further investigation.

\smallskip

We conclude with a remark about expressions using explicit equations. This is motivated, in the context of \cite{[KZ]} and others, by the fact, following from Theorem \ref{PryminJac}, that the Jacobian $J(X)$ of a Galois cover $X$ of $\mathbb{P}^{1}(\mathbb{C})$ has a part that is isogenous to the product of primitive Prym varieties of $Z_{m}$ curves (and $J(X)$ decomposes as the product of such primitive Prym varieties when the cover is Abelian). In general, the cyclicity of $Q$ implies that $\mathbb{C}(Y_{Q})$ is generated over $\mathbb{C}(S)$ by an element $w_{Q}\in\mathbb{C}(X)$ such that $w_{Q}^{|Q|}=F_{Q}\in\mathbb{C}(S)$, and the equation $w_{Q}^{|Q|}=F_{Q}$ is irreducible. Take a character $\chi\in\widehat{G}$ that represents $Q$ faithfully, and then any non-zero element of $\mathbb{C}(X)_{\chi}$ would produce such an equation. For a more explicit description, assume that $G$ is Abelian, expressed as a direct product of cyclic groups as in Propositions \ref{psiZn} and \ref{charZn}. Then if $\chi$ corresponds to the sequence $E=(e_{l})_{l=1}^{q}\in\prod_{l=1}^{q}(\mathbb{Z}/m_{l}\mathbb{Z})$ as in the latter proposition, we may take $w^{E}$ to be our $w_{Q}$. The associated number $\beta$ from Section \ref{GalCov} is $o(\chi)$, or equivalently $|Q|$, and the resulting equation producing the cyclic cover $Y_{Q}$ of $S$ whose primitive Prym variety is $B_{Q}$, or $B_{W}$ for the appropriate $W \in Irr_{\mathbb{Q}}(G)$, is $(w^{E})^{|Q|}=\prod_{l=1}^{q}F_{l}^{e_{l}|Q|/m_{l}}$ (and this equation is non-degenerate). In particular, when $S=\mathbb{P}^{1}(\mathbb{C})$ the Riemann surface $Y_{Q}$ is a $z_{|Q|}$ curve, and the latter equation is a defining $z_{|Q|}$ equation.

\noindent\textsc{Finance Department School of Business 2100, Hillside University of Connecticut, Storrs, CT 06268 \\ Einstein Institute of Mathematics, the Hebrew University of Jerusalem, Edmund Safra Campus, Jerusalem 91904, Israel}

\noindent E-mail address: zemels@math.huji.ac.il, yaacov.kopeliovich@uconn.edu

\end{document}